\documentclass[reqno,10pt]{amsart} 

\usepackage[colorlinks = true, linkcolor=blue,
            urlcolor=red,
            citecolor=olive]{hyperref}
\usepackage[utf8]{inputenc}
\usepackage[T1]{fontenc}
\usepackage{amsmath}

\usepackage{mathtools}      % improves amsmath
\usepackage{mathabx}        % makes many symbols (as $\leq$) more beautiful; must be loaded after mathtools 

\usepackage{tikz}
\usepackage{tikz-3dplot}
\usepackage{tikz-cd}
\usepackage{amssymb}
\usepackage{amsfonts}
\usepackage{graphicx}

\usepackage{multicol}

\usepackage[colorinlistoftodos]{todonotes}

\hypersetup{bookmarksdepth  =  3} % Depth of sections/subs... to have bookmark links in the pdf;
\usetikzlibrary{arrows,calc,backgrounds}

\newcommand{\rotateRPY}[3]% roll, pitch, yaw
{   \pgfmathsetmacro{\rollangle}{#1}
    \pgfmathsetmacro{\pitchangle}{#2}
    \pgfmathsetmacro{\yawangle}{#3}

    \pgfmathsetmacro{\newxx}{cos(\yawangle)*cos(\pitchangle)}
    \pgfmathsetmacro{\newxy}{sin(\yawangle)*cos(\pitchangle)}
    \pgfmathsetmacro{\newxz}{-sin(\pitchangle)}
    \path (\newxx,\newxy,\newxz);
    \pgfgetlastxy{\nxx}{\nxy};

    \pgfmathsetmacro{\newyx}{cos(\yawangle)*sin(\pitchangle)*sin(\rollangle)-sin(\yawangle)*cos(\rollangle)}
    \pgfmathsetmacro{\newyy}{sin(\yawangle)*sin(\pitchangle)*sin(\rollangle)+ cos(\yawangle)*cos(\rollangle)}
    \pgfmathsetmacro{\newyz}{cos(\pitchangle)*sin(\rollangle)}
    \path (\newyx,\newyy,\newyz);
    \pgfgetlastxy{\nyx}{\nyy};

    \pgfmathsetmacro{\newzx}{cos(\yawangle)*sin(\pitchangle)*cos(\rollangle)+ sin(\yawangle)*sin(\rollangle)}
    \pgfmathsetmacro{\newzy}{sin(\yawangle)*sin(\pitchangle)*cos(\rollangle)-cos(\yawangle)*sin(\rollangle)}
    \pgfmathsetmacro{\newzz}{cos(\pitchangle)*cos(\rollangle)}
    \path (\newzx,\newzy,\newzz);
    \pgfgetlastxy{\nzx}{\nzy};
}

\tikzset{RPY/.style={x={(\nxx,\nxy)},y={(\nyx,\nyy)},z={(\nzx,\nzy)}}}

\theoremstyle{plain}
\newtheorem{thm}{Theorem}[section]
\newtheorem{lem}[thm]{Lemma}
\newtheorem{prop}[thm]{Proposition}
\newtheorem{cor}[thm]{Corollary}

\theoremstyle{definition}
\newtheorem{defn}[thm]{Definition}
\newtheorem{defnProp}[thm]{Definition/Proposition}

\theoremstyle{remark}
\newtheorem{rem}[thm]{Remark}

\newtheorem{que}[thm]{Question}

\newcommand{\R}{\mathbb{R}}
\newcommand{\C}{\mathbb{C}}
\newcommand{\N}{\mathbb{N}}
\newcommand{\Sph}{\mathbb{S}}

\DeclareMathOperator{\Ker}{Ker}

\DeclareMathOperator{\SL}{SL}

\DeclareMathOperator{\Hom}{Hom}

\DeclareMathOperator{\SO}{SO}
\DeclareMathOperator{\Sp}{Sp}
\DeclareMathOperator{\PSp}{PSp}
\DeclareMathOperator{\Gr}{Gr}

\makeatletter
\DeclareRobustCommand*{\bfseries}{%
  \not@math@alphabet\bfseries\mathbf
  \fontseries\bfdefault\selectfont
  \boldmath
}
\makeatother

\newcommand{\ie}{i.e. }

\title{Geometric structures for maximal representations and pencils.}

\renewcommand{\Gr}{\text{Gr}}
\newcommand{\X}{\mathbb{X}}

\newcommand{\Z}{\mathbb{Z}}

\DeclareMathOperator{\Gra}{Gr}
\DeclareMathOperator{\Ein}{Ein}
\DeclareMathOperator{\Tr}{Tr}
\DeclareMathOperator{\pr}{pr}
\DeclareMathOperator{\mix}{mix}
\DeclareMathOperator{\Imu}{Im}
\DeclareMathOperator{\Hull}{Hull}

\newcommand{\hgline}[2]{
\pgfmathsetmacro{\thetaone}{#1}
\pgfmathsetmacro{\thetatwo}{#2}
\pgfmathsetmacro{\theta}{(\thetaone+\thetatwo)/2}
\pgfmathsetmacro{\phi}{abs(\thetaone-\thetatwo)/2}
\pgfmathsetmacro{\close}{less(abs(\phi-90),0.0001)}
\ifdim \close pt = 1pt
    \draw[blue] (\theta+180:1) -- (\theta:1);
\else
    \pgfmathsetmacro{\R}{tan(\phi)}
    \pgfmathsetmacro{\distance}{sqrt(1+\R^2)}
    \draw[blue] (\theta:\distance) circle (\R);
\fi
}

\usepackage{tkz-euclide}

\copyrightinfo{2021}{Colin Davalo}

\author{Colin Davalo}
\address{Università degli studi di Torino, Via Verdi 8, 10124 Torino, Italy}
\email{colinmarcthiery.davalo@unito.it}

\date{\today}
\begin{document}

\begin{abstract}
We study fibrations of the projective model for the symmetric space associated with $\SL(2n,\R)$ by codimension $2$ projective subspaces, or pencils of quadrics. In particular we show that if such a smooth fibration is equivariant with respect to a representation of a closed surface group, the representation is quasi-isometrically embedded, and even Anosov if the pencils in the image contain only non-degenerate quadrics. We use this to characterize maximal representations among representations of a closed surface group into $\Sp(2n,\R)$ by the existence of an equivariant continuous fibration of the associated symmetric space, satisfying an additional technical property. These fibrations extend to fibrations of the projective structures associated to maximal representations by bases of pencils of quadrics.

%We show that maximal representations in $\Sp(2n,\R)$ of the fundamental group of an oriented closed surface $S_g$ can be characterized by the existence of a continuous equivariant embedding of $\widetilde{S_g}$ into a particular subset of the space of pencils of quadrics on $\R^{2n}$ that admit what we call a fitting flow.
%Such an embedding defines a fibration of the locally symmetric space as well as a fibration of a manifold with a contact projective structure whose holonomy is the maximal representation.
%This characterization generalizes in some way the characterization of photon structures obtained by Collier-Tholozan-Toulisse corresponding to maximal representation in $\PSp(4, \R)$. 
%We show however that some quasi-Fuchsian representations in $\PSL(2,\C)$ do not admit such an equivariant embedding of Hermitian pencils of quadrics on $\C^2$ .

\end{abstract}
\maketitle

\tableofcontents

\section{Introduction}
In this paper we study representations of the fundamental group $\Gamma_g$ of a closed oriented surface of genus $g\geq 2$ into semi-simple Lie groups.
Our goal is to characterize a special class of representations into $\Sp(2n,\R)$, \emph{maximal representations}, by the existence of an equivariant fibration of the projective model for the symmetric space associated to $\SL(2n,\R)$, restricting to a fibration of the symmetric space associated to $\Sp(2n,\R)$. To illustrate these ideas we begin by discussing quasi-Fuchsian representations in $\SL(2,\C)$.

\subsection{Fibrations of  \texorpdfstring{$\mathbb{H}^3$}{H3} by geodesics.}

%We characterize maximal representations by describing the associated projective structure using pencils of quadrics. This characterization builds on the notion of fitting maps of pencils, which are maps that define locally a fibration of projective space, as well as the symmetric space of $\SL(2n,\R)$. 
%
%\medskip

Let $\rho:\Gamma_g\to\SL(2,\C)$ be the composition of a Fuchsian, \ie discrete and faithful, representation and the inclusion $\SL(2,\R)\subset \SL(2,\C)$.
The locally symmetric space $\mathbb{H}^3/\rho(\Gamma_g)$ is a fiber bundle over $S_g$ whose fibers are geodesics.
One can construct such a fibration by taking the geodesics orthogonal to the totally geodesic copy of $\mathbb{H}^2$  in $\mathbb{H}^3$ preserved by the action of $\SL(2,\R)$. 
This fibration extends to a fibration of an open domain in $\mathbb{H}^3 \cup \partial \mathbb{H}^3$.

\medskip

Such a fibration is described by $\rho$-equivariant map $u:\widetilde{S_g}\to \mathcal{G}$ where $\mathcal{G}$ is the space of geodesics in $\mathbb{H}^3$. We say that an immersion $u:\widetilde{S_g}\to \mathcal{G}$ is \emph{fitting} if the corresponding geodesics locally define a smooth fibration of $\mathbb{H}^3\cup \partial \mathbb{H}^3$.

\medskip

Let $\rho:\Gamma_g\to\SL(2,\C)$ be \emph{nearly Fuchsian}, \ie suppose that it admits an equivariant immersion $h:\widetilde{S_g}\to \mathbb{H}^3$ with principal curvature in $(-1,1)$. Epstein showed that the locally symmetric space $\mathbb{H}^3/\rho(\Gamma_g)$ admits a fibration described by the fitting immersion $\mathcal{G}h$ that associates to $x\in \widetilde{S_g}$ the geodesic orthogonal to $h(\widetilde{S_g})$ at $h(x)$ \cite{Epstein}. The map $\mathcal{G}h$ is sometimes referred to as the \emph{Gauss map} of $h$.

\medskip

Nearly Fuchsian representations are quasi Fuchsian, \ie are quasi-isometric embeddings \cite{Epstein}. We generalize this fact to any representation that admits an equivariant fitting immersion.

\begin{thm}
\label{thm:INTROfittingSLC}
Let $\rho:\Gamma_g\to \SL(2,\C)$ be a representation that admits an equivariant fitting immersion $u:\widetilde{S_g}\to \mathcal{G}$. The representation $\rho$ is quasi-Fuchsian.

\end{thm}

This theorem is a consequence of Theorem \ref{thm:WellFittedImplesAnosov}. There may a priori exist representations with equivariant fitting maps that are not nearly Fuchsian, see Remark \ref{rem:Nearly Fuchsian}. 
 However Theorem \ref{thm:INTROfittingSLC} does not provide a characterization of quasi Fuchsian representations in general because of the following result.

\begin{thm}[{Theorem \ref{thm:Quasi-Fuchsian with no fitting}}]
\label{INTRO:quasiFuchsian with no fibrations}
For a genus $g$ large enough, there exist quasi Fuchsian representations $\rho:\Gamma_g\to S_g$ that admit no equivariant fitting immersions $u:\widetilde{S_g}\to \mathcal{G}$.
\end{thm}

As a corollary we prove that there exist quasi Fuchsian representations that are not nearly Fuchsian.

\subsection{Fibrations of a convex set.} 

Let $V=\R^{2n}$. We denote by $S^2V$ be the space of symmetric tensors in $V\otimes V$, or equivalently the space of symmetric bilinear forms on $V^*$. Let $S^2V^{>0}$ be the convex cone of positive tensors, \ie tensors that define positive symmetric bilinear forms on $V^*$.

\medskip

The projective convex domain $\mathbb{P}(S^2V^{>0})$ is the projective model for the symmetric space $\X_{\SL}=\SL(2n,\R)/\SO(2n,\R)$. We study fibrations of this convex domain by projective codimension $2$ subspaces.

\medskip

A codimension $2$ projective subspace of $\mathbb{P}(S^2V)$ corresponds to a dimension $2$ subspace of the dual vector space $S^2V^*$, that can be interpreted as the space of \emph{quadrics} on $V$. We will write $\mathcal{Q}=S^2V^*$. A plane in $\mathcal{Q}$ is called a \emph{pencil of quadrics} and we will denote by $\Gra_2(\mathcal{Q})$ the space of such planes. Let $\Gra_2^\text{mix}(\mathcal{Q})$ be the set of \emph{mixed} pencils, \ie pencils $P$ whose corresponding codimension $2$ projective subspace intersects non-trivially the convex domain $\mathbb{P}(S^2V^{>0})$, or equivalently the pencils that do not contain any non-zero semi-positive quadric (in other words the signature of all quadrics of the pencil is mixed).

\medskip

In this setting, in a way that is analog as for immersions in the space $\mathcal{G}$ of geodesics in $\mathbb{H}^3$, we say that an immersion $u:S\to \Gra^\text{mix}_2(\mathcal{Q})$ from a surface $S$ is \emph{fitting} is the corresponding codimension two subsets define locally a smooth fibration of the closure $\mathbb{P}(S^2V^{\geq 0})$ of the convex domain $\mathbb{P}(S^2V^{> 0})$. As in $\mathbb {H}^3$ one can construct examples of such maps by taking the Gauss map of some totally geodesic surfaces in $\SL(2n,\R)/\SO(2n,\R)$, see Proposition \ref{prop:GaussMapSymmetric space for totally geodesic immersions}.

\medskip

Let $\rho:\Gamma\to\SL(2n,\R)$ be a representation and $u:\widetilde{S_g}\to \Gr_2^{\text{mix}}(\mathcal{Q})$ a $\rho$-equivariant fitting immersion, then $u$ is an embedding and defines a fibration of an open domain in $\mathbb{P}(S^2V^{\geq 0})$ that contains $\mathbb{P}(S^2V^{> 0})$, see Proposition \ref{prop:Completeness and fibration of DOD}.

 Moreover let $\Gr_2^{(n,n)}(\mathcal{Q})\subset \Gr_2(\mathcal{Q}) $  be the set of pencils of quadrics $P$ such that every non zero $q\in P$ has signature $(n,n)$, or equivalently the pencils containing no degenerate non-zero quadric.

\begin{thm}
Let $\rho:\Gamma_g\to \SL(2n,\R)$ be a representation that admits an equivariant fitting immersion $u:\widetilde{N}\to \Gr_2^{(n,n)}(\mathcal{Q})$. The representation $\rho$ is $\lbrace n\rbrace$-Anosov.
\end{thm}

This theorem can be generalized by replacing $S_g$ by a closed manifold of some dimension $d$ and considering equivariant maps into $\Gr_d(\mathcal{Q})$, see Theorem \ref{thm:WellFittedImplesAnosov}.
\medskip

In order to show this result we introduce the notion of a fitting flow.
Let $\mathcal{E}$ be the tautological rank $2$ vector bundle over $\Gr_2(\mathcal{Q})$, \ie:
$$\mathcal{E}=\lbrace (P,q)\mid P\in \Gr_2(\mathcal{Q}),\,q\in P \rbrace.$$ 
Let $\pi:\mathcal{E}\to \mathcal{Q}$  be the tautological projection, \ie the projection onto the second factor. The pullback $u^*\Sph \mathcal{E}$ of the circle bundle $\Sph \mathcal{E}$ for a map $u:S\to\Gr^{\text{mix}}_2(\mathcal{Q})$ is the space of pairs $(x,H)$ where $x\in S$ and $H$ is a co-oriented projective hyperplane in $\mathbb{P}(S^2V)$ containing the codimension $2$ projective subspace associated to $u(x)$. Note that a co-oriented projective hyperplane intersecting $\mathbb{P}(S^2V^{>0})$ defines a half space in the convex set $\mathbb{P}(S^2V^{\geq 0})$

A \emph{fitting flow} for the map $u:S\to\Gr^{\text{mix}}_d(\mathcal{Q})$ is a flow $\Phi$ on this circle bundle $ u^*\Sph\mathcal{E}$ over $S$ such that along flow lines the corresponding projective half spaces in $\mathbb{P}(S^2V^{\geq 0})$ are nested, see Defintion \ref{defn:Fitting flow}.

We show that a fitting immersion always admits a smooth fitting flow, see Proposition \ref{prop:ExistenceFittingFlow}, that flow lines must exist between the fibers over any pair of points in $S$ and must be quasi-geodesic, see Propoistion \ref{prop:Flow lines of a fitting flow are quasi geodesic and exist between points}. Finally we use the nestedness of the half spaces when following a flow line to show uniform contraction occuring in projective space that is sufficient to imply that the representation is Anosov, using an argument similar to the criterion on nested multicones from Bochi-Potrie-Sambarino \cite{BPS}.

\subsection{Maximal representations in  \texorpdfstring{$\Sp(2n,\R)$}{Sp(2n,R)}.}

Let us now focus on representations into the group $\Sp(2n,\R)$ of linear transformations of $\R^{2n}$ preserving a symplectic form. 

\medskip

Maximal representations of a closed surface group in $\Sp(2n,\R)$ are representation whose Toledo invariant takes its maximal possible value \cite{BIW10}. Equivalently maximal representations can be characterized as \emph{positive representations} \cite{BIW03}, or again equivalently as representations that admit an equivariant continuous and transverse map $\xi_\rho^n:\partial \Gamma\to \mathcal{L}_n$ having a specific homotopy type, with $\mathcal{L}_n$ the space of Lagrangians in $(\R^{2n},\omega)$, see Theorem \ref{thm:Caract Maximal reprs via homotopy}.

\medskip

 Maximal representations form connected components of the space of representations, and such representations are all discrete and faithful \cite{BIW03}. The space of maximal representations is therefore called a \emph{higher rank Teichmüller space}. The classical Teichmüller space as well as some higher rank Teichmüller spaces, can be interpreted as spaces of geometric structures. For instance Hitchin representations in $\SL(3,\R)$ can be interpreted as spaces of convex projectives structures on the considered surface. Maximal representations in $\Sp(4,\R)$ or $\SO_o(2,n)$ can be interpreted as spaces of projective or photon structures on a bundle over the surface that admit a special fibration \cite{Collier_2019}. In these two examples these structures modeled on flag manifolds can be interpreted as a part of a natural compactification of the locally symmetric space associated with the representation. In the present paper we focus on the locally symmetric structures and ask the following:

\begin{que} Can maximal representations can be characterized by the existence of some fibration of the associated locally symmetric space?
\end{que}

We provide some affirmative answer to this question. In order to study the symmetric space for $\Sp(2n,\R)$ we embed it into the projective model for the symmetric space of $\SL(2n,\R)$.

\subsection{Characterization of maximal representations.} 

The symmetric space $\X_{\Sp}$ associated to $\Sp(2n,\R)$ can be identified with a totally geodesic submanifold of $\X_{\SL}$.
Let $\Gra^\omega_2(\mathcal{Q})\subset \Gra^{(n,n)}_2(\mathcal{Q})$ be the set of pencils $P$ such that every non-zero $q\in P$ is positive on some Lagrangian and negative on some other Lagrangian in $\R^{2n}$.
We show that the projectivization of the codimension $2$ projective subspace of $\mathbb{P}(S^2V)$ corresponding to any $P\in\Gra^\omega_2(\mathcal{Q})$ intersects transversely the totally geodesic symmetric subspace $\X_{\Sp}\subset \mathbb{P}(S^2V^{>0})$, see Lemma \ref{lem:Transverse intersection with symmetric space}.

The set $\Gra^\omega_2(\mathcal{Q})$ is open in $\Gra_2(\mathcal{Q})$, but we show that it is disconnected. We select a special union of connected components that we denote by $\Gra_2^{\max} (\mathcal{Q})$ and we show the following:

\begin{thm}[{Theorem \ref{thm: fitting implies maximal}}]
Let $\rho:\Gamma_g\to \Sp(2n,\R)$ be a representation. If it admits a $\rho$-equivariant fitting immersion $u:\widetilde{S_g}\to \Gra_2^{\max}(\mathcal{Q})$ it is maximal for some orientation of $S_g$.
\end{thm}

If a representation $\rho:\Gamma_g\to \Sp(2n,\R)$  admits an equivariant fitting immersion $u:\widetilde{S_g}\to\Gra^{(n,n)}_2(\mathcal{Q})$  the image of $u$ lies necessarily in $\Gra^\omega_2(\mathcal{Q})$. We then use that the homotopy type of the boundary map of $\rho$ is determined by the connected component of $\Gra^\omega_2(\mathcal{Q})$ in which the image of $u$ lies to define $\Gra_2^{\max} (\mathcal{Q})$ and to prove this theorem, see Section \ref{subsec:DefnMaxPencil}. 

\medskip

For $n=2$ using results of \cite{Collier_2019} we show  a converse to this statement, which therefore provides a characterization of maximal representations, see Corollary \ref{cor:MaxReprSP4}. We compare this characterization with the one from Collier-Tholozan-Toulisse in Remark \ref{rem:comparaisonCCT}. For $n\geq 3$ we prove a weaker converse to this theorem. We construct equivariant maps $u:\widetilde{S_g}\to \Gra_2^{\max}(\mathcal{Q})$ for maximal representations that locally define a fibration of $\mathbb{P}(S^2V^{\geq 0})$ but are only continuous and not smooth. 

\begin{thm}[{Theorem~\ref{thm: fitting implies maximal}}]
\label{thm:INTRO charachterization}
\label{thm: fitting implies maximal Intro}
A representation $\rho:\Gamma_g\to \Sp(2n,\R)$ is maximal if and only if it admits a $\rho$-equivariant continuous map of pencils that admits an equivariant fitting flow:
$$u:\widetilde{S_g}\to \Gra_2^{\max}(\mathcal{Q}).$$
\end{thm}

%Note that the quasi-Fuchsian representations from Theorem \ref{INTRO:quasiFuchsian with no fibrations} do not admit continuous maps with an equivariant fitting flow.

A continuous equivariant map $u:\widetilde{S_g}\to \Gra_2^{{\max}}(\mathcal{Q})$ for a representation $\rho:\Gamma_g\to \Sp(2n,\R)$ with a fitting flow defines a $\rho$-equivariant continuous fibration of $\mathbb{P}(S^2V^{\geq0})$, but also of $\X_{\Sp}\subset \mathbb{P}(S^2V^{>0})$ as $\X_{\Sp}$ intersects the fibers transversely.
Thus $u$ defines a continuous fibration of the locally symmetric space $\X_{\Sp}/\rho(\Gamma_g)$.
% Moreover the intersection of the projective zeroes of the quadrics in a given pencil of in $\Gra_2^{\text{max}}(\mathcal{Q})$ defines a codimension $2$ subset of $\mathbb{P}(V)$. 

The space of rank one elements of $\mathbb{P}(S^2V^{\geq 0})$ can naturally be identified with $\mathbb{P}(V)$. 
Such a map $u$ defines a continuous $\rho$-equivariant fibration of a domain in projective space, that is equal to the domain of discontinuity in projective space constructed by Guichard-Wienhard \cite{GWDoD}. The quotient of this domain inherits a $(\Sp(2n,\R),\mathbb{RP}^{2n-1})$-structure, a contact projective structure.
Theorem \ref{thm:INTRO charachterization} implies the following characterization of the contact projective structures corresponding to maximal representations.

\begin{cor}
A contact projective structure on a fiber bundle $M$ with fiber $F$ over $S_g$ corresponds to a maximal representations by the construction of Guichard-Wienhard if and only if, up to homeomorphisms of $M$ that stabilize $\pi_1(F)$ and act trivially on $\pi_1(M)/\pi_1(F)\simeq \Gamma_g$, the fibers are mapped via the developing map onto the bases of maximal pencils of quadrics parametrized by a continuous map that admits an equivariant fitting flow.
\end{cor}

\medskip

In order to construct continuous maps that admits a fitting flow for maximal representations, we show  how to associate to a pair of transverse Lagrangians a quadric on $\R^{2n}$. %We discuss an other use for these quadrics in Section \ref{subsec:polygons of hyperboloids} to decompose the projective structures associated to maximal representations into smaller pieces which are polygons of quadric hypersurfaces. These quadrics were used by Burelle-Treib to study Schotty subgroups \cite{Burelle}.

\subsection{Organization of the paper}
This paper begins with a recall of some facts about maximal and Anosov representations in Section \ref{sec:Maximal and Anosov representations}.

The main definitions are introduced in Section \ref{sec:An invariant convex domain and its fibrations} where fibrations of a projective convex set by projective subspaces are discussed, more precisely fibrations of the projective model for the symmetric space of $\SL(2n,\R)$. 

In Section \ref{sec:Fitting flows} we introduce the notion of fitting flows.

In Section \ref{sec:The Anosov property and fibrations.} we discuss how the existence of an equivariant continuous map with a fitting flow implies the Anosov property, Theorem \ref{thm:WellFittedImplesAnosov}, and describes a fibration of a domain of discontinuity in projective space. 

In Section \ref{sec:Characterization of the structures.} we focus on representations into $\Sp(2n,\R)$ and prove Theorem \ref{thm: fitting implies maximal}, which is the characterization of maximal representations by the existence of a locally fitting map of maximal pencils of quadrics that admits a fitting flow.

In Section \ref{sec:Subsection Symmetric space} we prove two propositions relative to the symmetric space of $\SL(2n,\R)$ and $\Sp(2n,\R)$.

In Appendix \ref{sec:H3} we discuss fibrations of the hyperbolic $3$-space with geodesic fibers. 

In Appendix \ref{sec:SP4} we show how spacelike surfaces in $\mathbb{H}^{2,2}$ with a bound on their principal curvatures define a fitting immersion of pencils. 

%Finally in Appendix \ref{subsec:polygons of hyperboloids} we briefly discuss a decomposition of the projective structure associated to a maximal representation into hexagons of quadric hypersurfaces.

\subsection*{Acknowledgments}
I would like to thank Andrea Seppi for the helpful discussions during my stay in Grenoble that helped me a lot on this project. I would like to thank also Beatrice Pozzetti for discussing this topic and giving feedback.
 The author was funded through the DFG Emmy Noether project 427903332 of B. Pozzetti and the RTG 2229 “Asymptotic
Invariants and Limits of Groups and Spaces”, and partially by the European Union via the ERC 101124349 "GENERATE". 

Views and opinions expressed are however those of the authors only and do not necessarily reflect those of the European Union or the European Research Council Executive Agency. Neither the European Union nor the granting authority can be held responsible for them.

\section{Maximal and Anosov representations.}
\label{sec:Maximal and Anosov representations}
\subsection{Maximal representations.}
\label{sec:Subsec on maximal representations}
Let us fix a symplectic from $\omega$ on $\R^{2n}$, i.e. a non-degenerated bilinear antisymmetric pairing. A \emph{symplectic basis} of $\R^{2n}$ is a basis $(x_1,\cdots, x_n,y_1,\cdots, y_n)$ in which :
$$\omega= \sum_{i=1}^n x_i^*\wedge y_i^*.$$

We define $\Gra_n(\R^{2n})$ as the space of $n$-dimensional subspaces of $\R^{2n}$ A \emph{Lagrangian} in $(\R^{2n},\omega)$ is an element $\ell\in \Gra_n(\R^{2n})$ such that $\omega$ restricted to $\ell$ is equal to zero. We denote by $\mathcal{L}_n$ the space of Lagrangians in $(\R^{2n},\omega)$. We say that two Lagrangians are transverse if their intersection is trivial.

\medskip

Let $\Sp(2n,\R)$ be the subgroup of elements in $\SL(2n,\R)$ that preserves $\omega$. This groups acts transitively on $\mathcal{L}_n$, as well as on the space of pairs of transverse Lagrangians. Given a triple $(\ell_1,\ell_2,\ell_3)$ of transverse Lagrangians, one can find a symplectic basis such that for some $(\epsilon_i)\in \lbrace 1,-1\rbrace$:
$$\ell_1=\langle x_1,x_2,\cdots, x_n\rangle, $$
$$\ell_2=\langle x_1+\epsilon_1 y_1,x_2+ \epsilon_2 y_2,\cdots , x_n +\epsilon_ny_n\rangle,$$
$$\ell_3=\langle y_1,y_2,\cdots, y_n\rangle. $$

The sum of the $(\epsilon_i)$ is an invariant of the triple of flags that is called the \emph{Maslov index} $M(\ell_1,\ell_2,\ell_3)$. These facts can be found in \cite[{Section 1.5.7}]{lion2013weil}.
The group $\Sp(2n,\R)$ acts transitively on the space of triples of transverse Lagrangians with a given Maslov index. We say that $(\ell_1,\ell_2,\ell_3)$ is \emph{maximal} if the Maslov index of the triple in equal to $n$.

\medskip

The Lie group $\Sp(2n,\R)$ is of Hermitian type and tube type. Hence it admits a special class in its continuous cohomology group $[\tau]\in H^2_c(\Sp(2n,\R),\mathbb{Z})$. Let $S_g$ be a closed oriented surface of genus $g\geq 2$. The fundamental class of $S_g$ defines a cohomology class $[S_g]\in H^2(\pi_1(S_g),\mathbb{Z})\simeq \mathbb{Z}$. Given a representation $\rho:\pi_1(S_g)\to \Sp(2n,\R)$ one can consider the pullback of this class $\rho^*[\tau]=T(\rho)[S_g]$. the integer $T(\rho)$ is called the \emph{Toledo number} of $\rho$.

\medskip

The Toledo number can take only finitely many values as the space of representations can only have finitely many connected components. More precisely:

\begin{lem}[{\cite{BIW10}}]
Let $\rho:\pi_1(S_g)\to \Sp(2n,\R)$, the Toledo number satisfies :
$$(-2g+2)n \leq T(\rho)\leq (2g-2)n.$$
\end{lem}

Such a representation is called \emph{maximal} if its Toledo number is equal to $(2g-2)n$.
\subsection{Anosov representations.}
\label{sec:Subsection on anosov represetnations}

Let $\Gamma$ be a finitely generated group. Anosov representations are representations with some exponential gaps between singular values.

\medskip

Fix a word metric $|\cdot|$ on $\Gamma$ and a scalar product on $\R^{2n}$ allowing us to define the singular values $\sigma_1(g)\geq \sigma_2(g)\geq \cdots \geq \sigma_{2n}(g)$ of $g\in \SL(2n,\R)$ as the eigenvalues of $\sqrt{g^tg}$. The following definition is independent of these choices.

\begin{defn}[{\cite{BPS}}]
\label{defn:Anosov2}
We say that a representation $\rho:\Gamma\to \SL(2n,\R)$ is $\lbrace n\rbrace$-Anosov if there exist $A,B>0$ such that for  all $\gamma\in \Gamma$:
$$\frac{\sigma_n(\rho(\gamma))}{\sigma_{n+1}(\rho(\gamma))}\geq e^{A|\gamma|+B}.$$
\end{defn}

If a group admits an Anosov representation, it must be Gromov hyperbolic \cite{BPS}. We denote by $\partial \Gamma$ its Gromov boundary. Anosov representations come with boundary maps.

\begin{thm}[{\cite[{Theorem 1.1}]{Gu_ritaud_2017}}]
Let $\rho:\Gamma\to \SL(2n,\R)$ be $\lbrace n\rbrace$-Anosov. There exist a unique $\rho$-equivariant continuous map $\xi^n_\rho:\partial \Gamma\to \Gra_n(\R^{2n})$ such that:
\begin{itemize}
\item for all distinct $x,y\in \partial \Gamma$, $\xi^n_\rho(x)\oplus \xi^n_\rho(y)=\R^{2n}$ (\emph{transverse}),
\item for all $\gamma\in \Gamma$ that admit an attracting fixed point $\gamma^+\in \partial \Gamma$, $\xi^n_\rho(\gamma^+)$ is the attracting fixed of the action of $\rho(\gamma)$ on $\Gra_{n}(\R^{2n})$ (\emph{dynamic preserving}),
\end{itemize}
If moreover $\rho(\Gamma)\subset \Sp(2n,\R)$, then $\xi^n_\rho(x)$ is a Lagrangian for all $x\in \partial \Gamma$.
\end{thm}

The fact that $\xi^n_\rho(x)$ is a Lagrangian is a consequence of the fact that an attracting fixed $n$-dimensional subspace for an element $g\in \Sp(2n,\R)$ is necessarily Lagrangian, and every $x\in \partial \Gamma$  is a limit of attracting fixed points $\gamma^+$ of elements $\gamma\in \Gamma$.

\medskip

Maximal representations have been characterized in \cite{BurgerIozziWienhardLabourie}, \cite{BIW03}:

\begin{thm}
\label{thm:BIW charac}
A representation $\rho:\pi_1(S_g)\to \Sp(2n,\R)$ is maximal if and only if it is $\lbrace n\rbrace$-Anosov and for one and hence any positively oriented triple $(x,y,z)\in \partial \pi_1(S)$ the triple $(\xi^n_\rho(x),\xi^n_\rho(y),\xi^n_\rho(z))$ is a maximal triple of Lagragians.
\end{thm}

One can also characterize maximal representations among $\lbrace n\rbrace$-Anosov representations by looking at the homotopy type of their boundary map. The fundamental group of the space of Lagrangians $\mathcal{L}_n$ is isomorphic to $\Z$ \cite{ThesisTopologyFlag} where a generator is :
$$\tau:\theta\in \mathbb{S}^1\mapsto \langle \cos\left(\frac{\theta}{2}\right)x_1+\sin\left(\frac{\theta}{2}\right)y_1, x_2,\cdots, x_n\rangle\in \mathcal{L}_n.$$

\begin{thm}
\label{thm:Caract Maximal reprs via homotopy}
A representation $\rho:\pi_1(S_g)\to \Sp(2n,\R)$ is maximal if and only if it is $\lbrace n\rbrace$-Anosov and the free homotopy type of the curve $\xi^n_\rho$ is equal to $n[\tau]$.
\end{thm}

\begin{rem}
In Theorem \ref{thm:BIW charac} and hence also in Theorem  \ref{thm:Caract Maximal reprs via homotopy}, one can relax the assumption that $\rho$ is $\lbrace n\rbrace$-Anosov. In \cite[{Theorem 8}]{BIW03} is is shown that it is sufficient to assume that $\rho$ admits an equivariant, continuous and transverse map $\xi^n_\rho: \partial \Gamma \to \mathcal{L}_n$ satisfying the additional property of sending positive triples to maximal triples.
\end{rem}

\begin{proof}
Let $\rho:\pi_1(S_g)\to \Sp(2n,\R)$ be $\lbrace n\rbrace$-Anosov. Let $(x,y,z)$ be a positively oriented triple in $\partial \pi_1(S_g)$. Up to changing the symplectic basis, we can assume that for some $(\epsilon_i)\in \lbrace -1,1\rbrace$:
$$\xi^n_\rho(x)=\langle x_1,x_2,\cdots, x_n\rangle, $$
$$\xi^n_\rho(y)=\langle x_1+\epsilon_1 y_1,x_2+ \epsilon_2 y_2,\cdots, x_n +\epsilon_ny_n\rangle,$$
$$\xi^n_\rho(z)=\langle y_1,y_2,\cdots, y_n\rangle. $$

Here the Maslov index of the triple $(\xi^n_\rho(x),\xi^n_\rho(y),\xi^n_\rho(z))$ is equal to the sum of the $(\epsilon_i)$.

\medskip

Consider the following curve:
$$\tau_0:\theta\mapsto \langle \cos\left(\frac{\theta}{2}\right)x_1+\epsilon_1\sin\left(\frac{\theta}{2}\right)y_1,\cdots, \cos\left(\frac{\theta}{2}\right)x_n+ \epsilon_n\sin\left(\frac{\theta}{2}\right)y_n\rangle.$$

This loop is homotopic to the concatenation of the loops $\tau_i$ for $1\leq i\leq n$:
$$\tau_i:\theta\mapsto \langle x_1, x_2,\cdots , \cos\left(\frac{\theta}{2}\right)x_i+ \epsilon_i\sin\left(\frac{\theta}{2}\right)y_i, \cdots , x_n\rangle.$$
 
These loops are homotopic to $\tau$ or its inverse depending on the sign of $\epsilon_i$. The homotopy type of $\tau_0$ is hence equal to $(\epsilon_1+\epsilon_2+\cdots+\epsilon_n)[\tau]$. Moreover the set of Lagrangians tranverse to a given Lagrangian is contractible, so one can homotope $\tau_0$ on the intervals $[0,\frac{\pi}{2}]$, $[\frac{\pi}{2}, \pi]$ and $[\pi,2\pi]$ to coincide with $\xi^n_\rho$. Hence the free homotopy class of $\xi^n_\rho$ is equal to the one of $\tau_0$, which is equal to $M(\xi^n_\rho(x),\xi^n_\rho(y),\xi^n_\rho(z))[\tau]$. We therefore deduce that $\rho$ is maximal if and only if it is $\lbrace n\rbrace$-Anosov and $[\xi^n_\rho]=n[\tau]$.
\end{proof}

\section{An invariant convex domain and its fibrations}
\label{sec:An invariant convex domain and its fibrations}
In this section we define globally fitting maps and fitting immersions, which are maps that parametrize fibrations of the projective model for the symmetric space of $\SL(2n,\R)$ by projective subspaces of codimension $d$.

\subsection{Pencils of quadrics}
\label{sec:Subsection on geometry of Grassmanians}

Let $V$ be a finite even-dimensional vector space with a fixed volume form. Let $S^2V$ be the space of symmetric bilinear tensors on $V$, which we interpret as maps $V^*\to V$. The dual space $\mathcal{Q}=S^2V^*$ is the space of symmetric bilinear forms on $V$, or the space of \emph{quadrics} on $V$, that we interpret as maps $V\to V^*$.

We denote by $S^2V^{\geq 0}$ and $S^2V^{>0}$ respectively the space of semi-positive and positive symmetric tensors, \ie elements $s\in S^2V$ that define respectively a semi-positive and positive symmetric bilinear form on $V^*$. The Lie group $\SL(V)$ acts on $S^2V$, and preserves the properly convex set $\mathbb{P}(S^2V^{\geq 0})$. The convex domain $\mathbb{P}(S^2V^{>0})$ is a projective model for the symmetric space associated to $\SL(V)$.

\medskip

The Grassmanian of $d$-dimensional linear subspaces of $\mathcal{Q}$ will be denoted by $\Gra_d(\mathcal{Q})$. An element of $\Gra_2(\mathcal{Q})$ is usually called a \emph{pencil of quadrics} on $V$. We will here also call elements of $\Gra_d(\mathcal{Q})$ pencils of quadrics.

\medskip

To an element $P\in \Gra_d(\mathcal{Q})$ one can associate its  annihilator, or dual, codimension $d$ subspace $P^\circ\subset S^2V$. This dual space can be described as the space of symmetric tensors $s\in S^2V$ on which one has $q(s)=\Tr(q\circ s)=0$ for all $q\in P$. Note that the projectivization $\mathbb{P}(P^\circ)$ also has codimension $d$ in $\mathbb{P}(S^2V)$.

\medskip

The subspace $\mathbb{P}(P^\circ)$ does not necessarily intersect the convex $\mathbb{P}(S^2V^{>0})$.

\begin{defn} We say that a pencil $P\in \Gra_d(\mathcal{Q})$ is \emph{mixed} if $P^\circ$ contains a positive element, \ie if $P^\circ\cap S^2V^{>0}\neq \lbrace 0\rbrace$. We call the set of mixed pencils $\Gra_d^{{\mix}}(\mathcal{Q})$.
\end{defn}

\begin{prop}
\label{prop:MixingCharacterization}
A pencil $P\in \Gra_d(\mathcal{Q})$ is mixed if and only if there is no semi-positive quadric $0\neq q\in P$.
\end{prop}

Morally this proposition is a consequence of the fact that the dual of the cone of positive elements $S^2V^{>0}$ is the cone of semi-positive bilinear forms in $\mathcal{Q}$, in the sense that $S^2V^{>0}$ is the set of tensors $s$ such that $q(s)>0$ for all semi-positive $q\in \mathcal{Q}$.

\begin{proof}
Fix $P\in \Gra_d(\mathcal{Q})$ and suppose that there is a positive element $s\in E^\circ$. We write $s$ as a finite sum of positive rank one elements of the form $v_i\otimes v_i$ with $v_i\in V$ non-zero for $i\in I$ in a finite set. Indeed $s$ defines a positive bilinear form on $V^*$ and hence is a sum of rank one symmetric positive bilinear forms. Let $q\in P$ be a semi-positive element. One has $q(s)=0$ since $s\in P^\circ$ but $q(v_i\otimes v_i)=q(v_i,v_i)\geq 0$ since $q$ is positive and hence for all $i\in I$, $q(v_i,v_i)=0$. 

Note that for all $v\in V$ non-zero we can chose a decomposition of $s$ such that $v$ is one of the $v_i$ , since $s-\epsilon v\otimes v$ is still positive for $\epsilon$ small enough. 

We proved that $q(v,v)=0$ for all $v\in V$ so $q$ is necessarily equal to $0$. Hence $P$ does not contain any non-zero semi-positive element.

\medskip

Now suppose that $P^\circ$ is disjoint from the cone $S^2V^{>0}$, so there must be a hyperplane $H$ in $S^2V$ containing $P^\circ$ and disjoint from $S^2V^{>0}$. This hyperplane corresponds to $\langle q\rangle^\circ$ for some non-zero $q\in P$. This element has the property that $q(s)\neq 0$ for all $s\in S^2V^{>0}$. Up to exchanging $q$ by $-q$ one can assume that $q(s)>0$ on $S^2V^{>0}$, so $q(s)\geq 0$ on $S^2V^{\geq 0}$. In particular for all $v\in V$, one has $q(v\otimes v)=q(v,v)\geq 0$, so $q$ is semi-positive.
 
\end{proof}

\subsection{Fitting pairs.}

To a pencil $P\in \Gra_d(\mathcal{Q})$ we associate the codimension $d$ projective subspace $\mathbb{P}(P^\circ)$, and its intersection $\mathbb{P}(P^\circ\cap S^2V^{\geq 0})$ in the convex set $\mathbb{P}(S^2V^{\geq 0})$.

\begin{defn}

We say that two elements $P_1, P_2\in \Gra^\text{mix}_d(\mathcal{Q})$ form a \emph{fitting pair} if the associated subspaces $\mathbb{P}(P_1^\circ \cap S^2V^{\geq 0})$ and $\mathbb{P}(P_2^\circ \cap S^2V^{\geq 0})$ are disjoint.
\end{defn}

The structure of the convex set $\mathbb{P}(S^2V^{\geq 0})$ is involved, but the set of its extreme points $S^2\mathbb{P}(V)\subset \mathbb{P}(S^2V^{\geq 0})$ is the projectivization of the set of rank one tensors, which is in one-to-one correspondence with $\mathbb{P}(V)$. Here the rank of a symmetric tensor is also the rank of the corresponding map $V^*\to V$. The \emph{extreme points} of a closed properly convex subset $C$ of projective space are the points $x\in C$ such that if $s\subset C$ is a projective segment containing it, $x$ is an endpoint of $s$. We show that the condition of being a fitting pair can be checked by looking only at $\mathbb{P}(V)$.

\medskip

Given a symmetric bilinear form $q\in \mathcal{Q}$, we will write respectively $\lbrace q=0\rbrace$, $\lbrace q>0\rbrace$ and $\lbrace q\geq 0\rbrace \subset \mathbb{P}(V)$ the set of lines that are respectively null, positive and non-negative for $q$.

\begin{prop}
\label{prop:WellFitted=FibrationCContinuous}
Let $P_1,P_2\in \Gra^{{\mix}}_d(\mathcal{Q})$ be two mixed pencils.
\begin{itemize}
\item[(i)] $(P_1,P_2)$ form a fitting pair,
\item[(ii)] there exist $q_1\in P_1$ and $q_2\in P_2$ such that $q_2-q_1$ is positive,
\item[(iii)] there exist $q_1\in P_1$ and $q_2\in P_2$ such that $\lbrace q_1\geq 0\rbrace\subset \lbrace q_2> 0\rbrace$.

\end{itemize}
\end{prop}

In order to prove this statement, the key fact that we will prove and use the is that the convex set $\mathbb{P}(S^2V^{\geq 0})$ does not have any segment as a \emph{face}, i.e. the intersection of this convex with any hyperplane intersecting trivially the interior $\mathbb{P}(S^2V^{> 0})$ is never a non-trivial segment.

\medskip

 The set $\lbrace q=0\rbrace$ is identified via the identification $\mathbb{P}(V)\simeq S^2\mathbb{P}(V)\subset \mathbb{P}(S^2V)$ with the intersection $\langle q\rangle^\circ\cap S^2\mathbb{P}(V)$. 

\begin{lem}
\label{lem:The convex C is nice}
For all linear hyperplane $H$ in $S^2V$ the extreme points of $\mathbb{P}(H\cap S^2V^{\geq 0})$ are also extreme points of $\mathbb{P}(S^2V^{\geq 0})$. In particular:
$$\mathbb{P}(H\cap S^2V^{\geq 0})=\Hull\left(H\cap S^2\mathbb{P}(V)\right).$$
Furthermore if $H=\langle q\rangle ^\circ$ for $q\in \mathcal{Q}$, $H\cap S^2V=S^2\lbrace q=0\rbrace$.
\end{lem}

Recall that the Krein–Milman theorem states that the convex Hull of the extreme points of a closed convex set $C$ in a finite dimensional vector space is equal to $C$.

\begin{proof}

Let $f$ be a face of $\mathbb{P}(S^2V^{\geq 0})$, \ie the intersection of $\mathbb{P}(S^2V^{\geq 0})$ with a projective hyperplane corresponding to $[q]\in \mathbb{P}(S^2V^*)=\mathbb{P}(\mathcal{Q})$ not intersecting $\mathbb{P}(S^2V^{>0})$ trivially. It is the intersection of this convex with a projective hyperplane corresponding to $[q]\in \mathbb{P}(S^2V^*)=\mathbb{P}(\mathcal{Q})$. The fact that this hyperplanes intersects $\mathbb{P}(S^2V^{>0})$ trivially implies that there is a lift $q\in \mathcal{Q}$ that is a semi-positive element, by Proposition \ref{prop:MixingCharacterization}. Let $W\subset V$ be the vector subspace of isotropic vectors for $q$. The corresponding face is equal to $\mathbb{P}(S^2W^{\geq 0})$.

Hence faces of $\mathbb{P}(S^2V^{\geq 0})$ are of the form $\mathbb{P}(S^2W^{\geq 0})$ for $W\subset V$ a linear subspace.
This has dimension $0$ or at least $2$, and therefore no face is a segment.

\medskip

Suppose that for some general projective hyperplane, some extreme point $[s]$ of $\mathbb{P}(H\cap S^2V^{\geq 0})$ is not an extreme point of $\mathbb{P}(S^2V^{\geq 0})$, then it belongs to the interior of a face of $\mathbb{P}(S^2V^{\geq 0})$, that has dimension at least $2$ as stated previously. The intersection of this face with $H$ contains therefore an open segment containing $s$, so $s$ is not an extreme point of $H\cap \mathbb{P}(S^2V^{\geq 0})$.
\end{proof}

\begin{proof}[{Proof of Proposition \ref{prop:WellFitted=FibrationCContinuous}}]

Let us prove that $(i)$ implies $(ii)$. The set $P_1^\circ\cap P_2^\circ$ is disjoint from $\mathbb{P}(S^2V^{\geq 0})$ if and only if there exist an element in $P_1+P_2$ that belongs to the dual of $\mathbb{P}(S^2V^{\geq 0})$, i.e. if there exist a positive bilinear form $q\in P_1+P_2$.
This form can be written as $q=q_1-q_2$ with $q_1\in P_1$ and $q_2\in P_2$, therefore $(i)$ implies $(ii)$.

\medskip

Moreover $(ii)$ implies $(iii)$. Indeed, if $q_2-q_1$ is positive then $\lbrace q_1\geq 0\rbrace\subset \lbrace q_2> 0\rbrace$.

\medskip

It only remains to show that $(iii)$ implies $(i)$. Lemma \ref{lem:The convex C is nice} implies that:
$$\mathbb{P}(\langle q_1\rangle^\circ\cap S^2V^{\geq 0})=\text{Hull}(S^2\lbrace q_1=0\rbrace).$$

Hence $q_2\in P_2\subset \mathcal{Q}=S^2V^*$ is positive on the cone $\langle q_1\rangle^\circ\cap S^2V^{\geq 0}\subset S^2V$ and therefore $\mathbb{P}(P_2^\circ)$ does not intersect $\mathbb{P}(\langle q_1\rangle^\circ\cap S^2V^{\geq 0})$.
\end{proof}

\subsection{Fitting directions.}

The space $\Gra_d(\mathcal{Q})$ inherits the structure of a smooth manifold. A chart around a point $P\in \Gra_d(\mathcal{Q})$ can be constructed given a subspace $Q$ such that $P\oplus Q=\mathcal{Q}$. We denote by $U_Q\subset \Gra_d(\mathcal{Q})$ the open set of elements transverse to $Q$. Every element of $U_Q$ can be written uniquely as the graph $\lbrace x+u(x)|x\in P\rbrace$ for some linear map $u:P\to Q$. Hence $U_Q$ can be identified with the vector space $\Hom(P,Q)$.

\medskip

The tangent space $T_P\Gra_d(\mathcal{Q})$ can be naturally identified with $\Hom(P,\mathcal{Q}/P)$, so that for each chart $U_Q$ containing $P$, the tangent space identifies with the tangent space in the chart via the identification $\Hom(P,Q)\simeq \Hom(P,\mathcal{Q}/P)$.

\medskip

Because of the identification $T_{P}\Gra_d(\mathcal{Q})\simeq T_{P^\circ}\Gra_{N-d}(\mathcal{Q}^*)$ where $N=\dim(\mathcal{Q})$, to such an element $\mathrm{v}$ corresponds an element: $$\mathrm{v}^\circ\in \Hom\left(P^\circ,S^2V/P^\circ\right).$$

\medskip

Let $\mathrm{v}^\circ\in T_P\Gra_d(\mathcal{Q}^*)$, that we see as an element of $\Hom(P^\circ,\mathcal{Q}/P^\circ)$. Let $(P^\circ_t)_{t\geq 0}$ a curve of elements of $\Gra_{N-d}(\mathcal{Q}^*)$ with $P^\circ_0=P^\circ$ and derivative $\mathrm{v}^\circ$ at $t=0$.
One can interpret $\Ker(\mathrm{v}^\circ)$ as the set of vectors in $P^\circ$ which remain in $P^\circ_t$ at first order around $t=0$.
More precisely the lines $[s]$ in $\Ker(\mathrm{v}^\circ)$ are exactly the ones such that for any Riemannian metric on $\mathbb{P}(V)$ and any such curve $(P^\circ_t)_{t>0}$ one has $d\left(\mathbb{P}\left(P^\circ_t\right),[s]\right)=o(t)$ close to $t=0$. 

\medskip

We will call \emph{fitting directions} in the Grassmanian $\Gra_d(\mathcal{Q})$ the tangent directions such that if $(P_t)\in \Gra_d(\mathcal{Q})$ moves in this direction, the corresponding codimension $d$ subspaces $\mathbb{P}\left(P_t^\circ \cap S^2V^{\geq 0}\right)$ are disjoint from $\mathbb{P}\left(P_0^\circ \cap S^2V^{\geq 0}\right)$, at order one, see Proposition \ref{prop:fitting vector description}.

\begin{defnProp}
\label{defn:fitting vectors}
We say that a vector $\mathrm{v}\in T_P\Gra_d(\mathcal{Q})$ is \emph{fitting} if one of the following equivalent statements holds:
\begin{itemize}
\item $\Ker(\mathrm{v}^\circ)\subset P^\circ$ intersects trivially $S^2V^{\geq 0}$,
\item $\Imu(\mathrm{v})\subset \mathcal{Q}/P$ contains $[q]$ where $q$ is a positive element.
\end{itemize}
\end{defnProp}

We now check that these two statements in the definition are indeed equivalent.

\begin{proof}
 We first show that $\Ker(\mathrm{v}^\circ)=\Imu(\mathrm{v})^\circ$, where we use the natural identification $(\mathcal{Q}/P)^*\simeq P^\circ$.

 We prove this by writing an equation that relates $\mathrm{v}$ and $\mathrm{v}^\circ$. By definition, $\Tr(qs)=0$ for $q\in P$ and $s\in P^\circ$. Let us fix $q\in P$ and $s\in P^\circ$ and choose some representatives $\overline{\mathrm{v}(s)}\in S^2V$ and $\overline{\mathrm{v}^\circ(q)}\in \mathcal{Q}$ for $\mathrm{v}(s)\in S^2V/P^\circ$ and $\mathrm{v}^\circ(q)\in \mathcal{Q}/P$. If $(P_t)$ is a smooth curve with $P_0=P$ and with derivative $\mathrm{v}$ at $t=0$, $s+t\overline{\mathrm{v}^\circ(p)}+o(t)\in P_t$ and $q+t\overline{\mathrm{v}(q)}+o(t)\in P^\circ_t$. Hence we get:
$$\Tr\left(q \overline{\mathrm{v}^\circ(s)}+ \overline{\mathrm{v}(q)}s\right)=0.$$

An element $s\in P^\circ$ satisfies $s\in \Ker(\mathrm{v}^\circ)$ if and only if $\Tr(\overline{\mathrm{v}(q)}s)=0$ for all $q\in \mathcal{Q}$, hence if and only if the corresponding linear form on $Q/P$ belongs to $\Imu(\mathrm{v})^\circ$.

\medskip

Now we prove the equivalence of the two definitions. If there exist $[q]\in \Imu(\mathrm{v})$ with $q$ positive, then for any $s\in\Ker(\mathrm{v}^\circ)$, $\Tr(q\circ s)=0$. Hence $s$ is not a positive tensor. Therefore $\Ker(\mathrm{v}^\circ)\subset P^\circ$ intersects trivially $S^2V^{\geq 0}$.

\medskip

Conversely if $\Ker(\mathrm{v}^\circ)$ intersects trivially $S^2V^{\geq 0}$, there exist $q\in \mathcal{Q}$ that do not vanish on $S^2V^{\geq 0}$, \ie $q$ is positive. The class $[q]\in \mathcal{Q}/P$ belongs to $\Imu(\mathrm{v})$.

\end{proof}

Fitting directions are related to fitting pairs. More precisely:

\begin{prop}
\label{prop:fitting vector description}
A vector $\mathrm{v}\in T_P\Gra_d(\mathcal{Q})$ is fitting if and only if, for every $\mathcal{C}^1$ curve $\gamma:[0,1]\to \Gra_d(\mathcal{Q})$ with $\gamma_0=P$ and $\gamma'_0=\mathrm{v}$, the pair $(\gamma_t,\gamma_0)$ is fitting for all $t>0$ small enough. 

\medskip

Moreover in this case for any Riemannian metric on $\mathbb{P}(S^2V)$ there exists an $\epsilon>0$ such that for any $t>0$ small enough the Riemannian distance between $\mathbb{P}(\gamma_0^\circ\cap S^2V^{\geq 0})$ and $\mathbb{P}(\gamma_t^\circ\cap S^2V^{\geq 0})$ is greater that $\epsilon t$.
\end{prop}

\begin{proof}
Let us fix a complement $H$ of $P^\circ$ in $S^2V$. Let $N=\dim(\mathcal{Q})=\frac{n(n+1)}{2}$. We identify a neighborhood of $P^\circ\subset \Gra_{N-d}(\mathcal{Q})$ with $\Hom(P^\circ,H)$. If there exist a non-zero element $s\in \Ker(\mathrm{v}^\circ)\cap S^2V^{\geq 0}$, one can consider the curve where $\gamma_t$ corresponds to 
$$q\in P^\circ\mapsto t\mathrm{v}(q).$$ 

The non-zero element $s \in S^2V^{\geq 0}$ belongs to $P^\circ=\gamma_0^\circ$ and $\gamma_t^\circ$ hence the pair $(\gamma_0, \gamma_t)$ is not fitting for any $t>0$.

\medskip

In general the element corresponding to $\gamma_t$ is equal for $t$ close to $0$ to the element of $\Gra_{N-d}(\mathcal{Q})$ which is the graph of the map $P^\circ \to H$:
$$q \mapsto t\mathrm{v}(q)+o(t).$$ 

The pair $(\gamma_o, \gamma_t)$ is fitting if and only if $\Ker(\mathrm{v})\cap S^2V^{\geq 0}=\lbrace 0\rbrace$, hence if $\mathrm{v}$ is fitting.

\medskip

In this case the distance between $\mathbb{P}(\gamma_0^\circ\cap S^2V^{\geq 0})$ and $\mathbb{P}(\gamma_t^\circ\cap S^2V^{\geq 0})$ grows at least linearly in $t$ for $t$ close to $0$.
\end{proof}

Fitting vectors can be thought of analogs of spacelike vectors in a pseudo-Riemannian manifold of signature $(d,N)$ for some $N>0$. For instance the set of spacelike vectors is the union of a family of cones parametrized by $\mathbb{S}^{d-1}$, and so is the set of fitting vectors, as shown in the following proposition. This analogy is also emphasized by Remark \ref{rem:SpacelikeFittingH3} and Theorem \ref{thm:Gauss map is fitting in H22}.

\medskip

For a vector space $W$ we write $\Sph W=(W\setminus \lbrace 0\rbrace)/\R_{>0}$

\begin{prop}
\label{prop:The space of fitting vector is a sphere of cones}
The set of fitting vectors in $T_P\Gra_d(\mathcal{Q})$ is equal to the union:

$$\bigcup_{[q]\in \Sph\mathcal{E}_{P}}C_{[q]}.$$ Here $C_{[q]}$ is the convex open cone of elements $\mathrm{v}\in T_P\Gra_d(\mathcal{Q})$ such that there exist a positive element in the class $\mathrm{v}(q)\in \mathcal{Q}/P$. 
\end{prop}

Recall that we use the identification $\mathrm{v}\in T_P\Gra_d(\mathcal{Q})\simeq \text{Hom}(P,\mathcal{Q}/P)$.

\begin{proof}
Because of the second part of Definition \ref{defn:fitting vectors} a vector is well-fitting if and only if it belongs to $C_{[q]}$ for some $[q]\in P$.
We just check that the sets $C_{[q]}$ are indeed open convex cones. 

Let us fix a complement $H$ of $P$ in $\mathcal{Q}$ to identify $T_P\Gra_d(\mathcal{Q})$ with $\Hom(P,H)$.
If $\mathrm{v}_1,\mathrm{v}_2$ lie in $C_{[q]}$ and if $\lambda,\mu\in \R_{>0}$, then for some $q_1,q_2\in P$ one has $\mathrm{v}_1(q)+q_1$ and $\mathrm{v}_2(q)+q_2$ positive. Therefore $(\lambda\mathrm{v}_1+\mu\mathrm{v}_2)(q)+\lambda q_1+\mu q_2$ is positive so $\lambda\mathrm{v}_1+\mu\mathrm{v}_2$ belongs to $C_{[q]}$.

\end{proof}

%The condition that $u:M\to \Gra_d(\mathcal{Q})$ is a fitting immersion is a local property. For it to imply that $u$ defines a globally fitting map, it is necessary to require some completeness, as in the theorey of maximal surfaces in $\mathbb{H}^{2,n}$. A notion of completeness that would do here would be to require the fibers to cover all of $\mathbb{P}(S^2V^{>0})$. This is guaranteed if the fitting immersion is periodic.
%
%\begin{prop}
%\label{prop:Complete fitting is a fibration}
%Let $N$ be a compact connected manifold of dimension $d$ with fundamental group $\Gamma$. Let $\rho:\Gamma\to \SL(V)$ be a representation and let $u:\widetilde{N}\to \Gra_d^{{\mix}}(\mathcal{Q})$ be a $\rho$-equivariant fitting immersion. Then it is a continuous fitting map, and the fibers cover all of $\mathbb{P}(S^2V^{>0})$. 
%\end{prop}
% 
%\begin{proof}
%\wip
%\end{proof}
 
\subsection{Fibration of a convex set and globally fitting maps.}

We consider continuous and smooth fibrations of the $\SL(V)$-invariant convex set $\mathbb{P}(S^2V^{\geq 0})$ by projective codimension $d$ subspaces. Let $M$ be a connected manifold of dimension $d$. We are interested in continuous injective maps, or smooth immersions $u:M\to \Gra_d(\mathcal{Q})$.

\medskip

We write $u^\circ(x)=(u(x))^\circ$. The map $u$ determines a collection of projective subsets of codimension $d$:
\begin{equation}
\label{eq:Fibers}
\left(\mathbb{P}\left( u^\circ(x)\cap S^2V^{\geq 0}\right)\right)_{x\in M}.
\end{equation}

Note here that for any linear subspace $A\subset V$ the projective codimension of $P(A)$ in $\mathbb{P}(V)$ is equal to the codimension of $A$ in $V$.

\medskip

If the image of $u$ contains only mixed elements, then all the submanifolds in \eqref{eq:Fibers} are non-empty.

\begin{defn}

We call a continuous map $u:M\to \Gra_d^{{\mix}}(\mathcal{Q})$ a \emph{globally fitting map} if the subsets in the collection \eqref{eq:Fibers} are disjoint. 

A continuous map $u:M\to \Gra_d^{{\mix}}(\mathcal{Q})$ is a \emph{locally fitting map} if for all $x\in M$ there is a neighborhood $U\subset M$ of $x$ such that $u_{|U}$ is a globally fitting map.

\end{defn}

Since $\dim(M)=d$, the invariance of domain implies that the sets \eqref{eq:Fibers} for a globally fitting map form a fibration for all $x\in M$ of a neighborhood in  $\mathbb{P}(S^2V^{\geq 0})$ of $\mathbb{P}\left(u^\circ(x)\cap S^2V^{\geq 0}\right)$.

We now consider immersions from a manifold $M$ of dimension $d$ whose tangent directions are all fitting.

\begin{defn}
\label{defn:Smooth fibration of C}
A smooth immersion $u:M\to \Gra_d^{{\mix}}(\mathcal{Q})$ is a \emph{fitting immersion} if $\mathrm{d}u(v)$ is fitting for all $v\in TM$.
\end{defn}

\begin{rem} In Appendix \ref{subsec:Example with SL2C} we discuss the simpler case when the pencils are hermitian pecils of quadrics on $\C^2$. In this cases, fitting immersions are fibrations of $\mathbb{H}^3$ by geodesics.
\end{rem}

%We will be interested in a smooth version of this proposition. We consider some tangent directions on $\Gra_d(\mathcal{Q})$ that we call \emph{fitting}. fitting directions have similar properties as spacelike directions in pseudo-Riemannian manifolds.
%
Because of proposition \ref{prop:fitting vector description}, fitting immersions are locally fitting maps.

\medskip 

The following proposition is the infinitesimal equivalent of Proposition \ref{prop:WellFitted=FibrationCContinuous}. We write the statement in a way to emphasize this analogy.

\medskip

Consider the tautological rank $d$ vector bundle $\pr:\mathcal{E}\to\Gra_d(\mathcal{Q})$. The fiber at $P\in \Gra_d(\mathcal{Q})$  of bundle $\mathcal{E}$ is identified with the vector subspace $P\subset \mathcal{Q}$. Since all the fibers are naturally identified with subsets of $\mathcal{Q}$, there is a tautological projection $\pi:\mathcal{E}\to \mathcal{Q}$ such that for $(P,q)$ in the tautological bundle, \ie $P\in \Gra_d(\mathcal{Q})$ and $q\in P$ with $q\in P$, $\pi(P,q)=q$. We still denote by $\pi:u^*\mathcal{E}\to \mathcal{Q}$ the corresponding projection with a slight abuse of notations.

Note that that an element in $T\mathcal{Q}$ is a pair $(q, \dot{q})\in T_q\mathcal{Q}$ and we say that it is positive if the symmetric bilinear form $\dot{q}\in \mathcal{Q}$ is positive. 
\begin{prop}
\label{prop:fitting immersions 3 items}

Given an immersion $u:M\to \Gra^{{\mix}}_d(\mathcal{Q})$, let $x_0\in M$. The following are equivalent:
\begin{itemize}
\item[(i)] the manifolds $\left(\mathbb{P}(u^\circ(x))\right)_{x\in M}$ define locally a smooth fibration of an open neighborhood of $\mathbb{P}\left(u^\circ(x_0)\cap S^2V^{\geq 0}\right)$,
\item[(ii)] for all $v\in T_{x_0}M$, $\mathrm{d}u(v)$ is fitting,
\item[(iii)] for all  $v\in T_{x_0}M$ there exist $\mathrm{w}\in T\mathcal{E}$ such that $\mathrm{d}\pr(\mathrm{w})=\mathrm{d}u(v)$ and $\mathrm{d}\pi(\mathrm{w})\in T\mathcal{Q}$ is positive.
\end{itemize}
\end{prop}

\begin{proof}

Note that since $M$ has dimension $d$ and $u^\circ(x_0)$ has codimension $d$, the statement $(i)$ is equivalent to having for any Riemannian distance $d_R$ on $\mathbb{P}(S^2V)$ and $d_M$ on $M$, for some $\epsilon>0$ when $x$ is close to $x_0$:
$$d_R \left(\mathbb{P}\left(u^\circ(x)\cap S^2V^{\geq 0}\right),\mathbb{P}\left(u^\circ(x_0)\cap S^2V^{\geq 0}\right)\right)\geq \epsilon d_M(x,x_0).$$

Proposition \ref{prop:fitting vector description} shows that the fitting condition is equivalent to having this distance growing linearly, hence the statements $(i)$ and $(ii)$ are equivalent.

%The choice of $\mathrm{v}$ is a way to choose the derivative of a section of $u^*\mathcal{E}$, such that the differences between the corresponding elements of $\mathcal{E}$ are positive.

\medskip

A pair $(P,q_0)\in \Gra_d(\mathcal{Q})\times \mathcal{Q}$ correspond to an element in $\mathcal{E}$ if and only if $q_0\in P$. 
A pair $\left(\mathrm{v},\dot{q}\right)\in T_P\Gra_d(\mathcal{Q})\times T_q\mathcal{Q}$ can be written as $\left(\mathrm{d}\pr(\mathrm{w}),\mathrm{d}\pi(\mathrm{w})\right)$ for some $\mathrm{w}\in T\mathcal{E}$ if and only if $\dot{q}$ belongs to the class defined by  $\mathrm{v}(q_0)$. 

Let us show this last claim. Let $H$ be a complement of $P$ in $\mathcal{Q}$ let $u:[0,1]\to \Hom(P,H)$ and $q:[0,1]\to \mathcal{Q}$ be smooth curves with derivative $\mathrm{v}$ and $\dot{q}$ at $q=0$ and such that $q(t)$ belongs the the graph of $u(t)$ for $t\in[0,1]$. For some smooth curve $\tilde{q}:[0,1]\to P$ with $\tilde{q}(0)=q_0$ and all $t\in [0,1]$:
$$q(t)= \tilde{q}(t)+u(t)(\tilde{q}(t)).$$
Differentiating this at $t=0$ and we get exactly $\mathrm{v}(q_0)+\tilde{q}'(0)=\dot{q}$, so $\dot{q}$ belongs to the class defined by $\mathrm{v}(q_0)$, since $\tilde{q}'(0)\in P$. Reciprocally if this holds, one can construct such a curve $\tilde{q}$, so the pair corresponds to an element of $T\mathcal{E}$.

\medskip

We conclude that $(iii)$ is equivalent to the second characterization of fitting vectors in Definition \ref{defn:fitting vectors}: one can find such a positive lift $\mathrm{w}$ if and only if one can find a class in $\Imu(\mathrm{v})$ that contains a positive element.

%A pair $\left(\mathrm{v},(q,\dot{q})\right)\in T_P\Gra_2(\mathcal{Q})\times T\mathcal{Q}$ can be written as $\left(\mathrm{d}\pr(\mathrm{w}),\mathrm{d}\pi(\mathrm{w})\right)$ if and only if $q\in P$ and $\dot{q}$ belongs to the class defined by  $\mathrm{v}(q)$. Indeed $q\in P$ can move in the direction $\dot{q}$ while stating in the pencil while the pencil $P$ is moving in the direction $\mathrm{v}$. 

\end{proof}

In Proposition \ref{prop:GaussMapSymmetric space for totally geodesic immersions} we show how to construct some examples of fitting maps from a totally geodesic immersion in the symmetric space.

\begin{rem}
The definition of a fitting immersion and the previous two propositions can be generalized to the more general setup when $S^2V$ is replaced by a vector space $W$ and $S^2V^{\geq 0}$ is replaced by a closed proper convex cone $C$ in $W$. In this setup positive quadrics should be replaced by elements in the dual cone of $C$ in $W^*$.
\end{rem}

\section{Fitting flows.}
\label{sec:Fitting flows}

In this section we define the notion of a fitting flow, and study the consequence of the existence of such a flow. We show next that such flows always exist for fitting immersions. In this section let us fix a map $u:M\to \Gra_d(\mathcal{Q})$.

\subsection{Definition and application of fitting flows.}

The pullback $u^*\mathcal{E}$ of the tautological bundle $\pr:\mathcal{E}\to \Gra_d(\mathcal{Q})$ defines a rank $d$ vector bundle over $M$. We define the sphere bundle $\mathbb{S}u^*\mathcal{E}$ as the quotient of the vector bundle $u^*\mathcal{E}$ by the action of positive scalars. 

Recall that $M$ has dimension $d$. We consider flows on $\mathbb{S}u^*\mathcal{E}$ so that some form of contraction occurs along the flow lines. We denote also by $\pr$ the bundle maps $u^*\mathcal{E}\to M$, $\Sph u^*\mathcal{E}\to M$, with a slight abuse of notations. 

\medskip

The fiber at $P\in \Gra_d(\mathcal{Q})$ of the bundle $\mathcal{E}$ is identified with the vector subspace $P\subset \mathcal{Q}$. Recall that there is a natural projection $\pi:\mathcal{E}\to \mathcal{Q}$ defined by $\pi(P,q)=q$. We still denote by $\pi:u^*\mathcal{E}\to \mathcal{Q}$ the corresponding projection with a slight abuse of notations.

\begin{defn}
\label{defn:Fitting flow}
A \emph{fitting flow} for a continuous map $u:M\to\Gra_d(\mathcal{Q})$ is a continuous flow $(\Phi_t)_{t\in \R}$ on $\Sph u^*\mathcal{E}$ such that

one can choose representatives $q$ of $[q]$ and $q'$ of $[q']=\Phi_t([q])$ such that $\pi(q')-\pi(q)\in \mathcal{Q}$ is positive.

\end{defn}

%\begin{rem}
%If one is given a map $u:M\to \Gra_d(\mathcal{Q})$ with an identification of $\Sph TM$ and $\Sph u^*\mathcal{E}$ and a special metric on $M$, one can consider the geodesic flow and check if it is 
%\end{rem}

Note that the last condition is equivalent to asking that $\lbrace \pi(q)\geq 0\rbrace \subset \lbrace \pi(q')> 0\rbrace $ in $\mathbb{RP}^{2n-1}$.
%\begin{defn}
%We say that a \emph{fitting flow} for an immersion $u:M\to\Gra_d(\mathcal{Q})$ is a flow $\Phi$ on $\Sph u^*\mathcal{E}$ that is symmetric, \ie $\Phi_t(-\mathrm{v})=\Phi_{-t}(\mathrm{v})$, and such that every partial flow line $\gamma:I\to \Sph u^*\mathcal{E}$ can be lifted to a $\overline{\gamma}:I\to \mathcal{E}$ so that $\mathrm{d}\pi(\overline{\gamma})\in \mathcal{Q}$ is positive.
%\end{defn}
Along such flows, the associated quadric hypersurfaces are nested into one another. In particular if $u$ admits a fitting flow it is locally a fitting map.

\begin{lem}
\label{lem:Embedded fitting flow lines}
Let $u:M\to\Gra^{{\mix}}_d(\mathcal{Q})$ be a continuous map that admits a fitting flow. The projection to $\Gra^{{\mix}}_d(\mathcal{Q})$ of the flow lines of the fitting flow are embedded.
\end{lem}

\begin{proof}
Assume by contradiction that for some $t_0>0$ and $q\in \Sph u^*\mathcal{E}$ one has $u(x)=\pr(q)=\pr(\Phi_{t_0}(q))$. The fact that the flow is fitting implies that for some $\lambda>0$, $\lambda\pi(\Phi_{t_0}(q))-\pi(q)\in \mathcal{Q}$ is positive. This positive quadric belongs to $u(x)$, contradicting the fact that $u(x)\in  \Gra_d^{{\mix}}(\mathcal{Q})$ by Proposition \ref{prop:MixingCharacterization}.

\end{proof}

Some fitting flows can be constructed by taking a geodesic flow on $M$ for some Riemannian metric and identifying $u^*\mathcal{E}$ with the tangent bundle to $M$. In general the projections of the flow lines of a fitting flow satisfy the following topological property, which is clearly satisfied for geodesic flows.

\begin{lem}
\label{lem:HomotopyWellFitted}
Let $u:M\to\Gra^{{\mix}}_d(\mathcal{Q})$ be a continuous map equipped with a fitting flow $\Phi$ in a neighborhood of $x\in M$. For $t>0$ small enough the sphere $S_t:[q]\in \Sph u^*\mathcal{E}_{|x}\mapsto p\circ \Phi_t([q])\in M$ is homotopic to a generator of the homology of $U\setminus \lbrace x\rbrace$ for any open neighborhood $U$ of $x$ in $M$ that is diffeomorphic to $\R^d$.

\end{lem}
 The proof relies on the fact the dimension of $M$ is equal to $d$, and hence the manifolds $\mathbb{P}(u^\circ(x)\cap S^2V^{\geq 0})$ locally define a fibration of $\mathbb{P}(S^2V^{\geq 0})$.

\begin{proof}

Let $P=u(x)$. Let $s_0\in P^\circ\subset S^2V$ be a positive tensor, which exists since the pencil $P$ is assumed to be in $\Gra^{{\mix}}_d(\mathcal{Q})$.
Let $P'\in \Gra_d(S^2V)$ be a supplement of $P^\circ$.
Since $u$ is continuous for all $y$ close enough to $x$ the vector subspace $u^\circ(y)$ is transverse to $P'$, therefore there exist a unique vector $\phi(y)\in P'$ such that:
 $$s_0+\phi(y)\in (s_0+P') \cap u^\circ(y) \subset S^2V^{>0}.$$
This defines a continuous map $\phi$ from a neighborhood $U$ of $x\in M$ to $P'$.

\medskip

Let $[q]\in \Sph u^*\mathcal{E}_x$.  For all $t$ such that $s\circ \Phi_t(q)\in U$, the linear form $\pi\left(\Phi_t(q)\right)\in \mathcal{Q}=S^2V^*$ vanishes on $\phi\left( s\circ \Phi_t(q)\right)\in S^2V$ since this point belongs to $u^\circ(\pr\circ \Phi_t(q))$. Moreover $\pi\left(\Phi_t(q)\right)-\pi(q)$ is a positive bilinear form since $\Phi$ is a fitting flow.

In particular $\pi(q)\in \mathcal{Q}=S^2V^*$ is always negative on $\phi\left( \pr\circ \Phi_t(q)\right)\in S^2V$.
% Let $[q]\in \Sph u^*\mathcal{E}_x$.  For all $t$ such that $p\circ \Phi_t(q)\in U$, the evaluation of $\pi(q)\in \mathcal{Q}=S^2V^*$ to $\phi\left( p\circ \Phi_t(q)\right)\in S^2V$ is positive since $\pi\left(\Phi_t(q)\right)-\pi(q)$ is positive and the linear form $\pi\left(\Phi_t(q)\right)$ vanishes on $\phi\left( p\circ \Phi_t(q)\right)$
Hence for $t$ small enough $[\phi\circ S_t]:\Sph u^*\mathcal{E}_x \to \Sph P'$ has the same degree as $[\pi]: \Sph u^*\mathcal{E}_x \to \Sph P'^*\simeq \Sph P$ which associates to $[q]\in \Sph u^*\mathcal{E}_x $ the class $[\pi(q)]\in \Sph P$. The map $[\pi]$ is a diffeomorphism, so in particular $1=|\deg([\pi])|=|\deg(\phi)\deg(S_t)|$. Hence $S_t$ is a generator of the homotopy group of $U\setminus \lbrace x\rbrace$. 

\end{proof}

\begin{rem}
In particular for an immersion $u:M\to \Gra^{{\mix}}_d(\mathcal{Q})$ that admits a fitting flow, choosing continuously an orientation of the pencils $u(x)$ for $x\in M$ is equivalent to choosing an orientation of $M$.
%This lemma in particular defines a natural orientation on $M$ for every map $u:M\to\Gra^{o}_{d,+}(\mathcal{Q})$ into the space of oriented subspaces if the map admits a fitting flow. Indeed an orientation is given by the choice of a generator for tho homology of $U\setminus \lbrace x\rbrace$.
\end{rem}

Let $M=\widetilde{N}$ where $N$ is compact and $\Gamma=\pi_1(N)$. Let $\rho:\Gamma\to\SL(V)$ be a representation. Recall that a $(C,D)$-quasi geodesic in a metric space $(X,d)$ is a curve $\gamma:\R\to X$ such that for all $t_1,t_2\in \R$: 
$$ 1/C|t_1-t_2|-D\leq d(\gamma(t_1),\gamma(t_2))\leq C|t_1-t_2|+D.$$

\begin{prop}
\label{prop:Flow lines of a fitting flow are quasi geodesic and exist between points}
Let $u:\widetilde{N}\to\Gra^{{\mix}}_d(\mathcal{Q})$ be a $\rho$-equivariant continuous map that admits a $\rho$-equivariant fitting flow $\Phi$. There exist $C,D>0$ such that the projection to $\widetilde{N}$ of the flow lines of $\Phi$ are $(C,D)$-quasi geodesics. Moreover for every $(x,y)\in \widetilde{N}^2$ there exist a flow line whose projection to $\widetilde{N}$ starts at $x$ and ends at $y$.
\end{prop}

In the previous statement one can consider any fixed Riemannian metric on the compact space $\Sph u^*\mathcal{E}/\rho(\Gamma)$. 

\begin{rem}
\label{rem:Fitting flow implies globally fitting}
Note that since there exist a flow line between any pair of points in $\widetilde{N}$, the map $u$ is necessarily a globally fitting map.
\end{rem}

In this proof we will use the \emph{Hilbert distance} $d_H$ on the properly convex domain $\mathbb{P}(S^2V^{>0})$. It is defined which is defined using the cross ratio as $d_H([s_1],[s_2])=\log\left(\text{cr}\left([s^-],[s_1],[s_2],[s^+]\right)\right)$ where $[s_1],[s_2]\in \mathbb{P}(S^2V^{>0})$ and $[s^-]$, $[s^+]$ are the intersection of the projective line through $[s_1]$ and $[s_2]$ with the boundary of the domain $\mathbb{P}(S^2V^{>0})$. If the closure in $\mathbb{P}(S^2V)$ of two sets $A,B\subset \mathbb{P}(S^2V^{>0})$ are disjoint, then the two sets are at positive distance for the Hilbert distance. Indeed for every Riemannian metric $d_\Sph$ on the compact manifold $\mathbb{P}(S^2V)$, the Hilbert distance between any two points in $\mathbb{P}(S^2V^{>0})$ is bounded from below by some uniform multiple of $d_\Sph$. 

\begin{proof}
For all $[q]\in \Sph u^*\mathcal{E}$ we consider the convex $\mathbb{P}(\langle \pi(q)\rangle ^\circ \cap S^2V^{>0})$.

We choose a continuous and $\rho$-equivariant map $\sigma:\Sph u^*\mathcal{E}\to \mathbb{P}(S^2V^{>0})$ with the property that $\sigma([q])\in \mathbb{P}(\langle \pi(q)\rangle ^\circ \cap S^2V^{>0})$ for all $[q]\in \Sph u^*\mathcal{E}$. One can always find such a choice, as it amonts to finding a section of an open ball bundle over a compact set. Indeed $\mathbb{P}(\langle \pi(q)\rangle ^\circ \cap S^2V^{>0})$ is a non-empty convex set of codimension $1$ that forms an open ball for all  $[q]\in \Sph u^*\mathcal{E}$.

\medskip

We fix a $\Gamma$-invariant Riemannian metric $g$ on $\Sph u^*\mathcal{E}$ with associated distance $d_g$. We set $C_1$ to be the supremum of $d_H\left(\sigma([q_1]),\sigma([q_2])\right)$ for all $[q_1],[q_2]\in \Sph u^*\mathcal{E}$ such that $d_g([q_1],[q_2])\leq 1$, which exists since $\Gamma$ acts cocompactly on $\widetilde{N}$. The following inequality follows for all $[q_1],[q_2]\in \Sph u^*\mathcal{E}$ from the triangular inequality:

 $$d_H\left(\sigma([q_1]),\sigma([q_2])\right)\leq C_1 d_g([q_1],[q_2])+C_1.$$ 
 
Let $n$ be the unique integer $d_g([q_1],[q_2])\leq n< d_g([q_1],[q_2])+1$. One can find elements $x_0=[q_1], x_1,\cdots, x_n=[q_2]$ in $\Sph u^*\mathcal{E}$ such that for all $1\leq i\leq n$ one has $d_g(x_{i-1},x_i)\leq 1$. The triangular inequality for the Hilbert distance implies that:
$$d_H\left(\sigma([q_1]),\sigma([q_2])\right)\leq nC_1.$$
 
\medskip 
 
We now set $C_2$ to be the supremum of $d_H\left(\sigma([q_1]),\sigma(\Phi_t([q]))\right)$ for all $[q_1]\in \Sph u^*\mathcal{E}$ and $0\leq t\leq 1$. Similarly we get the following inequality for all $[q]\in \Sph u^*\mathcal{E}$ and $t\geq 0$:
$$d_H\left(\sigma([q]),\sigma(\Phi_t([q]))\right)\leq C_2 t+C_2.$$
 
  Let $K$ be a compact fundamental domain for the action of $\Gamma$ on $\Sph u^*\mathcal{E}$ and let $\epsilon$ be the infimum of the Hilbert distance for any $[q]\in K$ between $\mathbb{P}\left(\langle\pi(q)\rangle^\circ \cap S^2V^{>0}\right)$ and $\mathbb{P}\left(\langle \pi(\Phi_1(q))\rangle^\circ \cap S^2V^{>0}\right)$. Since the flow is fitting, the closures of these two sets in $\mathbb{P}(S^2V)$ are disjoint for any $[q]$, and hence their Hilbert distance is positive. Since $K$ is compact, the infimum $\epsilon$ is also positive. 
%  Let $C$ be the supremum for $q\in \Sph u^*\mathcal{E}$ and $t\in [0,1]$ of the Hilbert distance between $s(q)$ and $s(\Phi_t(q))$.

The Hilbert distance between $\sigma(q)$ and $\sigma(\Phi_t(q))$ for $t>0$ and $q\in \Sph u^*\mathcal{E}$ is greater than $\epsilon(t-1)$.
Indeed for all integer $0\leq n\leq t$ the projective segment between $\sigma(q)$ and $\sigma(\Phi_t(q))$, which is a geodesic for the Hilbert distance, intersects $\mathbb{P}\left(\langle \pi(\Phi_n(q))\rangle^\circ \cap S^2V^{>0}\right) $ in exactly one point $x_n$. Hence the Hilbert distance between $x_n$ and $x_{n+1}$ for $0\leq n\leq t-1$ is at least $\epsilon$.

\medskip

Putting all of these inequalities together we get that for all $t\geq 0$ and $[q]\in \Sph u^*\mathcal{E}$:
$$ \frac{\epsilon}{C_1}(t-1)-1\leq d_g\left(\sigma([q]),\sigma(\Phi_t([q]))\right)\leq C_2 (t+1) $$

Hence the flow lines are quasi-isometric embeddings.

\medskip

We now check that flow lines exist between any pair of points. 
Let $x\in M$. Given $t\in \R$ we consider the $d$-sphere $S_t:q\in \Sph u^*\mathcal{E}_{|x}\mapsto \Phi_t(q)\in u^*\mathcal{E}$.
Suppose that some $y\in M$, avoids the sphere $S_t$ for all $t>0$. Consider a curve $\eta$ between $x$ and $y$.
The homological intersection between this segment and the spheres $S_t$ in $M\setminus \lbrace x,y\rbrace$ is constant, and is equal to zero for $t$ large enough since the spheres $S_t$ are then uniformly far from $x$. However for $t$ small enough, the homotopy class of $S_t$ is the one of any small sphere encircling $x$ by Lemma \ref{lem:HomotopyWellFitted}. This leads to a contradiction since such a sphere will have homological intersection equal to $1$ or $-1$ with the curve $\eta$. Hence there exists a flow line joining any pair of points.

\end{proof}

As a consequence, we get the following.

\begin{cor}
\label{cor:Globally fitting}
Let $u:\widetilde{N}\to\Gra^{{\mix}}_d(\mathcal{Q})$ be a $\rho$-equivariant continuous map that admits an equivariant fitting flow, then it is an embedding. In particular it is a globally fitting map. Moreover $ \rho:\Gamma\to \SL(V)$ is a quasi-isometric embedding. 
\end{cor}

\begin{proof}
Since there are flow lines between any pair of points in $\widetilde{N}$, for every $x \neq y\in M$ one can find $q\in u(x)$ and $q'\in u(y)$ such that $q-q'$ is positive. Since the pencils $u(x)$ and $u(y)$ are mixed, one cannot have $q-q'\in u(x)=u(y)$ so $u(x)\neq u(y)$. Furthermore the pair $(u(x),u(y))$ is a fitting pair for all $x \neq y\in M$, hence $u$ is a globally fitting map by Proposition \ref{prop:WellFitted=FibrationCContinuous}.

\medskip

For any $\Gamma$-invariant Riemannian metric $g$ on $\Sph \mathcal{E}$, any map $\sigma$ as the proof of Proposition \ref{prop:Flow lines of a fitting flow are quasi geodesic and exist between points} is a quasi-isometric embedding. Indeed let $\delta$ be the maximum of $d_g(x,x')$ or $d_H(\sigma(x),\sigma(x'))$ for any $x,x'$ in the same fiber of $\pr:\Sph \mathcal{E}\to \widetilde{N}$. For every $x,y\in \Sph \mathcal{E}$ one can find $x',y'$ in the same fibers respectively as $x,y$ and in the same flow line for $\Phi$. Hence for the constants $C,D$ from Proposition \ref{prop:Flow lines of a fitting flow are quasi geodesic and exist between points}:
$$\frac{1}{C}d_g(x,y)-D-(\frac{2}{C}+2)\delta\leq d_H(\sigma(x),\sigma(y))\leq Cd_g(x,y)+D+(2C+2)\delta .$$

Hence $\rho$ is a quasi-isometric embedding.
\end{proof}

\subsection{Existence of fitting flows}

We now prove that the existence of a fitting flow is guaranteed on compacts for fitting immersions. It is not clear if it is the case in general for locally fitting maps. 

\begin{prop}
\label{prop:ExistenceFittingFlow}
Let $M$ be a manifold of dimension $d$.
\begin{itemize}

\item[$(i)$]A fitting immersion $u:M\to\Gra_d(\mathcal{Q})$ admits a fitting flow.

\item[$(ii)$]An equivariant fitting immersion $u:M\to\Gra_d(\mathcal{Q})$ for a representation $\rho:\Gamma\to \SL(V)$, and a proper action of $\Gamma$ on $M$ admits a $\rho$-equivariant fitting flow.

\end{itemize}
\end{prop}

Recall that $\pr:\mathcal{E}\to \Gra_d(\mathcal{Q})$ is the bundle map and $\pi:\mathcal{E}\to \mathcal{Q}$ is the tautological map. 

\medskip

In order to construct the fitting flow we construct the vector field $W$ on $\Sph u^*\mathcal{E}$ that generates the flow. This first step of the proof uses crucially the hypothesis that $\dim(M)=d$.

\medskip

%
%%A vector $\mathrm{w}\in T\mathcal{E}$ is characterized by its projection $\mathrm{d}\pr(\mathrm{w})$ to $T\Gra_d(\mathcal{Q})$ and its projection $\mathrm{d}\pi(w)$ to $T\mathcal{Q}$. Moreover to such a pair $((P,\mathrm{v}),(q,\dot{q}))\in T\Gra_d(\mathcal{Q})\times T\mathcal{Q}$ corresponds a vector $\mathrm{w}\in T\mathcal{E}$ if and only if $q\in P$ and $[\dot{q}]=\mathrm{v}(q)\in \mathcal{Q}/P$, where $\mathrm{v}$ is identified with an element of $\Hom(P,\mathcal{Q}/P)$. R
%
%\medskip
%
%The first step of the proof uses crucially the hypothesis that $\dim(M)=d$. 

\begin{lem}
\label{lem:fitting flow sphere}
 Let $u:M\to\Gra_d(\mathcal{Q})$ be a fitting immersion with $\dim(M)=d$. For every $x\in M$ and any $q\in u(x)$, there exists a lift $\mathrm{w}\in T_{(u(x),q)}u^*\mathcal{E}$ such that $\mathrm{d}\pi(\mathrm{w})\in T\mathcal{Q}$ is positive.
\end{lem}

We say that an element $(q,\dot{q})\in T\mathcal{Q}$ is positive if $\dot{q}\in \mathcal{Q}$ is positive. Given any quadric in the pencil $u(x)$, we want to find an infinitesimal direction in which to move this quadric as well as the pencil containing it inside the image of $u$ so that the derivative of the quadrics is positive. 

\medskip

This lemma is an inverse of point $(iii)$ of Proposition \ref{prop:fitting immersions 3 items}: here we fix an element $q\in u(x)$ whereas before we were fixing a $v\in TM$. In order to construct this inverse we will use the fact that a continuous odd map between spheres of equal dimensions must have odd degree \cite[{Proposition 2B.6}]{Hatcher} and therefore it must be surjective, as any non-surjective map into a sphere is contractible.

\begin{proof}

Let us fix $x\in M$. We construct a continuous map :
$$\phi:T_xM\setminus \lbrace 0\rbrace\to u^*\mathcal{E}_{x}\setminus  \lbrace 0\rbrace.$$

We require that this map satisfies $\lambda \phi(v)=\phi(\lambda v)$ for all $v\in T_xM\setminus \lbrace 0\rbrace$ and $\lambda\in \R$. Note that in particular $\phi$ defines an odd map:
$$\overline{\phi}:\Sph T_xM\to \Sph u^*\mathcal{E}_x.$$

We furthermore construct a lift:
$$\psi:T_xM\setminus \lbrace 0\rbrace\to Tu^*\mathcal{E}_.$$

In other words we assume that $\psi(v)\in T_{\phi(v)}(u^*\mathcal{E})$ for $v\in T_xM\setminus \lbrace 0\rbrace$.  
We require that $\psi(\lambda v)=\lambda(\mathrm{d}m_\lambda)(\psi(v))$ for all $v\in T_xM\setminus \lbrace 0\rbrace$ and $\lambda\in \R$ where $m_\lambda:u^*\mathcal{E}\to u^*\mathcal{E}$ is the multiplication by $\lambda$. We require $\mathrm{d}\pi(\psi(v))\in T\mathcal{Q}$ to be positive for all $v\in T_xM\setminus \lbrace 0\rbrace$.  

Finally we will make this construction so that in addition $\mathrm{d}\pr\left( \psi(v)\right)=v$ for $v\in T_xM\setminus \lbrace 0\rbrace$, but this property will be only used during the construction. 

\medskip

The following diagram illustrates the situation.

\begin{center}
\begin{tikzcd}[row sep=2.5em]
T u^*\mathcal{E}\arrow[loop, out=135+30, in=135-30,looseness=3,"\mathrm{d}m_\lambda"]\ar[r, "\mathrm{d}\pi"]\ar[d, "\mathrm{d}\pr"']\ar[rd]
& T\mathcal{Q} \ar[rd]
\\
TM \ar[rd] \ar[u, dashed, "\psi"', bend right]\ar[r, dashed,"\phi"]
& u^*\mathcal{E}\ar[loop, out=315+30, in=315-30,looseness=4,looseness=3, "m_\lambda"] \ar[r, "\pi"]\ar[d, "\pr"']
& \mathcal{Q}
\\
  & M
\end{tikzcd}
\end{center}

%\begin{center}
%\begin{tikzcd}[row sep=2.5em]
% T u^*\mathcal{E} \arrow[r]&u^*\mathcal{E} \arrow[r, "\pi"]\arrow[d,"p"]& \mathcal{Q}\\
%TM \arrow[u, "{\psi}"]\arrow[ur, "{\phi}"]\arrow{r}& M&
%\end{tikzcd}
%\end{center}
 
% and $\psi(- v)=-\mathrm{d}m_{-1}(\psi(v))$ for all $v\in T_xM\setminus \lbrace 0\rbrace$ in order to construct a fitting flow.
%\begin{figure}
%\begin{center}
%\begin{tikzcd}[row sep=2.5em]
%
%Tu^*\mathcal{E} \arrow[r]&u^*\mathcal{E} \arrow[r, "\pi"]\arrow[d, "p"]& \mathcal{Q}\\
%TM \arrow[u, "\psi"]\arrow[ur, "\phi"]\arrow[r]& M&
%
%
%
%%T\mathcal{E} \arrow[rr,"\mathrm{d}\phi"] \arrow[dr,swap,"\mathrm{d}\pr"] \arrow[dd,swap,"\pi_\mathcal{E}"] &&
%%  T\left(T\mathbb{P}(S^2V)\right) \arrow[dd,swap,"\pi_{T\X}" near start] \arrow[dr,"\mathrm{d}\pi_\X"] \\
%%& T\mathcal{Y} \arrow[rr,crossing over,"\mathrm{d}\pr_\mathcal{Y}" near start] &&
%%  T\X \arrow[dd,"\pi_\X"] \\
%%\mathcal{E} \arrow[rr,"\phi" near end] \arrow[dr,swap,"p_\mathcal{E}"] && T\mathbb{P}(S^2V) \arrow[dr,swap,"\pi_\X"] \\
%%& \mathcal{Y} \arrow[rr,"p_\mathcal{Y}"] \arrow[uu,<-,crossing over,"\pi_\mathcal{Y}" near end]&& \mathbb{P}(S^2V)
%\end{tikzcd}
%\end{center}
%\caption{A commutative diagramm}
%\label{fig:DiagramCube}
%\end{figure}

If we can construct such continuous maps, the fact that $\overline{\phi}$ is an odd map between two spheres of the same dimension implies that it is homotopically non-trivial and therefore surjective. In particular for all $[q]\in \Sph\mathcal{E}_{u(x)}\simeq\Sph u(x)$ there exist a $v\in T_xM$  such that $\phi(v)=q$. The element $\mathrm{w}=\psi(v)$ then satisfies the required conditions, so this finishes the proof.

\medskip

Now let us construct the maps $\phi$ and $\psi$. We first show that for all $v\in T_xM\setminus \lbrace 0\rbrace$ we can define $\phi(v)$ and $\psi(v)$, and then we explain how to glue these maps together to get a continuous map.

\medskip

Since $u$ is a fitting immersion, the point $(iii)$ of Proposition \ref{prop:fitting immersions 3 items} implies that given $v_0\in T_xM$ one can construct $\phi(v_0)\in u^*\mathcal{E}_{x}$ and $\psi(v_0)\in T_{\phi(v_0)} u^*\mathcal{E}$ such that $\mathrm{d}\pi(\psi(v_0))$ is positive and $\mathrm{d}\pr\left( \psi(v_0)\right)=v_0$. 

Note that the condition that $\mathrm{d}\pi(\psi(v))$ is positive is an open condition and the condition that $\mathrm{d}\pr\left( \psi(v)\right)=v$ requires that $\psi$ is a section of an affine sub-bundle. Hence for every $v\in T_xM\setminus \lbrace 0\rbrace$ we can find a small neighborhood $\mathcal{S}$ in a sphere in $T_xM$ containing $v_0$ on which we can define $\phi$ and a lift $\psi$ such that $\mathrm{d}\pi(\psi(v))$ is positive and $\mathrm{d}\pr\left( \psi(v)\right)=v$ for all $v\in \mathcal{S}$. We take $\mathcal{S}$ small enough so that it does not contain any antipodal pair of points.

We define $U$ to be the set of non-zero elements $\lambda v$ for all $\lambda\in \R$ and $v\in \mathcal{S}$, and we extend $\phi$ and $\psi$ to $U$ in a homogeneous way. We define $\phi$ on $U$ so that $\phi_i(\lambda v)=\lambda\phi_i(v)$ for all $\lambda\in \R$ non-zero and $v\in \mathcal{S}$. We set $\psi(\lambda v)=\lambda(\mathrm{d}m_\lambda) \psi_i(v)$ for $\lambda\in \R$, where $m_\lambda$ is the multiplication by $\lambda$ on $u^*\mathcal{E}_x$. Note that $\mathrm{d}\pi\left(\psi_i(\lambda v)\right)=\lambda ^2\mathrm{d}\pi\left(\psi_i(v)\right)$ is positive and $\mathrm{d}\pr\left( \psi_i( \lambda v)\right)=\lambda v$ for all $\lambda v\in U$. Indeed $\pi \circ m_\lambda=\lambda \pi$ so $\mathrm{d}\pi\circ \mathrm{d}m_\lambda=\lambda\mathrm{d}\pi$ and $p\circ m_\lambda=p$ so $\mathrm{d}\pr\circ\mathrm{d}m_\lambda=\mathrm{d}\pr$.

\medskip

%\begin{center}
%\begin{tikzcd}[row sep=2.5em]
%TM\arrow[d] &  Tu^*\mathcal{E} \arrow[r, "\mathrm{d}\pi"]\arrow[l, "\mathrm{d}\pr"]\arrow[d]\arrow[loop above, "\mathrm{d}m_\lambda"]& T\mathcal{Q}\arrow[d] \\
%M & u^*\mathcal{E} \arrow[r, "\pi"]\arrow[l, "p"]\arrow[loop below, "m_\lambda"]& \mathcal{Q}
%\end{tikzcd}
%\end{center}

 We therefore can construct an open cover $\lbrace U_i\rbrace _{i\in I}$ of $T_xM$ and continuous maps $\phi_i:U_i\to \Sph u^*\mathcal{E}\subset u^*\mathcal{E}$, and lifts $\psi_i:U_i\to Tu^*\mathcal{E}$ such that $\mathrm{d}\pi(\psi_i(v))$ is positive and $\mathrm{d}\pr\left( \psi_i(v)\right)=v$ for any $i\in I$ and $v\in U_i$. The $U_i$ can be assumed invariant by scalar multiplication, and $\phi$ and $\psi$ satisfy the aforementioned homogeneity conditions. 

\medskip

%Indeed if $\psi_i(v)\in T_{\phi_i(v)}\mathcal{E}$ is associated to the pair $(\mathrm{d}u(v), (\phi_i(v), \dot{q}))$, then one can set $\phi_i(-v)=-\phi(v)$ and define $\psi_i(-v)\in T_{q'}\mathcal{E}$ where $q'$ is a representent of $\phi_i(-v)$ to be the element associated ton the pair $(\mathrm{d}u(-v), (-\phi_i(v), \dot{q}))$. This new element is well defined and still satisfies $\mathrm{d}\pi(\psi_i(-v))$ is positive since $\dot{q}$ is positive, and $\mathrm{d}\pr(\psi_i(-v))=\mathrm{d}u(-v)$.

We now glue these maps together. Let $\chi_i:U_i\to [0,1]$ for $i\in I$ be a family of functions that forms a locally finite partition of the unit. We define :
$$\phi:v\in \Sph T_xM \mapsto \sum_{i\in I}\chi_i(v)\phi_i(v)\in \left(u^*\mathcal{E}\right)_x.$$

Let us check that $\phi(v)$ is always non-zero. This is where we use that $\mathrm{d}\pr\left( \psi_i(v)\right)=v$ for $i\in I$, and  we also use that the fitting immersion $u$ is defining a smooth fibration of the cone $S^2 V^{>0}$. Let $\gamma:\R \to M$ be a curve such that $\gamma(0)=x$  and $\gamma'(0)=v$. For all $t\in \R$ the intersection $ u^\circ(\gamma(t))\cap S^2V^{>0}$ is a non-empty convex set, so we can construct a section $s:t\in \R \mapsto S^2V^{>0}$ of the fibration, \ie such that for all $t\in \R$, $s(t)\in u^\circ(\gamma(t))$. 

Let us fix $i\in I$. Let $q:\R\to \Sph u^*\mathcal{E}$ be such that $q'(0)=\psi(v)$, which implies that $p\circ q(t)$ is equal to $\gamma(t)$ at the first order around $t=0$ since $\gamma'(0)=\mathrm{d}\pr\left(\psi(v)\right)=v$. Since $s(t)\in u^\circ(\gamma(t))$, one has $ \pi(q(t))\cdot s(t)=0$ at the first order around $t=0$. Taking the first derivative at $t=0$ of this equation we get:
$$\pi(\phi_i(v))\cdot s'(0)+\mathrm{d}\pi(\psi_i(v) )\cdot s(0)=0.$$
 Since $\mathrm{d}\pi(\psi_i(v) )$ is positive and $s(0)$ is a positive tensor:
 $$\mathrm{d}\pi(\psi_i(v) )\cdot s(0)>0.$$

 Hence for all $i\in I$, $\pi(\phi_i(v))\cdot s'(0)<0$ and therefore $\pi(\phi(v))\cdot s'(0)<0$. In particular $\phi(v)$ does not vanish.

\medskip

In order to glue the $\psi_i$ we need to be careful since the vectors $\psi_i(v)$ do not belong to the same fiber of the tangent bundle $Tu^*\mathcal{E}$. Let $\Sigma$ be the following map:
$$\Sigma: (q_i)_{i\in I}\in \left(u^*\mathcal{E}\right)_x^I\mapsto \sum_{i\in I}q_i\in \left(u^*\mathcal{E}\right)_x.$$ 

Given $v\in \Sph T_xM$ we set:
$$\psi(v)=\mathrm{d}\Sigma\left(\left(\chi_i(v) \psi_i(v)\right)_{i\in I}\right).$$

This combination still satisfies that $\mathrm{d}\pi(\psi)$ is positive, indeed:
$$\mathrm{d}\pi(\psi(v))=\sum_{i\in I}\chi_i(v)\mathrm{d}\pi(\psi_i(v)).$$

Note that we also get the following:
$$\mathrm{d}\pr(\psi(v))=\sum_{i\in I}\chi_i(v)\mathrm{d}\pr(\psi_i(v))=\left(\sum_{i\in I}\chi_i(v)\right)v=v.$$

%\medskip
%
%It remains to check that $\psi(v)\neq 0$ for $v\in \Sph T_xM$, using the map $\psi$. Let $i$ be the immersion of $u^*\mathcal{E}$ into $ Tu^*\mathcal{E}$ where each fiber  $u^*\mathcal{E}_x$ is identified with its tangent space. We have the following short exact sequence:
%
%\begin{center}
%\begin{tikzcd}[row sep=2.5em]
% 0\arrow[r,"i"]& u^*\mathcal{E} \arrow[r]& Tu^*\mathcal{E}\arrow[r, "\mathrm{d}\pr"]& TM \arrow[r]&0
%\end{tikzcd}
%\end{center}
%
%Moreover for $q\in u^*\mathcal{E}$, $\pi(q)=\mathrm{d}\pi\circ i(q)$. Since $\mathrm{d}\pi(\psi(v))$ is positive and $u(x)$ does not contain any positive element, $\mathrm{d}\pr(\psi(v))$ does not vanish.
%
%\medskip

This concludes the construction of $\phi$ and $\psi$, and hence this concludes the proof.

\end{proof}

To prove Proposition \ref{prop:ExistenceFittingFlow} we use the directions $\mathrm{w}$ from Lemma \ref{lem:fitting flow sphere} and we glue these vectors into a vector field using one again a partition of the unit. We construct the vector field $W:\Sph u^*\mathcal{E}\to T\Sph u^*\mathcal{E}$ in a similar manner as $\psi:\Sph TM\to T\Sph u^*\mathcal{E}$. Morally "$W=\psi\circ \phi^{-1}$", but the map $\phi$ constructed previously is not a priori bijective.

\begin{proof}[{Proof of Proposition \ref{prop:ExistenceFittingFlow}}]

%We now do the opposite construction of Lemma \ref{lem:fitting flow sphere}.
 We construct a continuous vector field $\mathrm{W}$ over $u^*\mathcal{E}$, except the zero section that is homogeneous, \ie such that for all $x\in M$ and non zero $q\in u(x)$ one has $W_{\lambda q}=\mathrm{d}(m_\lambda)W_q$ for all $\lambda\in \R^{> 0}$, where $m_\lambda$ is the multiplication by $\lambda$ on $\mathcal{E}$.
Such a vector field defines a vector field $\overline{W}$ on $\Sph u^*\mathcal{E}$. 
We require moreover that $\mathrm{d}\pi(W)$ is always positive. 

\medskip

Given any non-zero $q_0\in u^*\mathcal{E}$, Lemma \ref{lem:fitting flow sphere} provides the existence of some $\mathrm{w}_0\in T_{(u(x),q)}\mathcal{E}$ for all $(x,q)\in u^*\mathcal{E}$ such that:
\begin{enumerate}
 \item[$(i)$] $\mathrm{d}\pr(\mathrm{w}_0)\in \mathrm{d}u(T_xM)$,
 \item[$(ii)$] $\mathrm{d}\pi(\mathrm{w}_0)$ is positive.
 \end{enumerate}
 
Property $(i)$ implies that $\mathrm{w}_0$ defines a vector in $T_{x,q_0}u^*\mathcal{E}$. For each such $q_0$ one can find a neighborhood $U$ of it in $u^*\mathcal{E}$ that is invariant by the $\R^{> 0}$-action, and on which one can define a map $W$ satisfying $(i)$ as well as the homogeneity condition $W_{\lambda q}=\mathrm{d}(m_\lambda)W_q$. Condition $(ii)$ being an open condition invariant by the action of $m_\lambda$ it is satisfied automatically for $U$ small enough.

\medskip

Using a partition of the unit as in Lemma \ref{lem:fitting flow sphere} we construct the desired vector field on $u^*\mathcal{E}$. Note that both property $(i)$ and the homogeneity are preserved by linear combinations, and $(ii)$ is preserved by positive combinations.

\medskip

Finally note that the collection $U_i$ as well as the partition of the unity can be chosen to be $\rho$-equivariant, so that the vector field $W$ is also $\rho$-equivariant, and hence also the fitting flow $\Phi$.

%Lemma \ref{prop:ExistenceFittingFlow} implies that one can find an open cover $\lbrace V_i\rbrace _{i\in I}$ of the compact space $\Sph u^*\mathcal{E}$ and a vector fields $\psi_i:V_i\to T
%u^*\mathcal{E}$ such that $\mathrm{d}\pi(\overline{\psi}_i(q))$ is positive. 
%
%\medskip
%
%Using a partition of unity, one can glue these vector fields to a vector field defined over $\Sph u^*\mathcal{E}$. The flow of this vector field defines a fitting flow over any compact set $K$. When $M=\widetilde{N}$ and $u$ is equivariant, the vector field can be made $\rho$-equivariant so the flow can be made $\rho$-equivariant.
%
%\wip

\end{proof}

\section{The Anosov property and fibrations.}
\label{sec:The Anosov property and fibrations.}

In this section we show that the existence of an equivariant map of pencils that admits a fitting flow implies that the representation is Anosov. Moreover we describe the domain that is fibered in $\mathbb{RP}^{2n-1}$. Finally we apply this to show that some quasi-Fuchsian representations do not admit equivariant maps that admit a fitting flow.

\subsection{The Anosov property}
In order to show that a uniform contraction is taking place along the flow lines of the fitting flow, we define a way to measure the distance between two quadric hypersurfaces nested into one another. The characterization of Anosov representations that we use is similar to the characterization in terms of inclusion of multicones from \cite{BPS}.

\medskip

Let $\Sph\mathcal{Q}^{{\mix}}$ be the set of quadrics that are not semi-positive or semi-negative up to a positive scalar.

\medskip

 If we fix an affine chart of $L$ so that $\ell_1,\ell_2,\ell_2',\ell_1'$ correspond to the real numbers $x_1'< x_2'\leq x_2< x_1$, the cross ratio $[\ell_1,\ell_2,\ell_2',\ell_1']$ is defined as:
\begin{equation}\label{eq:crossratio}[\ell_1,\ell_2,\ell_2',\ell_1']=\frac{x_2-x_1'}{x_2-x_2'}\times\frac{x_1-x_2'}{x_1-x_1'} .
\end{equation}

An other equivalent definition, if we take non-zero representatives $v_1,v_2,v_2',v_1'$ of $\ell_1,\ell_2,\ell_2',\ell_1'$:
$$[\ell_1,\ell_2,\ell_2',\ell_1']=\frac{v_2\wedge v_1'}{v_2\wedge v_2'}\times\frac{v_1\wedge v_2'}{v_1\wedge v_1'} .$$
This expression makes sense as $\bigwedge^2 L\simeq \R$, and it is independent of the chosen identification. The equivalence betweens these formulations is due to the fact that if $v_1=e_1+x_1e_2$ and $v_2=e_1+x_2e_2$ then $v_2\wedge v_1=(x_2-x_1)e_2\wedge e_1$.

\begin{defnProp}
\label{defn:Cross ratio distance}

Let $[q_1],[q_2]\in \Sph\mathcal{Q}^{\mix}$ be such that for some choice of representatives, the difference $q_2-q_1$ is positive. For every projective line $L$ intersecting $\lbrace q_1> 0\rbrace$ and $\lbrace q_2< 0\rbrace$  we can define $\ell_1$ and $\ell_1'$ are the zeros of $q_1$ on $L$ and $\ell'_1$ and $\ell_2'$ and the zeros of $q_2$ on $L$ so that the points $\ell_1,\ell_2,\ell_2',\ell_1'\in\mathbb{P}(V)$ are cyclically in this order on $L$. We define the  \emph{cross-ratio distance} $\text{cr}([q_2],[q_1])\in [1, \infty)$ between $[q_2]$ and $[q_1]$ as the minimum of the cross ratio $[\ell_1,\ell_2,\ell_2',\ell_1']$ for any such projective line $L$, which is well defined and in $(1,\infty)$.
\end{defnProp}

We now verify that this minimum exists and is in $(1,\infty)$.
\begin{proof}
First note that the assumption $q_2-q_1$ is positive implies that $\lbrace q_1\geq 0 \rbrace\subset\lbrace q_2>0\rbrace$ and $\lbrace q_2\leq 0 \rbrace\subset\lbrace q_1<0\rbrace$. If $L$ intersects $\lbrace q_1>0\rbrace$ and $\lbrace q_1< 0 \rbrace$ then  the quadratic forms $q_1$ and $q_2$ are both positive on some points and negative on some points of $L$, and hence they both admit exactly two zeroes on $L$. These zeroes, denoted respectively $\ell_1$,$\ell_1'$ for $q_1$ and $\ell_2$,$\ell_2'$ for $q_2$ are in the following cyclic order on L up to exchanging $\ell_1$ and $\ell_1'$: $(\ell_1,\ell_2,\ell_2',\ell_1')$. One can refer to Figure \ref{fig:Anosov 1} where $\lbrace q_i>0\rbrace$ is the interior of the corresponding ellipse.

\medskip

Consider the compact space $X$ of projective lines $L$ intersecting $\lbrace q_1\geq  0\rbrace$ and $\lbrace q_2\leq 0\rbrace$. For every such line we can define the lines $\ell_1,\ell_1'$ and $\ell_2, \ell_2'$ as in previously where $\ell_1,\ell_1'\neq \ell_2,\ell_2'$ but eventually $\ell_1=\ell_1'$ or $\ell_2=\ell_2'$ which occur respectively exactly if $L$ intersects trivially $\lbrace q_1>0\rbrace$ or $\lbrace q_2>0\rbrace$. Indeed in these cases the quadratic forms $q_i$ are semi-positive or semi negative on $L$ and hence have a double zero.
By looking at the formula \eqref{eq:crossratio} on an affine chart such that $x_1'< x_2'\leq x_2< x_1$ we see the cross ratio $[\ell_1,\ell_2,\ell_2',\ell_1']$ is in $(1,\infty]$. This ratio is finite if and only if $\ell_1\neq\ell_1'$ and $\ell_2\neq\ell_2'$ which occurs exactly when $L$ belongs to the interior of $X$, \ie when $L$ intersects $\lbrace q_1> 0\rbrace$ and $\lbrace q_2< 0\rbrace$. Finally note that this ratio depends continuously on $L\in X$  and that the interior of $X$ is non-empty since $q_1$ and $q_2$ are mixed. Therefore the minimum on $X$ is well defined in $ (1,\infty)$ and is reached on the interior of $X$.
\end{proof}

%$$ \frac{\overline{\ell_1}\wedge \overline{\ell_2}}{\overline{\ell_1}\wedge \overline{\ell_1'}}\times \frac{\overline{\ell_1}\wedge \overline{\ell_1'}}{\overline{\ell_2'}\wedge \overline{\ell_2}}.$$

 The logarithm of this quantity satisfies a reversed triangular inequality.

\begin{prop}

\label{prop:TriangIneqCr}
Let $[q_1],[q_2],[q_3]\in \Sph\mathcal{Q}^{\mix}$ be such that $q_3-q_2$ and $q_2-q_1$ are positive. Then :
$$\text{cr}([q_3],[q_1])\geq \text{cr}([q_3],[q_2])\text{cr}([q_2],[q_1]) .$$
\end{prop}

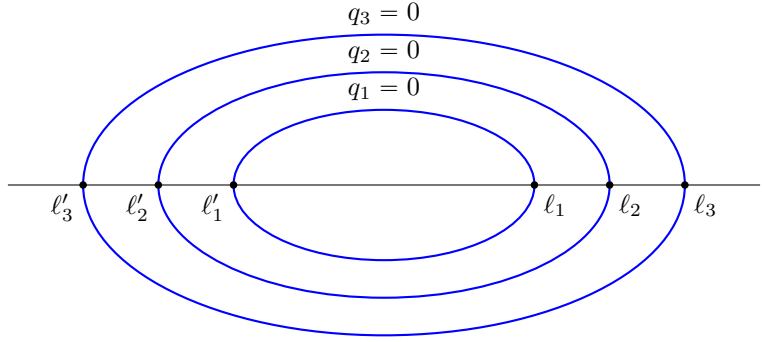
\begin{figure}
\begin{center}
\begin{tikzpicture}

\draw[thick, blue] (0,0) ellipse (2 and 1);
\draw[thick, blue] (0,0) ellipse (3 and 1.5);
\draw[thick, blue] (0,0) ellipse (4 and 2);
\draw (-5,0) -- (5,0);

\node[above] () at (0,1) {$q_1=0$};
\node[above] () at (0,1.5) {$q_2=0$};
\node[above] () at (0,2) {$q_3=0$};

\fill (2,0) circle(.05);
\fill (3,0) circle(.05);
\fill (4,0) circle(.05);
\fill (-2,0) circle(.05);
\fill (-3,0) circle(.05);
\fill (-4,0) circle(.05);
\node[below right] () at (2,0) {$\ell_1$};
\node[below right] () at (3,0) {$\ell_2$};
\node[below right] () at (4,0) {$\ell_3$};
\node[below left] () at (-2,0) {$\ell_1'$};
\node[below left] () at (-3,0) {$\ell_2'$};
\node[below left] () at (-4,0) {$\ell_3'$};

\end{tikzpicture}
\end{center}
\caption{Illustration of Proposition \ref{prop:TriangIneqCr}}
\label{fig:Anosov 1}
\end{figure}

We illustrate this proposition and its proof in Figure \ref{fig:Anosov 1}. This figure illustrates 3 quadrics of signature $(1,2)$ in $\mathbb{RP}^2$. In these pictures the quadrics are positive on the inside of the ellipse they define.

\begin{proof}
Pick any projective line that crosses $\lbrace q_1>0\rbrace$ and $\lbrace q_3<0\rbrace$. Let the intersections of $L$ with the zeroes of $q_1$, $q_2$ and $q_3$ be respectively $(\ell_1,\ell_2,\ell_3,\ell_3',\ell_2',\ell_1')$, counted with multiplicity and cyclically ordered. We fix an affine chart for $L$ such that this tuple corresponds to the tuple $x_3'< x_2'<x_1'\leq x_1<x_2<x_3$ of real numbers. This yields the following :
$$[\ell_1,\ell_3,\ell_3',\ell_1']=\frac{x_3-x_1'}{x_3-x_3'}\times\frac{x_1-x_3'}{x_1-x_1'} .$$
$$[\ell_1,\ell_3,\ell_3',\ell_1']= \left( \frac{x_3-x_1'}{x_3-x_2'}\times\frac{x_1-x_2'}{x_1-x_1'}\right)\times \left(\frac{x_3-x_2'}{x_3-x_3'}\times\frac{x_1-x_3'}{x_1-x_2'}\right) .$$

Note that if $B>A>0$, for all $c>0$ one has  $\frac{A}{B}\leq \frac{A+c}{B+c}$ and respectively $\frac{B+c}{A+c}\leq \frac{B}{A}$. Hence, due to the ordering of the tuple, we have the following inequalities by setting respectively $A=x_2-x_1'$, $B=x_2-x_2'$, $c=x_3-x_2$ and $A=x_1-x_2'$, $B=x_1-x_3'$, $c=x_2-x_1$:
$$\frac{x_2-x_1'}{x_2-x_2'}\leq \frac{x_3-x_1'}{x_3-x_2'}\;,\;\frac{x_2-x_3'}{x_2-x_2'} \leq \frac{x_1-x_3'}{x_1-x_2'}.$$

Hence we obtain:
$$\text{cr}([q_2],[q_1])\leq [\ell_1,\ell_2,\ell_2',\ell_1']=\frac{x_2-x_1'}{x_2-x_2'}\times\frac{x_1-x_2'}{x_1-x_1'} \leq \frac{x_3-x_1'}{x_3-x_2'}\times\frac{x_1-x_2'}{x_1-x_1'} .$$
$$\text{cr}([q_3],[q_2])\leq [\ell_2,\ell_3,\ell_3',\ell_2']=\frac{x_3-x_2'}{x_3-x_3'}\times\frac{x_2-x_3'}{x_2-x_2'} \leq \frac{x_3-x_2'}{x_3-x_3'}\times\frac{x_1-x_3'}{x_1-x_2'}.$$

Hence one has $[\ell_1,\ell_3,\ell_3',\ell_1']\geq \text{cr}([q_2],[q_1])\text{cr}([q_3],[q_2])$ for every such projective line $L$.

Therefore $\text{cr}([q_3],[q_1])\geq \text{cr}([q_3],[q_2])\text{cr}([q_2],[q_1])$.
\end{proof}

A sequence of quadrics such that the cross ratio distance between the first and last quadric goes to $+\infty$ satisfies that the intersection of all half-spaces determined by the quadrics is a projective subspace.

\begin{prop}
\label{prop:Intersection linear}
Let $(q_n)_{n\in \mathbb{N}}$ be a sequence of quadrics such that $q_{n+1}-q_n$ is positive for all $n\in \mathbb{N}$ and $\text{cr}([q_{n}],[q_0])$ goes to $+\infty$. Then $\bigcap_{n\in \N} \lbrace q_n\leq 0\rbrace \subset \mathbb{P}(V)$ is a projective subspace.

\end{prop}

This proposition is proven in \cite{BG}. We write a version of the argument here for the sake of completeness. We illustrate this proposition with quadrics of signature $(1,2)$ in $\mathbb{RP}^2$ in Figure \ref{fig:Anosov 2}. 

\begin{figure}
\begin{center}
\begin{tikzpicture}

\draw[thick, blue] (0,0) ellipse (1 and 0.75);
\draw[thick, blue] (0,0) ellipse (2 and 1.5);
\draw (-4,0) -- (4,0);

\node[above] () at (0,0.75) {$q_{n_0}=0$};
\node[above] () at (0,1.5) {$q_n=0$};

\fill (4,0) circle(.05);
\fill (-4,0) circle(.05);
\fill (0,0) circle(.05);
\fill (2,0) circle(.05);
\fill (-2,0) circle(.05);
\fill (1,0) circle(.05);
\fill (-1,0) circle(.05);
\node[right] () at (4,0) {$y$};
\node[left] () at (-4,0) {$x$};
\node[below right] () at (2,0) {$z^+_n$};
\node[below left] () at (-2,0) {$z^-_n$};
\node[below right] () at (1,0) {$z^+_{n_0}$};
\node[below left] () at (-1,0) {$z^-_{n_0}$};
\node[below] () at (0,0) {$z$};

\end{tikzpicture}
\end{center}
\caption{Illustration of Proposition \ref{prop:Intersection linear}}
\label{fig:Anosov 2}
\end{figure}
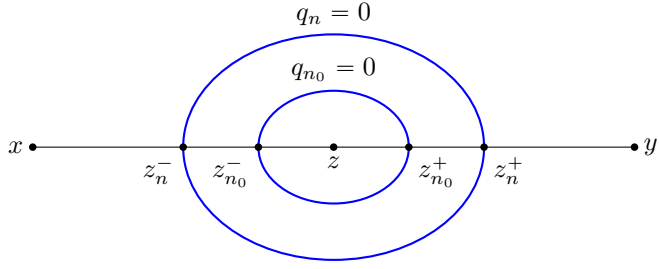

\begin{proof}
The intersection $I=\bigcap_{n\in \N} \lbrace q_n\leq 0\rbrace$ is a  compact non-empty subset. Let $x\neq y\in I$ and let $L$ be the projective line from $x$ to $y$. Suppose that there exist $z\in L$ such that $z\notin I$. Without any loss of generality one can assume that the open interval $S\subset L$ bounded by $x,y$ and containing $z$ does not intersect $I$. Indeed one can otherwise replace $x,y$ by the points on $L\cap I$ closest to $z$ on both 
sides.

\medskip

Since $z\notin I$ there exist $n_0\in \N$ such that $q_{n_0}$ is positive on $z$. For $n\geq n_0$, let $z_1,z_2$ be the two intersections of $L$ with the zeroes of $q_n$ so that the points $(x,z_n^-,z,z_n^+,y)$ are cyclically ordered. The sequences $(z_n^-)$ and $(z_n^+)$ are monotonic in $S$ and must converge to $x$ and $y$ since $S\cap I=\emptyset$. 
The value of $\text{cr}([q_{n_0}],[q_n])$ is bounded from above by the cross ratio $[z_{n_0}^+,z_n^+,z_n^-,z_{n_0}^-]$, which is turns converges to the cross ratio $[z_{n_0}^+,x,y,z_{n_0}^-]<\infty$ when $n$ goes to $+\infty$.
This contradicts the fact that $\text{cr}([q_{n}],[q_0])$ goes to $+\infty$.
Hence for every pair of points in $I$, the associated projective line is contained in $I$. In particular $I$ is a projective subspace.
\end{proof}

Let $\Gra_d^{(n,n)}(\mathcal{Q})\subset \Gra_d^\text{mix}(\mathcal{Q})$ be the open subset of $\Gra_d(\mathcal{Q})$ of pencils $P$ such that for all $q\in P$ non-zero, $q$ has signature $(n,n)$. We now apply the previous results to prove that representations that admit an equivariant fitting immersion in this open set, and more generally a fitting map with a fitting flow, are Anosov. For examples of such immersions, see Propositon \ref{prop:GaussMapSymmetric space for totally geodesic immersions}. Here again $M=\widetilde{N}$ for a compact $N$, $\Gamma=\pi_1(N)$ and $\rho:\Gamma\to \SL(V)$ is a representation.

%\begin{lem}
%Let $[q_1],[q_2]\in \Sph(\mathcal{Q}_V)$ be of signature $(p,q)$ and such that $q_1-q_2$ is positive. Let $g_1,g_2\in \mathcal{Q}^*_V$ be positive and such that $\tr(g_1q_1)=\tr(g_2q_2)=0$. Then the ratio of the $p$ and $p+1$ largest singular values of $q_1^{-1}q_2$ is greater than $\text{cr}([q_1],[q_2])$.
%\end{lem}

%Let $E\oplus F\subset V$ be a decomposition of $V$ with $\dim(E)=p$. Let $g\in \mathcal{Q}_V$ be positive such that $E$ and $F$ are orthogonal with respect to $g$.  the \emph{standard quadric} $q(E,F,g)$ is the symmetric bilinear form whose evaluation at $(v,v)$ for $v=v_e+v_f\in E\oplus F$ equals $g(e,e)-g(f,f)$. 
%
%\begin{lem} Suppose that for some  positive $g_1,g_2$ one has $q(E,F,g_2)-q(E,F,g_1)$ positive. The ratio $\frac{\sigma_p}{\sigma_{p+1}}$ where $\sigma_1\geq\cdots \geq \sigma_n$ are the eigenvalues of $g_2g_1^{-1}$ is greater than $\text{cr}([q(E,F,g_2)],[q(E,F,g_1)])$.
%\end{lem}
%
%\begin{proof}
%
%Let $v\in E$ and $w\in F$ be non-zero vectors. 
%
%\wip
%
%\end{proof}

\begin{thm}
\label{thm:WellFittedImplesAnosov}
Let $u:\widetilde{N}\to \Gra_d^{(n,n)}(\mathcal{Q})$ be a continuous $\rho$-equivariant map that admits an equivariant fitting flow. The representation $\rho$ is $\lbrace n\rbrace$-Anosov, and for any $q\in \Sph u^*\mathcal{E}$, the limit map $\xi^n_\rho(\zeta)\in \Gra_n(\R^{2n})$ at the limit $\zeta\in \partial \Gamma$ of the the flow line $(\Phi_t(q))$ is the following projective subspace:

$$\bigcap_{t\geq 0} \lbrace \pi(\Phi_t(q))\leq 0\rbrace.$$
\end{thm}

Note that the flow lines of an equivariant fitting flow are quasi-geodesics  by Proposition \ref{prop:Flow lines of a fitting flow are quasi geodesic and exist between points}, and $\mathbb{S}u^*\mathcal{E}$ is quasi-isometric to $\Gamma$, hence the limit point $\zeta$ is well-defined.

\begin{figure}
\begin{center}
\begin{tikzpicture}

\draw[thick, dashed, blue] (0.5,0) ellipse (2 and 1);
\draw[thick, red] (-0.5,0) ellipse (2 and 1);

\draw[thick, dashed, blue] (0.5,0) ellipse (5 and 2);
\draw[thick, red] (-0.5,0) ellipse (5 and 2);

\draw (-6,0) -- (6,0);

\node () at (2,1) {$\bar{h}_q=0$};
\node () at (-2,1) {$\pi(q)=0$};
\node () at (4,2) {$\bar{h}_{\Phi_t(q)}=0$};
\node () at (-4,2) {$\pi(\Phi_t(q))=0$};

\fill (5.5,0) circle(.05);
\fill (-5.5,0) circle(.05);
\fill (4.5,0) circle(.05);
\fill (-4.5,0) circle(.05);
\fill (-2.5,0) circle(.05);
\fill (2.5,0) circle(.05);
\fill (-1.5,0) circle(.05);
\fill (1.5,0) circle(.05);
\fill (0,0) circle(.05);
%\fill (6.5,0) circle(.05);

\node[below] () at (0,0) {$e$};
%\node[below] () at (6.5,0) {$f$};
\node[above] () at (-5.75,0) {$L$};
\node[below left] () at (-2.5,0) {$\ell_1'$};
\node[below right] () at (1.5,0) {$\ell_1$};
\node[below left] () at (-1.5,0) {$\ell_2'$};
\node[below right] () at (2.5,0) {$\ell_2$};
\node[below left] () at (-5.5,0) {$\ell_3'$};
\node[below right] () at (4.5,0) {$\ell_3$};
\node[below left] () at (-4.5,0) {$\ell_4'$};
\node[below right] () at (5.5,0) {$\ell_4$};

\node[below] () at (3,-2) {$f\text{ at }\infty\text{ on  }L \text{ } \to$};

\end{tikzpicture}
\end{center}
\caption{Illustration of the proof of Theorem \ref{thm:WellFittedImplesAnosov}.}
\label{fig: Anosov 3}
\end{figure}
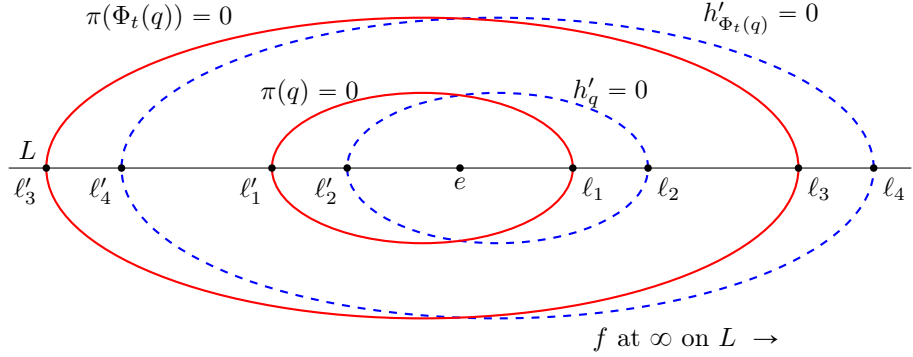

\begin{proof}

The flat $V$-bundle over $\Sph u^*\mathcal{E}$ associated with $\rho$ admits a continuous splitting $E\oplus F$ where for $q\in \Sph u^*\mathcal{E}$:
$$\mathbb{P}\left(E_q\right) =\bigcap_{t\geq 0} \lbrace \pi(\Phi_{-t}(q))\geq 0\rbrace,$$
$$\mathbb{P}\left(F_q\right) =\bigcap_{t\geq 0} \lbrace \pi(\Phi_{t}(q))\leq 0\rbrace.$$

This defines transverse vector subspaces by Proposition \ref{prop:Intersection linear} since $\text{cr}(\Phi_t(q),q)$ and $\text{cr}(-\Phi_{-t}(q),-q)$ go to $+\infty$ when $t$ goes to $+\infty$.
Moreover the quadrics in the pencils in the image of $u$ are of signature $(n,n)$, one must have $\dim(E_q)=\dim(F_q)=n$ so this splitting is well-defined.
This splitting is preserved by $\Phi$.

\medskip

We now construct a metric $h$ on this flat $V$-bundle over $\Sph u^*\mathcal{E}$. Given $q\in \Sph u^*\mathcal{E}$ we define the symmetric bilinear form $h_q$ on $V=E_q\oplus F_q$ so that this sum is orthogonal and $h_q$ is equal to $\pi(q)\in \mathcal{Q}$ on $E_q$ and $-\pi(q)\in \mathcal{Q}$ on $F_q$. Note that by definition of $E_q$ and $F_q$, $h_q$ is positive.

We also introduce an auxiliary symmetric bilinear form $\bar{h}$ of signature $(n,n)$ on this flat $V$-bundle over $\Sph u^*\mathcal{E}$ so that the sum $E_q\oplus F_q$ is orthogonal and $\bar{h}_q$ is equal to $\pi(q)\in \mathcal{Q}$ on $E_q$ and on $F_q$. 

\medskip

Our first step is to compare the quadric $\bar{h}_q$ with $\pi(q)$. Let $L$ be a projective line intersecting $\mathbb{P}(E_q)$ at some $e$ and $\mathbb{P}(F_q)$ at some $f$. In Figure \ref{fig: Anosov 3}, we illustrate some $\mathbb{RP}^2\subset \mathbb{RP}^{2n-1}$ containing the projective line $L$. Let $\ell_1,\ell'_1$ be the zeroes of $\pi(q)$ on $L$ and $\ell_2,\ell'_2$ be the zeroes of $\bar{h}_q$ on $L$, so that $\ell_1,\ell_2$ lie in the same connected component of $L\setminus \lbrace e,f\rbrace$. Since $N$ is compact, there exist a maximum $\delta<\infty$ for all $q\in \Sph u^*\mathcal{E}$ and all such projective line $L$ of the quantity $|\log\left([\ell_1,\ell_2,f,e]\right)|$.

\medskip

Now we turn our attention to the contraction properties of $\Phi$. Let $t>0$ be a real number and let $q\in \Sph u^*\mathcal{E}$. Let $v\in E_q$ and $w\in F_q$. We are interested in the following ratio :
$$R=\frac{h_{\Phi_t(q)}(v)h_q(w)}{h_{\Phi_t(q)}(w)h_q(v)}.$$
Let $e,f$ be the lines generated by $v,w$ and $L$ be the projective line joining them. Let $\ell_1,\ell'_1$ be the zeroes of $\pi(q)$ on $L$, $\ell_2,\ell'_2$ the zeroes of $\bar{h}_q$ on $L$, $\ell_3,\ell'_3$ be the zeroes of $\pi(\Phi_t(q))$ on $L$ and finally $\ell_4,\ell'_4$ the zeroes of $h_{\Phi_t(q)}$ on $L$.
We assume that $\ell_1,\ell_2,\ell_3,\ell_4$ all lie on the same component of $L\setminus \lbrace e,f\rbrace$.

The cross ratio $[\ell_2,\ell_4,f,e]$ is equal to $R^{\frac{1}{2}}$. Indeed $\ell_2$ is generated by $h^{\frac{1}{2}}_q(w)v+h^{\frac{1}{2}}_q(v)w$ and $\ell_4$ is generated by $h^{\frac{1}{2}}_{\Phi_t(q)}(w)v+h^{\frac{1}{2}}_{\Phi_t(q)}(v)w$, up to changing $\ell_i$ by $\ell_i'$ for $1\leq i\leq 4$. Hence:
$$[\ell_2,\ell_4,f,e]=\frac{\left( h^{\frac{1}{2}}_{q}(w)v+h^{\frac{1}{2}}_{q}(v)w\right)\wedge w}{\left( h^{\frac{1}{2}}_{q}(w)v+h^{\frac{1}{2}}_{q}(v)w\right)\wedge v}\times \frac{\left( h^{\frac{1}{2}}_{\Phi_t(q)}(w)v+h^{\frac{1}{2}}_{\Phi_t(q)}(v)w\right)\wedge v}{\left( h^{\frac{1}{2}}_{\Phi_t(q)}(w)v+h^{\frac{1}{2}}_{\Phi_t(q)}(v)w\right)\wedge w}=R^{\frac{1}{2}}.$$
However due to our comparison of $\pi(q)$ and $\bar{h}_q$ one has:
$$e^{-\delta}\leq [\ell_2,\ell_4,f,e],[\ell_1,\ell_3,f,e]\leq e^{\delta}.$$
Therefore $[\ell_2,\ell_4,f,e]/[\ell_1,\ell_3,f,e]\geq e^{-2\delta}$.
Hence $R^{\frac{1}{2}}>e^{-2\delta}[\ell_1,\ell_3,f,e]$. This last cross ratio is larger than:
$$[\ell_1,\ell_3,\ell_3',\ell_1']\geq \text{cr}\left(\pi(\Phi_t(q)),\pi(q),\right).$$

Since $\Phi$ is a fitting flow and since $N$ is compact, there exist $\alpha>0$ such that $\text{cr}\left(\pi(\Phi_1(q)),\pi(q)\right)\geq e^\alpha$ for all $q\in \Sph u^*\mathcal{E}$. Hence by the triangular inequality from Proposition \ref{prop:TriangIneqCr}, for all $t>0$:
$$\text{cr}\left(\pi(\Phi_t(q)),\pi(q)\right)\geq e^{\alpha (t-1)}.$$

Hence we get the following estimate:
$$\frac{h_{\Phi_t(q)}(v)h_q(w)}{h_{\Phi_t(q)}(w)h_q(v)}\geq e^{2\alpha t-4\delta-2\alpha} .$$

This implies that the splitting $V=E_q\oplus F_q$ is $\lbrace n\rbrace $-contracting in the sense of \cite{BPS} with respect to the flow $\Phi$ for the metric $h$. Moreover $\Gamma$ acts cocompactly on $\Sph u^*\mathcal{E}$ and every geodesic in $\Gamma$ is at uniform distance from a flow line of $\Phi$. The domination of this splitting implies an exponential gap for the singular values \cite[{Theorem 2.2}]{BPS}, which implies that $\rho$ is $\lbrace n\rbrace $-Anosov. The vector subspace $\xi_\rho^n(\zeta)$ is the contracted subspace $F_q$.

\end{proof}

%Suppose that $\rho$ factors through a representation $\eta:G\to\GL(V)$ so that for some set of simple roots $\Theta$. We define $\mathcal{Q}_V^\Theta$ the space of quadrics of signature $(n,n)$ that are positive and respectively negative on some $n$-dimensional space $E_+,E_-$ such that $G\cap \text{Stab}(E_+)$ and $G\cap \text{Stab}(E_-)$ are conjugate to $P_\Theta$.
%
%
%\begin{thm}
%Let $u:\widetilde{N}\to \Gra_d(\mathcal{Q}^{\Theta}_V)$ be a fitting $\rho$-equivariant immersion. Then $\rho$ is $\Theta$-Anosov.
%\end{thm}

\subsection{Fibered domain of discontinuity}

Such an equivariant map into the space of pencils that admits a fitting flow induces a fibration of the Guichard-Wienhard domain of discontinuity.

\medskip

Recall that $u^o(x)\subset S^2V$ is the anihilator subspace of $u(x)\subset \mathcal{Q}=\left(S^2V\right)^*$.

\begin{prop}
\label{prop:Completeness and fibration of DOD}
Let $N$ be a compact manifold of dimension $d$ with fundamental group $\Gamma$. Let $\rho:\Gamma\to \SL(V)$ and $u:\widetilde{N}\to \Gra^{(n,n)}_d(\mathcal{Q})$ be a $\rho$-equivariant continuous map that admits an equivariant fitting flow $\Phi$ on $\Sph u^*\mathcal{E}$. 
The union of $\mathbb{P}(u^\circ(x)\cap S^2V^{>0})$ for $x\in \widetilde{N}$ cover all of $\mathbb{P}(S^2V^{>0})$, and the closure of this union intersects the space of rank one points $\mathbb{P}(V)\simeq S^2\mathbb{P}(V)\subset \mathbb{P}(S^2V^{\geq 0})$ exactly at the domain of discontinuity for $\lbrace n\rbrace$-Anosov representations considered by Guichard-Wienhard \cite{GWDoD}:

\begin{equation}
\label{eq:GWDOD}
\Omega=\mathbb{P}(V)\setminus \bigcup_{\zeta\in \partial \Gamma} \mathbb{P}(\xi^n_\rho(\zeta)).
\end{equation}

The intersection of $\mathbb{P}(u^\circ(x)\cap S^2V^{>0})$ with the set of rank one points for $x\in \widetilde{N}$ defines a fibration over $\widetilde{N}$ of $\Omega$.
\end{prop}

In this argument we will use the Hilbert distance on $\mathbb{P}(S^2V^{>0})$ already introduced for the proof  of Proposition \ref{prop:Flow lines of a fitting flow are quasi geodesic and exist between points}.
For a subset $A\subset \mathbb{P}(V)$ we write $S^2 A\subset \mathbb{P}(S^2V)$ the corresponding set of rank one lines. 

Proving first that $\mathbb{P}(S^2V^{>0})$ is fully covered helps us proving that the Guichard -Wienhard domain is also fully covered.

\begin{proof}
Let us fist prove that all of $\mathbb{P}(S^2V^{>0})$ is covered by the union of $\mathbb{P}(u^\circ(x)\cap S^2V^{>0})$ for $x\in \widetilde{N}$. We fix a Riemanian metric on $N$ that defines a Riemannian metric $g$ on $\widetilde{N}$, with associated distance $d_g$. Since $u$ is globally fitting, see Remark \ref{rem:Fitting flow implies globally fitting}, for all $x\in \widetilde{N}$ there exist a neighborhood $U$ of $\mathbb{P}(u^\circ(x)\cap S^2V^{\geq 0})$ in $\mathbb{P}(S^2V^{\geq 0})$ that is covered by the manifolds $\mathbb{P}(u^\circ(y)\cap S^2V^{\geq 0})$ for $y$ in the ball for $d_g$ of radius $1$ centered at $0$, by Proposition \ref{prop:fitting immersions 3 items}, part (i).

This neighborhood contains an $\epsilon$-neighborhood of $\mathbb{P}(u^\circ(x)\cap S^2V^{> 0})$ for the Hilbert metric on $\mathbb{P}(S^2V^{>0})$ for some $\epsilon>0$. Since $N$ is compact, this $\epsilon>0$ can be chosen independently of $x$.

\medskip

Now let us fix some $x_0\in \widetilde{N}$ and some $[s_0]\in \mathbb{P}(u^\circ(x_0)\cap S^2V^{> 0})$. Given any $[s']\in \mathbb{P}(S^2V^{> 0})$ one can find a finite sequence $s_0,s_1,\cdots, s_k=s'$  in $S^2V^{>0}$ so that the Hilbert distance between $[s_i]$ and $[s_{i+1}]$ is less than $\epsilon$ for all $0\leq i< k$. By induction, and since the Riemannian metric $d_g$ is complete, one can construct for all $1\leq i\leq k$ a point $x_i\in \widetilde{N}$ such that $[s_i]\in \mathbb{P}(u^\circ(x_i)\cap S^2V^{> 0})$  and $d_H(x_{i-1},x_i)\leq 1$. Therefore the manifolds $\mathbb{P}(u^\circ(x)\cap S^2V^{> 0})$ cover all of $\mathbb{P}(S^2V^{>0})$.

\medskip

Now let us consider the fibered domain in projective space. We first prove that the fibers are contained in the domain $S^2\Omega$.
Consider a rank one line $[s]\in S^2\mathbb{P}(\xi^n_\rho(x))$ for some $\zeta\in \partial \Gamma$. Suppose that $p$ belongs to $u^\circ(x)$ for some $x\in \widetilde{N}$.
There exist a flow line $(\Phi_t([q]))_{t\geq 0}$ starting at $\pr([q])=x$ and converging to $\zeta \in \partial \Gamma \simeq \partial \widetilde{N}$ by Proposition \ref{prop:Flow lines of a fitting flow are quasi geodesic and exist between points}.
Theorem \ref{thm:WellFittedImplesAnosov} implies that $\pi(q)\in\mathcal{Q}$ must be negative on $\xi^2_\rho(\zeta)$ and hence $p$ cannot belong to $u^\circ(x)$. 

Conversely let us prove that every point in $S^2\Omega$ is covered by some fiber. Fix a point $x\in \widetilde{N}$ and take any rank one point $[s]\in S^2\Omega$ in the Guichard-Wienhard domain of discontinuity. There exist a sequence $(x_n)$ such that $[s]$ belongs to the limit of $\mathbb{P}(u^\circ(x_n)\cap S^2V^{\geq 0})$, since these manifolds cover $\mathbb{P}(S^2V^{>0})$. We consider some $[q_n]\in \Sph\mathcal{E}_x$ such that $\Phi_t([q_n])\in \Sph\mathcal{E}_{x_n}$ for some $t_n>0$, which exist by Proposition \ref{prop:Flow lines of a fitting flow are quasi geodesic and exist between points}. If $t_n$ diverges when $n$ varies, then the set $\mathbb{P}(u^\circ(x_n)\cap S^2\mathbb{P}(V))$ becomes arbitrarily close to $S^2\mathbb{P}(\xi^n_\rho(\zeta_n))$ where $\zeta_n$ is the limit when $t$ goes to $+\infty$ of $\Phi_t([q_n])$. This would contradict the fact that $[s]\in S^2\Omega$, as in this case $[s]\in \mathbb{P}(\xi_\rho^n(\zeta))$ where $\zeta$ is a limit point of $(\zeta_n)$. Hence the sequence $(x_n)$ is bounded and therefore converges up to subsequence to some $x_\infty\in \widetilde{N}$, and $[s]\in \mathbb{P}(u^\circ(x_\infty))$.

\end{proof}

\section{Fitting maps and maximal representations.}
\label{sec:Characterization of the structures.}
Let us consider representations $\rho:\Gamma\to \Sp(2n,\R)$. We prove our main result, which is the characterization of maximal representations in terms of maps of pencils.  The first part introduces $\omega$-regular pencils, as well as a connected component of the space of $\omega$-regular pencils. We then state the characterization, and then present the construction of a map of pencils with a fitting flow for any maximal representation. This construction relies to a map from the space of pairs of Lagragians to the space of quadrics. %We briefly discuss how one can also use this map to decompose the projective structure into polygons of quadric hypersurfaces in Section \ref{subsec:polygons of hyperboloids}.

\medskip

Throughout this section we set $d=2$, and consider the case when $N=S_g$ is a surface.

%\subsection{Characterization as foliations.}
%
%\begin{defn}
%A contact projective structure of dimension $n$ on a compact manifold is \emph{foliated by contact hyperboloids} if its tangent bundle admits a fibration such that each leaf of the foliation is locally mapped via the differential of the developing map to the tangent bundle of a contact hyperboloid without ends.
%\end{defn}
%
%Each leaf inherits a structure modeled on a contact hyperboloid (considered withpout its ends) $\hyperb$, with structure groups the stabilizer $$H\subset\PSp(2n,\R)$ of $\hyperb$ which is isomorphic to $PGL(n,\R)$. We say a $(H,\hyperb)$-structure is \emph{complete} if the corresponding developing map is surjective onto $\hyperb$.
%
%\medskip
%
%We say that a contact projective structure foliated by contact hyperboloids is \emph{complete} if the induced structure on every of its leaves is complete.
%
%\begin{lem}
%If $M$ is compact, every contact projective structure on $M$ foliated by contact hyperboloids is \emph{complete}.
%\end{lem} 
%
%\begin{lem}% True ??
%
%Let $M$ be a connected manifold with a contact projective structure foliated by contact hyperboloids and let $n\geq 3$. Every $2$ points in $M$ belong to a common leaf.
%
%\end{lem}
%
%\begin{prop}
%Let $M$ be a compact manifold with a contact projective structure foliated by contact hyperboloids. The uniform cover of each leaf is mapped bijectively via the developing map onto a full hyperboloid without ends. The image of the develloping map si equal to the 
%\end{prop}
%
%\begin{proof}
%
%\end{proof}

\subsection{Definition of maximal pencils.}
\label{subsec:DefnMaxPencil}

We say that a quadric $q$ in $\left(\R^{2n},\omega\right)$ is $\omega$-\emph{regular} if it is positive on some Lagrangian $\ell_1$ and negative on some Lagrangian $\ell_2$. Note that if $\ell_1,\ell_2$ are only assumed to be $n$ dimensional subspaces, this condition is exactly the condition of having signature $(n,n)$. We call $\Gra_2^\omega(\mathcal{Q})$ the space of pencils hose non-zero elements are all $\omega$-regular, which is an open subset of  $\Gra_2^{{(n,n)}}(\mathcal{Q})$. These pencils have in particular the property that the corresponding subsets of $\mathbb{P}(S^2V)$ intersect transversely the symmetric space of $\Sp(2n,\R)$, see Lemma \ref{lem:Transverse intersection with symmetric space}.

\begin{rem}\label{rem:BoundaryMapLagran}
If a locally fitting map $u:\widetilde{S_g}\to \Gra_2^{{(n,n)}}(\mathcal{Q})$ admits a fitting flow which is equivariant with respect to a representation $\rho:\Gamma\to \Sp(2n,\R)$, the image of $u$ must lie in $\Gra_2^\omega(\mathcal{Q})$ as $\rho$ is $\lbrace n\rbrace$-Anosov by Theorem \ref{thm:WellFittedImplesAnosov}, and the limit map of $\lbrace n\rbrace$-Anosov representations in $\Sp(2n,\R)$ take values in the space of Lagrangians. This last fact is a consequence of the fact that an attractive fixed $n$-dimensional subspace of an element in $\Sp(2n,\R)$ is necessarily Lagrangian, see for instance \cite{GWDoD}.
\end{rem}

There are non-maximal representations admitting equivariant fitting immersions, for instance almost-Fuchsian representations in $\SL(2,\C)\subset \Sp(4,\R)$, see Appendix \ref{sec:H3}. In order to obtain the maximality property, we need to restrict ourselves to the correct union of connected component of $\Gra_2^\omega(\mathcal{Q})$. 

\medskip

Let $P\in \Gra_2^\omega(\mathcal{Q})$ be a pencil, and fix an orientation for $P$. Recall that $\mathcal{L}_n$ in the space of Lagrangians in $\R^{2n}$. We construct a "boundary map" for an $\omega$-regular pencil of quadrics, defined up to homotopy. Before defining this map, note the following:

\begin{lem}
\label{lem:There exist a boundary map of pencil}
Let $q\in \mathcal{Q}$ be an $\omega$-regular element. The set of Lagrangians $\ell$ such that $q$ is positive on $\ell$ is homeomorphic to an open ball.
\end{lem}

\begin{proof}
There exist some $\ell_+\in \mathcal{L}_n$ on which $q$ is positive. Moreover there exist some $\ell_-\in \mathcal{L}_n$ on which $q$ is negative. 

\medskip

Every Lagrangian $\ell$ on which $q$ is positive must be transverse to $\ell_-$, hence it can be written as the graph $\lbrace x+u(x)\mid x\in \ell_+\rbrace$ of some linear map $u:\ell_+\to \ell_-$, and one has for all $v\in \ell^+$: 
$$q(v,v)+q(u(v),u(v))+2q(v,u(v))=q(v+u(v),v+u(v))>0.$$

Since $q(v,v)>0$ and $q(u(v),u(v))<0$, for all $0<\lambda<1$:
%$$q(v+\lambda u(v),v+ \lambda u(v))=q(v,v)+\lambda^2 q(u(v),u(v))+2\lambda q(v,u(v))>0;$$
$$q(v+\lambda u(v),v+ \lambda u(v))= q(v,v)+\lambda^2 q(u(v),u(v))+2\lambda q(v,u(v))>0.$$

We can identify the elements of $\mathcal{L}_n$ transverse to $\ell_-$ as the vector subspace of the space of maps $u:\ell_+\to \ell_-$. We just proved that in this chart the set of elements on which $q$ is positive is open and star-shaped, hence it is a open ball.
\end{proof}

We now define the "boundary map" of the pencil.

\begin{prop}
\label{prop:There exist a boundary map of pencil}
Let $P\in \Gra_2^\omega(\mathcal{Q})$. There exist a continuous map $\xi_P:\Sph P\to \mathcal{L}_n$ such that for all $[q]\in \Sph P$, $q>0$ on $\xi_P([q])$. Moreover any two such maps are homotopic, so the free homotopy type $[\xi_P]$ is well defined.
\end{prop}

\begin{proof}
A map $\xi_P$ is exactly a section of the bundle $\lbrace ([q],\ell)\mid q_{\mid \ell}>0\rbrace\to \Sph P$, which is a fiber bundle whose fibers are open balls. Such sections always exist and are unique up to homotopy.

%One can define a map $\xi'_P$ that is continuous except at a finite number of points such that $q$ is positive on $\xi'_P([q])$ for all $q\in P\setminus \lbrace 0\rbrace$ since $P\in \Gra_2^\omega(\mathcal{Q})$. For each $[q]$ for which $\xi'_P$ is discontinuous, we fix some Lagrangian $\ell_-,\ell^+$ as before. We change $\xi'_P$ in a neighborhood $U$ of $[q]$ so that it becomes a continuous map $\xi_P$ and $q'$ is still positive on $\xi_P([q'])$ for $q'\in U$. Indeed for $U$ small enough one can extend $\xi'_P$ on the left and on the right to continuous maps $\xi^+_P,\xi_P^-$ satisfying this condition, and one can construct $\xi_P$ as a convex combination of these maps in this fixed affine chart so that it coincides with $\xi'_P$ on the extremities of $U$.
%
%\medskip
%
%Moreover if we have two such maps $\xi_1$ and $\xi_2$, We can define for $\lambda\in [0,1]$ the map that associates to $[q]$ the  Langrangian which is the graph of $\lambda u_1+(1-\lambda)u_2:\xi_1([q])\to \xi_1([-q])$, where $u_1,u_2$ are the maps whose graph is equal to $\xi_1([q])$ and $\xi_2([q])$ respectively. This defines a homotopy between $\xi_1$ and $\xi_2$.
\end{proof}

We say that a pencil is \emph{maximal} for some orientation if $[\xi_P]=n[\tau]$, where $[\tau]$ is the generator of $\pi_1(\mathcal{L}_n)$ introduced in Section \ref{sec:Subsection on anosov represetnations}. We denote by $\Gra_2^{\max}(\mathcal{Q})$ the space of pencils that are maximal for some orientation. This is a union of connected components of $\Gra^\omega_2(\mathcal{Q})$ as the homotopy type $[\xi_P]$ is locally invariant for $P\in \Gra_2^\omega(\mathcal{Q})$.

\begin{rem}
The previous discussion allow us to distinguish several connected components of the open subspace $\Gra^\omega_2(\mathcal{Q})$ by looking at the homotopy type of the boundary map $\xi_P$.
\end{rem}

As a recall we have the following inclusions:
$$\Gra_2^{\max}(\mathcal{Q})\subset \Gra_2^\omega(\mathcal{Q})\subset \Gra_2^{(n,n)}(\mathcal{Q})\subset \Gra_2^{{\mix}}(\mathcal{Q})\subset \Gra_2(\mathcal{Q}).$$

All these inclusions are open, and the inclusion $Gr_2^{\max}(\mathcal{Q})\subset \Gra_2^\omega(\mathcal{Q})$ is a union of connected components.

\subsection{Statement of the characterization.}

We obtain the following characterization of maximal representations in terms of the existence of locally fitting maps that admit a fitting flow.

\begin{thm}
\label{thm: fitting implies maximal}
A representation $\rho:\Gamma_g\to \Sp(2n,\R)$ admits a $\rho$-equivariant locally fitting map $u:\widetilde{S_g}\to \Gra_2^{\max}(\mathcal{Q})$ that admits a $\rho$-invariant fitting flow if and only if it is maximal for some orientation of $S_g$.
\end{thm}

In this case the orientation of $S_g$ for which $\rho$ is maximal is induced by the orientation of the maximal pencils $u(x)$ for $x\in \widetilde{S_g}$ and Lemma \ref{lem:HomotopyWellFitted}.

\medskip

In particular if a representation $\rho$ admits an equivariant fitting immersion $u:\widetilde{S_g}\to \Gra_2^{\max}(\mathcal{Q})$ then it is maximal because of Proposition \ref{prop:ExistenceFittingFlow}. This theorem leaves the following question open:
%However it is not clear if fitting immersions exist for all maximal representation when $n\geq 3$.

\begin{que}
Given a maximal representation $\rho:\Gamma_g\to \Sp(2n,\R)$, is there always an equivariant fitting immersion $u:\widetilde{S_g}\to \Gra_2^{\max}(\mathcal{Q})$ ?
\end{que}

%A positive answer to this question would improve Theorem \ref{thm: fitting implies maximal}, and a fitting immersion always admits a fitting flow, see Proposition \ref{prop:ExistenceFittingFlow}. 
We show in Section \ref{sec:SP4} that this is true for $\Sp(4,\R)$, and in this case there exist a fitting immersion whose image lies in a single special $\Sp(4,\R)$-orbit of $\Gra_2^{\max}(\mathcal{Q})$.

\medskip

The following Lemma shows one direction of Theorem \ref{thm: fitting implies maximal}, the other direction is proven by Lemma \ref{prop:Exist fitting flow}.

\begin{lem}
\label{lem:Sufficient condition to be maximal}
Let $\rho:\Gamma_g\to \Sp(2n,\R)$ be a representation that admits an equivariant continuous map $u:\widetilde{S_g}\to \Gra_2^{\max}(\mathcal{Q})$ that admits an equivariant fitting flow. Then $\rho$ is maximal for some orientation of $S_g$.
\end{lem}

Recall that the limit map of $\lbrace n\rbrace$-Anosov representations in $\Sp(2n,\R)$ take values in the space of Lagrangians, see Remark \ref{rem:BoundaryMapLagran}.

\begin{proof}

We first apply Theorem \ref{thm:WellFittedImplesAnosov} which shows that $\rho$ is $\lbrace n\rbrace$-Anosov. Let $x\in \widetilde{S_g}$, we prove that the homotopy type of the boundary curve $\xi^n_\rho$ is the same as the homotopy type $[\xi_{u(x)}]$ from Proposition \ref{prop:There exist a boundary map of pencil}. We consider the map $\zeta_\infty$ that associates to $[q]\in \Sph u^*\mathcal{E}_{x}$ the limit of $\Phi_t([q])$ in $\partial \Gamma$. This map is homotopic in $\widetilde{S_g}\cup \partial \Gamma\setminus \lbrace x\rbrace$ to the map $\zeta_{1}$ that associates the projection $p(\Phi_{1})$. Hence by Lemma \ref{lem:HomotopyWellFitted}, $\zeta_1$ defines a diffeomorphism of degree $1$ between the circle $\Sph u(x)$ with its maximal orientation and the boundary $\partial \Gamma$ for the induced orientation. 

\medskip

The map $\xi^n_\rho\circ \zeta_\infty$ has the homotopy type $[\xi_{u(x)}]$ associated to $u(x)$: it defines a boundary map as in Proposition \ref{prop:There exist a boundary map of pencil}. The fact that $\rho$ is maximal is then a consequence of the characterization of maximal representations from Theorem \ref{thm:Caract Maximal reprs via homotopy}. Indeed $u(x)\in \Gra_2^{\max}(\mathcal{Q})$ implies that $[\xi_{u(x)}]=n[\tau]$, and we already know that the degree of $\zeta_\infty$ is equal to the degree of $\zeta_1$ which is equal to $1$.

\end{proof}

\subsection{Construction of a fitting flow}
\label{subsec:quadrics special}
In this section we study special quadrics in $\R^{2n}$ associated to pairs of transverse Lagrangians. These objects will allow us to construct fitting continuous embeddings of pencils.

\begin{defn}
Let $\ell_1$, $\ell_2$ be two transverse Lagrangians in $\R^{2n}$. We define $q_{\ell_1,\ell_2}$ to be the symmetric bilinear form on $\R^{2n}$ such that if $\pi_1,\pi_2$ are the projections on $\ell_1,\ell_2$ associated to the direct sum $\ell_1\oplus \ell_2=V$:
$$q_{\ell_1,\ell_2}(v,v)=\omega(\pi_1(v),\pi_2(v)).$$
\end{defn}

Note that $q_{\ell_2,\ell_1}=-q_{\ell_1,\ell_2}$.

\begin{rem}
In particular $q_{\ell_1,\ell_2}$ is characterized by the fact that $\ell_1$ and $\ell_2$ are isotropic and for all $v\in \ell_1,w\in \ell_2$:
$$q_{\ell_1,\ell_2}(v,w)=\omega(v,w).$$
\end{rem}

%The \emph{contact hyperboloid} associated to $\ell_1$ and $\ell_2$ is the set $\hyperb(\ell_1,\ell_2)\subset \mathbb{P}(\R^{2n})$  of all lines that are isotropic for $q_{\ell_1,\ell_2}$.
%
%In other words $\hyperb(\ell_1,\ell_2)$ is the set of vectors $\mathrm{v}+\mathrm{w}\in \R^{2n}$ with $\mathrm{v}\in \ell_1$ and $\mathrm{w}\in\ell_2$ such that $\omega(\mathrm{v},\mathrm{w})=0$.

Maximal triples of Lagrangians can be characterized as follows:

\begin{lem}
\label{lem:Triples of Lagrangians quadric}
A triple of Lagrangians $(\ell_1,\ell_2,\ell_3)$ is maximal if and only if $q_{\ell_1,\ell_3}$ is positive on $\ell_2$.
\end{lem}

\begin{proof}

Let us write $\ell_1,\ell_2,\ell_3$ as in Section \ref{sec:Subsec on maximal representations} for some symplectic basis :
$$\ell_1=\langle x_1,x_2,\cdots x_n\rangle, $$
$$\ell_2=\langle x_1+\epsilon_1 y_1,x_2+ \epsilon_2 y_2,\cdots x_n +\epsilon_ny_n\rangle,$$
$$\ell_3=\langle y_1,y_2,\cdots y_n\rangle. $$
 
The form $q_{\ell_1,\ell_3}$ can be written in this basis as:

$$q_{\ell_1,\ell_3}=\frac{1}{2}\sum_{i=1}^n x_i^*\otimes y_i^* +y_i^*\otimes x_i^*.$$

This form is positive on $\ell_2$ if and only if all of the $\epsilon_i$ are positive, and hence if the triple is maximal.
\end{proof}

These quadrics have also the following remarkable properties for maximal quadruples of Lagrangians.

%Since $\hyperb(\ell_1,\ell_2)$ depends only on $\ell_1,\ell_2$ once the symplectic form $\omega$ has been fixed, the image $g\cdot \hyperb( \ell_1, \ell_2)$ of a contact hyperboloid by a symplectic projective transformation $g\in PSp(2n,\R)$ must be the contact hyperboloid $\hyperb(g\cdot \ell_1, g\cdot \ell_2)$. 

A quadruple of Lagrangians $(\ell_1,\ell_2,\ell_3,\ell_4)$ is called \emph{maximal} if each cyclic oriented subtriple is maximal.

\begin{lem}
\label{lem:ContactHyperboloidsPositive1}
Let $(\ell_1,\ell_2,\ell_3,\ell_4)$ be a maximal quadruple of Lagrangians. The bilinear form $q_{\ell_4,\ell_3}-q_{\ell_1,\ell_2}$ is positive. In particular the zero set of these quadrics define two disjoint quadric hypersurfaces in $\mathbb{P}(\R^{2n})$.
\end{lem}
Note that a triple of Lagrangians $(\ell_1,\ell_2,\ell_3)$ is maximal if and only if corresponding linear map $u\in \Hom(\ell_1,\ell_3)$ whose graph is equal to $\ell_2$ is such that $\omega(\cdot , u(\cdot))$ is positive on $\ell_1$.

\begin{proof}

 Let us prove the first part of the statement. Since $(\ell_4,\ell_1,\ell_3)$ is a maximal triple of Lagrangians, $\ell_1$ can be written as the graph of some linear map $u_1:\ell_4\to\ell_3$ such that $\omega\left(\cdot ,u_1(\cdot)\right)$ is a positive bilinear form on $\ell_4$. Similarly since $(\ell_4,\ell_2,\ell_3)$ is a maximal triple of Lagrangians, $\ell_2$ can be written as the graph of some linear map $u_2:\ell_4\to\ell_3$ such that $\omega\left(\cdot ,u_2(\cdot)\right)$ is a positive bilinear form on $\ell_4$. 
 
 Let ${v}\in \R^{2n}$, it can be decomposed uniquely as ${v}={v}_1+{v}_2$ with ${v}_1\in \ell_1$ and ${v}_2\in \ell_2$. Moreover there exist some unique $x,y\in \ell_4$ such that ${v}_1=x+u_1(x)$ and ${v}_2=y+u_2(y)$. The vector ${v}$ decomposes therefore as ${v}=x+y+u_1(x)+u_2(y)$. One computes that :
 \begin{flalign*}
 q_{\ell_4,\ell_3}({v},{v})-q_{\ell_1,\ell_2}({v},{v})&=\omega\left(x+y,u_1(x)+u_2(y)\right)-\omega\left(x+u_1(x),y+u_2(y)\right),\\
 &=\omega\left(x,u_1(x)\right)+\omega\left(y,u_2(y)\right)+2\omega\left(y,u_1(x)\right).
 \end{flalign*}

 Finally the fact that $(\ell_1,\ell_2,\ell_3)$ forms a maximal triple implies that the bilinear form $\omega\left(\cdot,\widetilde{u_2}(\cdot)\right)$ is positive on $\ell_1$, where $\widetilde{u_2}\in\Hom(\ell_1,\ell_3)$ corresponding to $u_2$. In particular for all non-zero $x\in \ell_4$: 
 $$\omega(x+u_1(x),\widetilde{u_2}(x+u_1(x))=\omega(x,\widetilde{u_2}(x+u_1(x)) > 0.$$
 However note that by definition $x+u_1(x)+\widetilde{u_2}(x+u_1(x))=x+u_2(x)$ and in particular $\widetilde{u_2}(x+u_1(x))=(u_2-u_1)(x)$. Hence the symmetric bilinear form $\omega\left(\cdot,(u_2-u_1)(\cdot)\right)$ is positive on $\ell_4$. 
 
In particular for $y\neq 0$ the previous expression is strictly greater than the following $\omega\left(x,u_1(x)\right)+\omega\left(y,u_1(y)\right)+2\omega\left(y,u_1(x)\right)$, which is non-negative since $\omega\left(\cdot ,u_1(\cdot)\right)$ is positive. In the case when $y=0$, this last inequality is strict for $x\neq 0$. Otherwise the previous inequality $\omega\left(y,(u_2-u_1)(y)\right)>0$ is strict. Therefore for ${v}\neq 0$,  $q_{\ell_4,\ell_3}({v},{v})-q_{\ell_1,\ell_2}({v},{v})>0$.
\end{proof}

We now state an infinitesimal version of Lemma \ref{lem:ContactHyperboloidsPositive1} that we will use in Section \ref{sec:SP4}.

\begin{lem}
\label{lem:ContactHyperboloidsPositive2}

Let $\ell^+,\ell^-:[0,1]\to \mathcal{L}_n$ be smooth and such that $\ell_0^+=\ell^+(0)$ and $\ell_0^-=\ell^-(0)$ are transverse and the linear maps $\dot{u}^+\in \Hom(\ell^+_0,\ell^-_0)$ and $\dot{u}^-\in \Hom(\ell^+_0,\ell^+_0)$ corresponding to $(\ell^+)'(0)$ and $(\ell^-)'(0)$ are such that $\omega(\cdot, \dot{u}^+(\cdot))$ and $\omega(\dot{u}^-(\cdot), \cdot)$ are positive respectively on $\ell^+_0$ and $\ell^-_0$. The derivative at $t=0$ of $q_{\ell^+(t), \ell^-(t)}$ is positive.
\end{lem}

\begin{proof}

The proof is similar to the previous one. Let $u^+(t)\in \Hom(\ell^+_0,\ell^-_0)$ and $u^-(t)\in \Hom(\ell^-_0,\ell^+_0)$ be the linear maps whose graph is equal to $\ell^+(t)$ and $\ell^-(t)$ respectively. Let ${v}\in \R^{2n}$. The derivative of the evaluation of $q_t=q_{\ell^+(t), \ell^-(t)}$ to ${v}$ can be written in term of the derivative  $\dot{{v}}_t^+$, $\dot{{v}}_t^-$ of the vectors ${v}_t^+\in \ell_+(t)$, ${v}_t^-\in \ell_-(t)$ such that ${v}={v}_t^++{v}_t^-$:
$$\dot{q}_0({v},{v})= \omega({v}^+_0,\dot{{v}}^-_0)+\omega(\dot{{v}}^+_0,{v}^-_0).$$

One has $\dot{{v}}^+_0=-\dot{{v}}^-_0$ since ${v}=v^+_t+v^-_t$ does not vary with $t$. The fact that $v^+_t\in \ell_+(t)$ for all $t$ implies that $\dot{{v}}^+_0$ can be written ${w}^++\dot{u}^+({{v}}_0^+)$ with ${w}^+\in \ell^+_0$. Similarly $\dot{{v}}^-_0$ can be written ${w}^-+\dot{u}^-({{v}}_0^+)$ with ${w}^-\in \ell^-_0$.
%The vector $v_0$ admits a unique decomposition in the direct sum $\ell^+_0\oplus \ell^-_0$, therefore $\dot{{v}}^+_0=\dot{u}^+({{v}}_0^+)-\dot{u}^-({{v}}_0^-)$ and $\dot{{v}}^-_0=\dot{u}^-({{v}}_0^-)-\dot{u}^+({{v}}_0^+)$.
Note that $\omega(v_0^+,w^+)=0$ as they belong to the same Lagrangian $\ell^+_0$. Similarly $\omega(v_0^-,w^-)=0$ and hence:
$$\dot{q}_0= 2\omega({v}^+_0,\dot{u}^+({v}^+_0))+ 2\omega(\dot{u}^-({v}^-_0),{v}^-_0).$$
This is positive by our assumption.

\end{proof}

We now construct an equivariant continuous map that admits a well-fitting flow.

\begin{prop}
\label{prop:Exist fitting flow}
Let $\rho:\Gamma_g\to \Sp(2n,\R)$ be a maximal representation. There exist a continuous map $u:\widetilde{S_g}\to \Gra_2^{\max}(\mathcal{Q})$ that admits a fitting flow.
\end{prop}

We prove this lemma using some averaging argument, where the basic building blocks are the quadrics associated to pairs of Lagrangians in the limit curve. This construction is not unique, as we start by fixing a hyperbolic metric.

\begin{proof}

Fix a hyperbolic metric on $S_g$. Since $S_g$ is oriented it admits an associated complex structure $J$. For $v\in T^1_x\widetilde{S_g}$ write $\ell_v=\xi^n_\rho(\zeta_v)$ where $\zeta_v\in \partial \Gamma$ is the limit point of the geodesic with initial derivative $v$. We furthermore define:
$$q^\circ_v=q_{\ell_{Jv},\ell_{-Jv}}.$$

These quadrics for a fixed $x\in \widetilde{S_g}$ do not in general define a pencil of quadrics. We therefore define the following quadric associated to $v\in T_x\widetilde{S_g}$ :

$$q_v=\int_{w\in T^1_x\widetilde{S}}\langle v, w\rangle q^\circ_w\mathrm{d}\lambda.$$

Here we take the integral for the measure $\lambda$ on $T^1_x\widetilde{S}$ induced by the hyperbolic metric. For each $x\in \widetilde{S}$, we consider the pencil $u(x)= \lbrace q_v |v\in T\widetilde{S_g}\rbrace\in \Gra_2(\mathcal{Q})$ which is well defined since $q_v$ depends linearly on $v\in T_xS$.

\medskip

First we check that these pencils are in $\Gra_2^{(n,n)}(\mathcal{Q})$, by proving that they are actually $\omega$-regular. 

Let $v\in T^1_x\widetilde{S_g}$  be non-zero for some $x\in \widetilde{S_g}$. For all $w\in T_x\widetilde{S_g}$ if $\langle w,v\rangle> 0$, the triple $(\zeta_{-Jw},\zeta_v,\zeta_{Jw})$ is positively oriented and hence $(\ell_{-Jw},\ell_v,\ell_{Jw})$ is maximal and hence $q^\circ_w$ is negative on $\ell_v$. If $\langle w,v\rangle<0$, the triple $(\ell_{Jw},\ell_v,\ell_{-Jw})$ is maximal and hence $q^\circ_w$ is positive on $\ell_v$. Hence $q_v$ is negative on $\ell_v$, and by a similar argument $q_v$ is positive on $\ell_{-v}$, which are Lagrangians. In particular $q_v$ is $\omega$-regular for all $v\in T^1_x\widetilde{S_g}$, and so $u(x)$ is $\omega$-regular.

\begin{figure}[h!]
\begin{center}
\begin{tikzpicture}[scale=2.5]
\begin{scope}
    \clip (0,0) circle (1);
    \hgline{12}{105};
    \draw[blue](-45:1) -- (135:1);
\end{scope}

\draw (0,0) circle (1);

\node at (0,.25/2+.505) [left] {$v'$};
\node at (.18/2,.18/2+.505) [below right] {$\phi(w)$};
\draw[->,thick] (0,0+.505) -- (0,.5/2+.505);
\draw[->,thick] (0,0+.505) -- (.35/2,.35/2+.505);
\fill (0,.505) circle(.02);
\node at (0,0.505) [below] {$y$};
\fill (12:1) circle(.02);
\node at (12:1) [right] {$\zeta_{-J\phi(w)}$};
\fill (105:1) circle(.02);
\node at (105:1) [above left] {$\zeta_{J\phi(w)}$};

\node at (0,0) [below left] {$x$};
\fill (0,0) circle(.02);
\node at (0,.25/2) [left] {$v$};
\node at (.18/2,.18/2) [below right] {$w$};
\draw[->,thick] (0,0) -- (0,.5/2);
\draw[->,thick] (0,0) -- (.35/2,.35/2);
\fill (90:1) circle(.02);
\node at (90:1) [above] {$\zeta_v$};
\fill (-45:1) circle(.02);
\node at (-45:1) [below right] {$\zeta_{-Jw}$};
\fill (135:1) circle(.02);
\node at (135:1) [above left] {$\zeta_{Jw}$};

\end{tikzpicture}
\end{center}
\caption{Proof of Lemma \ref{prop:Exist fitting flow}.}
\label{fig:HypebolicConstruction}
\end{figure}
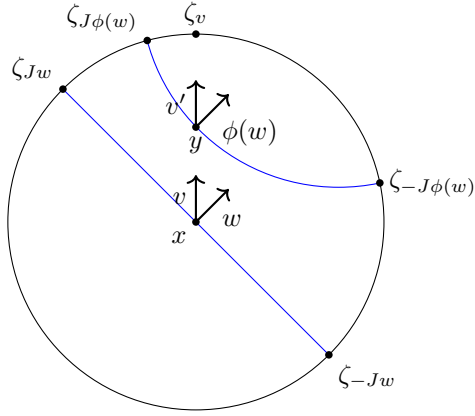

\medskip

We consider the geodesic flow on $u^*\Sph \mathcal{E}\simeq T^1\widetilde{S_g}$, and we prove that this flow is fitting. Let $t>0$, $x\in \widetilde{S_g}$ and $v\in T^1_x\widetilde{S_g}$. Let $(y,v')$ be the image of $(x,v)$ by the geodesic flow at time $t$, and let $\phi:T^1_x\widetilde{S_g}\to T^1_y\widetilde{S_g}$ be the identification given by the parallel transport along the geodesic between $x$ and $y$.

\medskip

Let $w\in T^1_x\widetilde{S}$ be such that $\langle w,v\rangle>0$. The following quadruple is positive $(\zeta_{Jw},\zeta_{-Jw}, \zeta_{-J\phi(w)},\zeta_{J\phi(w)})$ due to the negative curvature of the metric we put on $\widetilde{S_g}$, see Figure \ref{fig:HypebolicConstruction}. Hence the corresponding quadruple of Lagrangians is maximal. 
Therefore $q_{\phi(w)}^\circ-q_{w}^\circ$ is positive by Lemma \ref{lem:ContactHyperboloidsPositive1}. 
When $\langle w,v\rangle<0$, the following quadruple is positive $(\zeta_{-Jw}, \zeta_{Jw},\zeta_{J\phi(w)}, \zeta_{-J\phi(w) })$ and therefore $q_{\phi(w)}^\circ-q_{w}^\circ$ is negative. Hence :

$$q_{v'}-q_v= \int_{w\in T^1_x\widetilde{S}}\langle v, w\rangle \left(q_{\phi(w)}^\circ-q_{w}^\circ\right)\mathrm{d}\lambda>0.$$

We therefore have proven that the geodesic flow for the fixed hyperbolic metric is a fitting flow on $ T^1\widetilde{S_g}\simeq \Sph u^\mathcal{E}$.

\medskip

Finally as in the proof of Lemma \ref{lem:Sufficient condition to be maximal} the homotopy type of $\xi^n_\rho$ is equal to the homotopy type of $\xi_{u(x)}$ for all $x\in\widetilde{S_g}$. Hence these pencils are in $\Gra_2^{\max}(\mathcal{Q})$.

\end{proof}

\begin{rem}
These pencils always lie in the same connected component of $\Gra_2^\omega(\mathcal{Q})$. Indeed on can construct such a pencil $u(x)\in \Gra_2^\omega(\mathcal{Q})$ given any maximal continuous map $\xi$ from $\partial\widetilde{S_g}$ into the space of Lagrangians. The space of such maximal continuous maps being path connected, any two such pencils can be joined by a path in $\Gra_2^\omega(\mathcal{Q})$. It is not clear if $\Gra_2^{\max}(\mathcal{Q})$ only contains  this connected component.
\end{rem}

\section{Geometry of the symmetric space.}
\label{sec:Subsection Symmetric space}

In this section we prove Proposition \ref{prop:GaussMapSymmetric space for totally geodesic immersions} and Lemma \ref{lem:Transverse intersection with symmetric space} which are two facts independent from the main results of the paper. We show how to construct fitting immersions of pencils using totally geodesic surfaces in the symmetric space $\mathbb{P}(S^2V^{>0})$ as in \cite{Dav23}. We then prove that the codimension $d$ submanifolds corresponding to pencils in $\Gra^\omega_d(\mathcal{Q})$ intersect transversely the symmetric space of $\Sp(2n,\R)$ embedded in  $\mathbb{P}(S^2V^{>0})$.

\medskip

Given an immersion  $h:M\to \mathbb{P}(S^2V^{>0})$ from a manifold $M$ of dimension $d$ we define its \emph{Gauss map} $\mathcal{G}h :M\to \Gra_d(\mathcal{Q})$, that associates to $x\in M$ the pencil $P$ associated with the codimension $d$ projective subspace of $\Gr_d(\mathcal{Q})$ orthogonal to $\mathrm{d}h(TM)$ at $h(x)$, for the $\SL(V)$-invariant Riemannian metric on $\mathbb{P}(S^2V)$.

\medskip

The invariant Riemannian metric of the symmetric space associated to $\SL(V)$ can be described by a natural identification between $\mathbb{P}(S^2V^{>0})$ and its dual cone $\mathbb{P}(S^2(V^*)^{>0})$. We therefore reformulate the definition of $\mathcal{G}h$ as follows.

We first identify $\mathbb{P}(S^2V^{>0})$ with $\mathbb{P}(S^2(V^*)^{>0})$ via the map $[X]\mapsto [X^{-1}]$. Note as once again we view elements of $S^2V$ and $S^2(V^*)=\mathcal{Q}$ respectively as maps $V^*\to V$ and $V\to V^*$. An immersion $h:M\to \mathbb{P}(S^2V^{>0})$ hence defines a dual immersion $h^*:M\to \mathbb{P}(S^2(V^*)^{>0})$. Fixing a volume form on $V$ and $V^*$ allows us to lift this map to a map $\overline{h}^*$ into the space of elements in $S^2(V^*)^{>0}$ whose corresponding map $V\to V^*$ has determinant $1$. We define $\mathcal{G}h(x)=\mathrm{d}\overline{h}^*(T_xM)\subset S^2V^*=\mathcal{Q}$.

%
%Recall that $\mathbb{P}(S^2V^{>0})$ is a projective model for the symmetric space of $\SL(V)$ and it hence admits a $\SL(V)$-invariant Riemanian metric on $\mathbb{P}(S^2V^{>0})$ that is unique up to scaling. 
%
% One can identify  $\mathbb{P}(S^2V^{>0})$ with the space of tensors $V^*\to V$ whose determinant for some choice of a volume form is equal to $1$.
%Hence we can identify $\mathrm{d}h(T_xM)$ with a subspaceof dimension $d$ of $S^2V$.
%The inverse of $h(x)$ determines a bilinear pairing $V\to V^*$ that induces an identification of $S^2V$ with its dual $\mathcal{Q}$. 
%We define $\mathcal{G}h(x)\subset \mathcal{Q}$ to be the element corresponding to $\mathrm{d}h(T_xM)\subset S^2V$ via this identification.

\begin{prop} 
\label{prop:GaussMapSymmetric space for totally geodesic immersions}
Let $h:M\to \mathbb{P}(S^2V^{>0})$ be a totally geodesic immersion. Suppose that the image of the Gauss map $\mathcal{G}h :M\to \Gra_d(\mathcal{Q})$ contains only regular pencils, \ie pencils containing only non-degenerate quadrics. The immersion $\mathcal{G}h $ is then a fitting immersion. If $h$ is complete it is a globally fitting map. 
\end{prop}

This proposition can be applied for instance to the totally geodesic surface associated to any representation of $\SL(2,\R)$ into $\Sp(2n,\R)$. This can also be applied to the totally geodesic immersion of $\mathbb{H}^3$ induced by the inclusion $\SL(2,\C)\subset \Sp(4,\R)$.

\begin{rem}
The fibration of a domain of $\mathbb{P}(V)$ induced by the fitting map in this proposition is a particular case of Theorem 6.3 from \cite{Dav23}.
\end{rem}

Note that if $d\geq 2$, the signature $(a,b)$ of the quadrics of a regular pencil must satisfy $a=b$ since $\Sph^{d-1}$ is connected. Hence the condition that the pencils are regular can be replaced in this case by the condition that the pencils belong to $\Gra_d^{(n,n)}(\mathcal{Q})$ where $\dim(V)=2n$.

\medskip

As a corollary on can construct fitting immersions for some representations that factor through $\SL(2,\R)$. Indeed if one has a representation $\iota:\SL(2,\R)\to \SL(V)$ there exist a $\iota$-equivariant totally geodesic map $h:\mathbb{H}^2\to \SL(V)$, see Proposition \ref{prop:GaussMapSymmetric space for totally geodesic immersions}.

\begin{proof}
Let $\gamma:\R\to M$ be a geodesic for the metric induced by $h$. Recall that $\overline{h}(\gamma(t))$ for $t\in \R$ is an element of $S^2V$, and hence is defines a positive symmetric endomorphism of $V^*$ and as such it is diagonalizable is an orthonormal basis with positive coefficients. We write the representative of $h(\gamma(o))$ with determinant $1$ in a basis $(e_i)_{i\in I}$ such that for some $\lambda^i\in \R$, for all $t\in \R$:
$$\overline{h}(\gamma(t))=\sum_{i\in I}e^{t\lambda_i}e_i\otimes e_i.$$

The dual immersion can be written as:
$$\overline{h}^*(\gamma(t))=\sum_{i\in I}e^{-t\lambda_i}e^*_i\otimes e^*_i.$$

The element $q_t\in\mathcal{Q}$ corresponding to $(\overline{h}^*\circ \gamma)'(t)$ is the symmetric bilinear form:
$$q_t=\sum_{i\in I}-\lambda_ie^{-t\lambda_i}e_i^*\otimes e_i^*.$$

The derivative of $(q_t)$ at $t=0$ equals:
$$\sum_{i\in I}\lambda_i^2 e^{-t\lambda_i}e_i^*\otimes e_i^*.$$

This is a positive bilinear form if and only if all the $\lambda_i$ are non-zero, which is the case if and only if the bilinear forms $q_t$ are non-degenerate, \ie if the image of $\mathcal{G}h$ contains only regular pencils.
In this case, the positivity of the derivative of $(q_t)$ implies that the geodesic flow on $\mathbb{S}TM$ induces a fitting flow on $\mathcal{G}h^*\mathbb{S}\mathcal{E}$, so $\mathcal{G}h$ is a fitting immersion.

\medskip

If moreover $h$ is complete, it is $\rho$-equivariant for the discrete and faithful action of some closed surface group $\Gamma$. Hence Corollary \ref{cor:Globally fitting} implies that $h$ is a globally fitting map.
\end{proof}

Let us fix a symplectic form $\omega$ on $V=\R^{2n}$. Let $\X_{\Sp}$ be the subset of $\mathbb{P}(S^2V^{>0})$ consisting of tensors $[q^{-1}]$ that are compatible with $\omega$, i.e. such that for some complex structure $J$ on $\R^{2n}$, the bilinear form $q+i\omega$ is a $J$-hermitian metric on $V$. 
Recall that $q:V\to V^*$ is a bilinear form, and $q^{-1}:V^*\to V$ is a tensor. 

\medskip

The space $\X_{\Sp}$ is a copy of the symmetric space associated to $\Sp(2n,\R)$, which is a totally geodesic subspace of the symmetric space associated to the Lie group $\SL(2n,\R)$ whose model is $\mathbb{P}(S^2V^{>0})$.
However it is not a projective subset : the closure of the projective convex hull of $\X_{\Sp}$ in $\mathbb{P}(S^2V^{>0})$ is equal to $\mathbb{P}(S^2V^{\geq 0})$ since it contains all the extreme points of $\mathbb{P}(S^2V^{\geq 0})$, \ie the rank one elements $S^2\mathbb{P}(\R^{2n})$.

 Indeed, we can write matrices of the form $\text{Diag}(\lambda, 1, \cdots, 1,\lambda^{-1})\in \Sp(2n,\R)$ for any $\lambda\in \R_{>0}$ and such that the first basis spans any fixed line $\ell\in \mathbb{RP}^{2n-1}$, one can observe that the closure of the orbit of any point of $\X_{\Sp}$ by these elements contains the rank one point corresponding to $\ell$. Hence the closure of $\X_{\Sp}$ contains all rank one points.

\medskip

The intersection of $\X_{\Sp}$ with a general linear subspace is not necessarily transverse.
However it is the case for some special subspaces. 

\begin{defn}
\label{defn:omegaRegularQuadric}
We say that an element $q\in \mathcal{Q}$ is \emph{$\omega$-regular} if for some Lagrangians $\ell^+$, $\ell^-$ the bilinear form $q$ is positive on $\ell^+$ and negative on $\ell^-$.

\medskip

We denote by $\Gra^\omega_d(\mathcal{Q})$ the set of pencils whose non-zero elements are $\omega$-regular.
\end{defn}

In particular an $\omega$-regular pencil $q$ has signature $(n,n)$. 

\begin{lem}
\label{lem:Transverse intersection with symmetric space}
Let $P\in \Gra^\omega_2(\mathcal{Q})$ be an $\omega$-regular pencil, i.e. such that all its non-zero elements are $\omega$-regular. The space $\mathbb{P}(P^\circ)$ intersects transversely the manifold $\X_{\Sp}$ in a codimension $2$ submanifold.
\end{lem}

\begin{proof}
Let $q\in \mathcal{Q}$ be $\omega$-regular, and let $x\in \X_{\Sp}\cap \mathbb{P}(\langle q\rangle ^\circ)$ be an intersection point. Up to acting by $\Sp(2n,\R)$, one can assume that $x=[X^{-1}]\in \mathbb{P}(S^2V)$ where $X$ is the bilinear form whose associated matrix in  some symplectic basis is:
$$\begin{pmatrix}
I_n &  0\\
0 & I_n 
\end{pmatrix}.$$

The anihilator $(T\X_{\Sp})^\circ \subset S^2V$ of the tangent space to $\X_{\Sp}$ at this point, viewed as a linear subspace of $\mathcal{Q}$ can be identified with the space of following symmetric matrices where $A$ is symmetric and $B$ is antisymmetric:
$$\begin{pmatrix}
A &  B\\
-B & A 
\end{pmatrix}.$$

Suppose that the intersection is not transverse, i.e. that $q$ can be written in this form.
Let $\ell^+$ and $\ell^-$ be Lagrangians on which $q$ is respectively positive and negative.
Since the maximal compact $U(n)$ acts transitively on the space of Lagrangians one can assume that $\ell^+=\langle x_1,\cdots ,x_n\rangle$. Let $\ell_0^-=\langle y_1,\cdots ,y_n\rangle$ be its orthogonal for the fixed euclidean metric. The Lagrangian $\ell^-$ is transverse to $\ell^+$ so for some symmetric matrix $U$, one can $\ell^-$  as the image of :
$$\begin{pmatrix}
U \\
I_n 
\end{pmatrix}. $$

The fact that $q$ is positive on $\ell^+$ implies that $A$ is positive. The fact that $q$ is negative on $\ell^-$ implies that the following is negative:
$$UAU+A-BU+UB.$$
But $A$ and $UAU$ are both positive and the bracket $[B,U]$ has trace zero so it cannot be negative. Hence the intersection must be transverse.
\end{proof}

\begin{prop}
A tangent vector in $T\X_{\Sp} \subset T\mathbb{P}(S^2(V^*)^{>0})$ is $\omega$-regular if and only if the corresponding element of $\mathcal{Q}$ is non-degenerate.
\end{prop}

\begin{proof}
Given $\mathrm{v}\in T\X_{\Sp}$ one can write the corresponding element $q\in \mathcal{Q}$ for some symplectic basis $(x_1,\cdots x_n,y_1,\cdots y_n)$ as:
$$ \sum_{i=1}^n \lambda_i(x_i^*\otimes x_i^*-y_i^*\otimes y_i^*).$$

This bilinear form is regular if and only if all the $\lambda_i$ are non-zero. Else we can assume that $\lambda_i>0$ up to exchanging $x_i$ by $y_i$ and $y_i$ by $-x_i$ in the symplectic basis. Then the bilinear form is positive on the Lagrangian $\langle x_1,\cdots, x_n\rangle$ and negative on the Lagrangian $\langle y_1,\cdots, y_n\rangle$.
\end{proof}

\newpage

\appendix

\section{Immersions in the space of geodesics in \texorpdfstring{$\mathbb{H}^3$}{H3}}
\label{sec:H3}

\subsection{The space of geodesics.}
\label{subsec:Example with SL2C}

In this section we set $V=\C^2$, and restrict ourselves to pencils of Hermitian quadrics. Note that we can see this way $\SL(2,\C)$ as a subgroup of $\Sp(4,\R)\simeq \Sp(V,\omega)$ for $\omega$ the real part of a volume form on $\C^2$. The results of this Section serve as an illustration of the previously introduced notions of fitting pairs and fitting directions.

\medskip

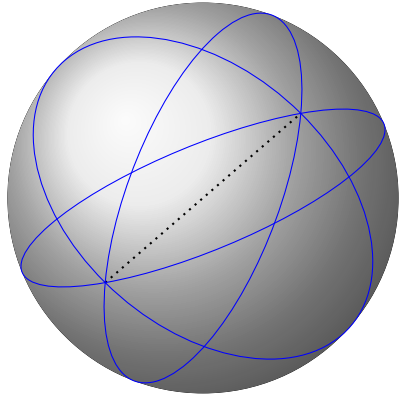
\begin{figure}
\begin{center}
\pgfmathsetmacro{\r}{2.6} %  
\tdplotsetmaincoords{60}{120}
\begin{tikzpicture}[
tdplot_main_coords,
%tdplot_rotated_coords,
font=\footnotesize,
Helpcircle/.style={blue,
%densely dashed
},
]

\pgfmathsetmacro{\h}{0.9*\r} %  
\coordinate (M) at (0,0,0);  %[label=$M$]
\coordinate (S) at (0,0,\h); 
\coordinate (A) at (0,\h,0); 

\rotateRPY{0}{35}{90}
\begin{scope}[RPY]
\draw[Helpcircle] (M) circle[radius=\r];
\draw[thick,dotted] (0,-\r,0) -- (0,\r,0);
\rotateRPY{0}{60}{0}
\begin{scope}[RPY]
\draw[Helpcircle] (M) circle[radius=\r];
\end{scope}
\rotateRPY{0}{120}{0}
\begin{scope}[RPY]
\draw[Helpcircle] (M) circle[radius=\r];
\end{scope}
\end{scope}

%\draw[Helpcircle] (S) circle[radius=sqrt(\r*\r-0.81*\r*\r)];

% Sphere
\begin{scope}[tdplot_screen_coords, on background layer]
\fill[ball color= gray!20, opacity = 0.4] (M) circle (\r); 
\end{scope}

%%% Points
%\foreach \P in {M,S, A}{
%\shade[ball color=white] (\P) circle (1.75pt);
%}

\end{tikzpicture}
\end{center}
\caption{Three Hermitian quadrics in a pencil  and the corresponding geodesic in $\mathbb{H}^3$.}
\label{fig:Pencil}
\end{figure}

Let $\mathcal{H}\subset \mathcal{Q}$ be the subset of Hermitian bilinear forms on $V=\C^2$. A Hermitian form $q\in \mathcal{H}$ that is not semi-positive or semi-negative is of Hermitian signature $(1,1)$, and hence its zero set in $\mathbb{CP}^1$ is a circle. 

\medskip

Let $S^2V=S^2_hV\oplus S^2_aV$ be the eigenspace decomposition for the operator $J\otimes J$ where $J$ is the complex conjugation. 
Here $S^2_hV$ is the $4$-dimensional eigenspace associated to $1$ and $S^2_aV=\mathcal{H}^\circ$, the $6$-dimensional space associated to $-1$. 
The intersection of $\mathbb{P}(S^2_hV)$ with the space of positive tensors is the projective Klein model for $\mathbb{H}^3$. 

\medskip

The annihilator of a pencil $P\in \Gra_2^{{\mix}}(\mathcal{H})$, \ie a plane $P\subset \mathcal{H}$ such that no non-zero vectors are semi-positive, is equal to $P^\circ=\mathcal{H}^\circ \oplus H$ where  $H$ is a plane in $S^2_hV$, which in turns corresponds to a geodesic in $\mathbb{H}^3$ in the projective Klein model.

\medskip

A pencil of quadrics $P\in \Gra^{{\mix}}_2(\mathcal{H})$ vanishes completely on two points in $\mathbb{CP}^1$, as in Figure \ref{fig:Pencil} where the zero set of three elements of the pencil are depicted. The space of mixed pencils in $\mathcal{H}$ can be identified with the space $\mathcal{G}$ of unoriented geodesics in $\mathbb{H}^3$. 

\begin{rem}
The space $\mathcal{G}$ is a $\SL(2,\C)$-homogeneous spac. Note that however the larger spaces $P\in \Gra^{{\mix}}_2(\mathcal{Q})$ or $P\in \Gra^{{\omega}}_2(\mathcal{Q})$ from the rest of the paper are not $\Sp(2n,\R)$ or $\SL(2n,\R)$ homogeneous spaces.
\end{rem}

\medskip
Two pencils in $\mathcal{H}$ form a fitting pair if and only if we can find a circle in each of the pencils that are disjoint. This is possible if and only if the two corresponding geodesics in $\mathbb{H}^3$ are disjoint, with disjoint endpoints. A fitting pair of pencils is illustrated in Figure \ref{fig:TwoPencils}.
 
\medskip

The space $\mathcal{G}$ of unoriented geodesics in $\mathbb{H}^3$ admits a pseudo-Riemannian metric of signature $(2,2)$ : we now compare the notion of fitting vectors with the notion of spacelike vectors. The tangent space at a geodesic with endpoints $(x,y)\in \mathbb{CP}^1$ can be identified with $T_x\mathbb{CP}^1\times T_y\mathbb{CP}^1$. The choice of a point in the geodesic provides an identification $\phi:T_x\mathbb{CP}^1\to T_y\mathbb{CP}^1$ and a metric $q_0$ on $T_y\mathbb{CP}^1$. Taking a different point in the geodesic means replacing $\phi$ by $\lambda\phi$ and $q_0$ by $\lambda^{-1} q_0$ for some $\lambda\in R_{>0}$.

\medskip

Hence we can consider the pseudo Riemannian metric that is invariant by the action of the isometry group of $\mathbb{H}^3$:
$$q:(T_x\mathbb{CP}^1\times T_y\mathbb{CP}^1)^2\to \R$$
$$(v_1,w_1),(v_2,w_2)\mapsto q_0(\phi(v_1),w_2)+q_0(\phi(v_2),w_1).$$

For this metric a vector $(v,w)\in T_x\mathbb{CP}^1\times T_y\mathbb{CP}^1$ is spacelike if and only if $q_0(\phi(v),w)>0$.

\begin{prop}
\label{prop:fitting in H3}

A pair of geodesics is a fitting pair if and only if the corresponding geodesics are disjoint.

\medskip

A tangent vector $(v,w)$ to $\mathcal{G}\simeq\Gra_2(\mathcal{H})$ is fitting if and only if $\phi(v)$ and $w$ are not positively anti-colinear, i.e there are no $\lambda,\mu\in \R^{\geq 0}$ such that $\lambda \phi(v)=- \mu w$.
\end{prop}

\begin{rem}
\label{rem:SpacelikeFittingH3}
In particular spacelike vectors are fitting, but not all fitting vectors are spacelike.

\end{rem}

\begin{proof}

Let $P_1,P_2\in \Gra^{{\mix}}_2(\mathcal{H})$. One has $\mathbb{P}(S_a^2V)\cap \mathbb{P}(S^2V^{\geq 0})=\emptyset$. Hence the subsets $\mathbb{P}(P_1^\circ\cap S^2V^{\geq 0})$ and $\mathbb{P}(P_2^\circ\cap S^2V^{\geq 0})$ for $P_1,P_2\in \Gra_2^{{\mix}}(\mathcal{H})$  are disjoint if and only if the corresponding geodesics are disjoint.

\medskip

Let $(v,w)$ be a tangent vector to $\gamma\in \mathcal{G}$ and let $\gamma_t$ be a curve in $\mathcal{G}$ with this derivative at $t=0$.
If $\phi(v)$ and $w$ are not non-positively colinear, the distance between $\gamma=\gamma_0$ and $\gamma_t$ is greater that $\epsilon t$ for some $\epsilon>0$ and $t$ small enough. Indeed there exist $z\in T_x\mathbb{CP}^1$ such that $q_0(\phi(z),\phi(x)),q_0(\phi(z),w)\geq 0$, and the totally geodesic disk in $\mathbb{H}^3$ through $x,y$ normal to $z$ at $x$ contains $\gamma_0$ while being at distance $\epsilon t$ to $\gamma_t$.

Hence the distance between the subsets $\mathbb{P}(P_0^\circ\cap S^2V^{\geq 0})$ and $\mathbb{P}(P_t^\circ\cap S^2V^{\geq 0})$ is also greater than $\epsilon't$ for some $\epsilon'>0$ and all $t$ small enough. 
Therefore by Proposition \ref{prop:fitting vector description} this direction is fitting. 

Conversely if  If $\phi(v)$ and $w$ are non-positively colinear, then there is such a curve $\gamma_t$ such that the corresponding geodesics all have a common point. Therefore by Proposition \ref{prop:fitting vector description} this direction is not fitting. 

%If $(v,w)$ are not negatively collinear, there exist a circle in $\mathbb{CP}^1$ passing through $x,y$ such that $v,w$ both point to the same open half-sphere determined by this circle.
%Let $H\subset \mathbb{H}^2$ be the hyperbolic plane whose boundary is this chosen circle.
%This circle corresponds to a non-zero element $q\in P\subset \mathcal{H}$ where $P\in \Gra_2^{{\mix}}(\mathcal{H})$ is the element corresponding to the geodesic between $x$ and $y$.
%
%\medskip
%
%Let $\ell$ be a compelment of $P$ in $\mathcal{H}$ and let $\mathrm{v}\in T\Gra_2(\mathcal{H})simeq \Hom(P, \ell)$ be the tangent vector correspoding to $(v,w)$. The hermitian form $\mathrm{v}(q)$ is positive on \wip
%
%\medskip
%
%The last fact comes from the fact that two pairs of points in $\mathbb{CP}^1$ belong to the respective boundaries of two disjoint disks if and only if the geodesics they define are disjoint.
%
%%Given any smooth curve $(\lbrace x_t,y_t\rbrace)_{t>0}$ with derivative $(v,w)$ at $t=0$, the distance between $H$ and the geodesic joining $x_t$ and $y_t$ is greater than $\epsilon t$ for $t$ small enough and some $\epsilon>0$. Hence this is a fitting curve of pencils and 
%
%\wip

%This circle determines a class $[q]\in \mathbb{S}P$ when $P\in \mathcal{H}$ is the pencil corresponding to the geodesic with endpoints $(x,y)$.
%
%\medskip
%
%Let $\mathrm{v}\in \Hom(P,H/P)$ be the element corresponding to $(v,w)$, then $\mathrm{v}(q)$ is positive on $x,y\in \mathbb{CP}^1$, and hence $\mathrm{v}(q)$ is positive on $P^\circ\cap $...
%\wip

\end{proof}

\begin{rem}
\label{rem:Nearly Fuchsian}
If $S$ is a surface in $\mathbb{H}^3$ with principal curvature in $(-1,1)$, then the set of normal geodesics forms spacelike surface, for the pseudo-Riemannian structure on the space of geodesics described in Section \ref{subsec:Example with SL2C}. The corresponding map $\mathcal{G}u:S\to\mathcal{G}$ is called the \emph{Gauss map}. Nearly Fuchsian representations are representations of a closed surface group $\Gamma_g$ admitting an equivariant surface  with principal curvature in $(-1,1)$. They are a priori a larger class of representations than almost Fuchsian representations, for which the equivariant surface with principal curvature in $(-1,1)$ is required to be minimal.

The space $\mathcal{G}$ also admits a special $\SL_2(\mathbb{C})$-invariant symplectic structure. An immersion in $\mathcal{G}$ is locally the Gauss map of an immersion with principal curvature in $(-1,1)$ if and only if it is spacelike and Lagrangian for this symplectic structure \cite{Gauss}.
Therefore if the fitting immersion is not Lagrangian, it does not come as the Gauss map of a surface in $\mathbb{H}^3$. Hence there could be representations admitting fitting immersions that are not nearly Fuchsian.
\end{rem}

\begin{figure}
\begin{center}
%\begin{tikzpicture}
%\draw[blue] (0:1) -- (0+180:1);
%\draw[red, thick] (90:1) -- (90+180:1);
%\draw[blue] (45:1) -- (45+180:1);
%\draw[blue] (135:1) -- (135+180:1);
%\begin{scope}[shift={(3,0)}]
%\draw[blue] (0:1) -- (0+180:1);
%\draw[red, thick] (90:1) -- (90+180:1);
%\draw[blue] (45:1) -- (45+180:1);
%\draw[blue] (135:1) -- (135+180:1);
%\end{scope}
%\end{tikzpicture}
\pgfmathsetmacro{\r}{2.6} %  
\tdplotsetmaincoords{60}{110}
\begin{tikzpicture}[
tdplot_main_coords,
%tdplot_rotated_coords,
font=\footnotesize,
Helpcircle/.style={blue, very nearly transparent
%densely dashed
},
Helpcirclebis/.style={red,
%densely dashed
},
]

\pgfmathsetmacro{\h}{0.7*\r}
\pgfmathsetmacro{\rh}{sqrt(\r*\r-\h*\h)} %  
\coordinate (M) at (0,0,0);  %[label=$M$]
\coordinate (S) at (0,0,\h); 
\coordinate (A) at (0,\h,0); 

\rotateRPY{0}{30}{85}
\begin{scope}[RPY]

\draw[very thick,dotted, blue] (0,-\r,0) -- (0,\r,0);
\draw[dotted, blue] (\r,0,0) arc (0:180:\r);
\draw[blue] (\r,0,0) arc (0:-180:\r);
\draw[Helpcirclebis] (0,0,\h) circle[radius=sqrt(\r*\r-\h*\h)];
\rotateRPY{0}{0}{40}
\begin{scope}[RPY]
\draw[very thick,dotted, red] (0,-\rh,\h) -- (0,\rh,\h);
\end{scope}
\rotateRPY{0}{60}{0}
\begin{scope}[RPY]
%\draw[Helpcircle] (0,0,0) circle[radius=\r];
\end{scope}
\rotateRPY{0}{120}{0}
\begin{scope}[RPY]
%\draw[Helpcircle] (0,0,0) circle[radius=\r];
\end{scope}
\end{scope}

% Sphere
\begin{scope}[tdplot_screen_coords, on background layer]
\fill[ball color= gray!20, opacity = 0.4] (M) circle (\r); 
\end{scope}

%%% Points
%\foreach \P in {M,S, A}{
%\shade[ball color=white] (\P) circle (1.75pt);
%}

\end{tikzpicture}
\end{center}
\caption{Two disjoint geodesics in $\mathbb{H}^3$ and disjoint circles in $\mathbb{CP}^1$ between their endpoints.}
\label{fig:TwoPencils}
\end{figure}
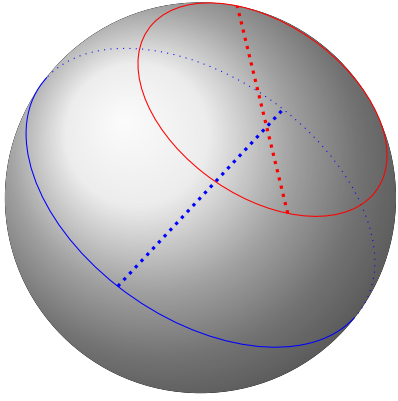

%Let $N$ be a compact $d$-dimensional manifold with universal cover $\widetilde{N}$. Let $\Gamma$ be its fundamental group, and let $\rho:\Gamma\to\GL(V)$ be a representation.
%
%\begin{thm}
%If $u:\widetilde{N}\to \Gra^\text{cal}_d(\mathcal{Q}_V)$ is a $\rho$-equivariant fitting immersion, then it is an embedding.
%\end{thm}
%
%In order to prove that $u$ is an embedding, we need to show the completeness 
%
%\begin{proof}
%\wip
%\end{proof}

\subsection{A quasi-Fuchsian representation with no fitting immersions.}
\label{subsec:QuasiFuchsian}

Having a representation that is Anosov is not sufficient to ensure that there exist an equivariant fitting immersion. We show that there are quasi-Fuchsian representations that admit no such immersions of Hermitian pencils of quadrics. We use here the notations from Section \ref{subsec:Example with SL2C}.

\begin{thm}
\label{thm:Quasi-Fuchsian with no fitting}
There exist a quasi-Fuchsian representation $\rho:\Gamma_g\to \SL(2,\C)$ for some genus $g$ large enough that admits no $\rho$-equivariant fitting immersion $u:\widetilde{S_g}\to \Gra_2^{{\mix}}(\mathcal{H})=\mathcal{G}$. 

Moreover it also admits no continuous map $u:\widetilde{S_g}\to \Gra_2^{{\mix}}(\mathcal{H})=\mathcal{G}$ that admits a $\rho$-equivariant fitting flow.
\end{thm}

A particular corollary is the following result on Nearly Fuchsian representations, introduced in Remark \ref{rem:Nearly Fuchsian}.

\begin{cor}
There exist quasi-Fuchsian representations that are not nearly Fuchsian.
\end{cor}

Almost Fuchsian representations were already known to not contain all quasi-Fuchsian representations see for isntance \cite{QuasiFuchMinSurfaces}, but a recent result of Nguyen, Schlenker and Seppi \cite{SeppiNearly} proves that nearly Fuchsian representations are not all almost Fuchsian, making this corollary new. 

\medskip

%Note that, as previously mentioned, almost-Fuchsian representation do always admit an equivariant fitting immersion obtained as the Gauss map of the corresponding equivariant minimal surface in $\mathbb{H}^2$.
%
%\medskip

Here $V=\C^2$ and if $\rho:\Gamma_g \to \SL(2,\C)\subset \Sp(4,\R)$, the Guichard-Wienhard domain of discontinuity $\Omega_\rho$, mentionned in Proposition \ref{prop:Completeness and fibration of DOD} corresponds to the pullback in $\mathbb{RP}^3=\mathbb{P}_\R(\C^2)$, of a domain of discontinuity $\Omega_\rho^0\subset \mathbb{CP}^1=\mathbb{P}_\C(\C^2)$. This domain is the complement in $\mathbb{CP}^1$ of the limit set of $\rho$.

Since $\Gamma_g$ is a surface group, this domain in $\mathbb{CP}^1$ is the union of two topological disks, and hence for each $x\in \widetilde{S_g}$ the geodesic corresponding to $u(x)$ has one endpoint in each of these discs. This fact will be used in the proof of Theorem \ref{thm:Quasi-Fuchsian with no fitting}.

\medskip

An other ingredient of the proof of Theorem \ref{thm:Quasi-Fuchsian with no fitting} is the following.

\begin{prop}[{\cite[{Corollary 3.5}]{QuasiFuchMinSurfaces}}]
\label{prop:QuasiFuchsianApprox}
Given any $\mathcal{C}^1$ embedded circle $\gamma$ in $\mathbb{CP}^1$, and any $\epsilon >0$, there exist a quasi-Fuchsian representation $\rho:\Gamma_g\to \SL(2,\C)$ for some genus $g$ large enough whose limit set has Hausdorff distance at most $\epsilon$ with $\gamma$.
\end{prop}

\begin{proof}[{Proof of Theorem \ref{thm:Quasi-Fuchsian with no fitting}}]

We consider the Jordan curve $\gamma$ from Figure \ref{fig:Illustration curve quasi Fuchsian}. Let $x,y,z$ be as in the figure. We consider a quasi-Fuchsian representation of $\Gamma_g$ for a genus large enough such that its limit set $\Lambda$ contains $x,z$ and is close enough to $\gamma$ using Proposition \ref{prop:QuasiFuchsianApprox}. More precisely let $s_x$ and $s_z$ be the open arcs of the circle of $\mathbb{CP}^1$ passing through $x,y,z$ respectively between $x,y$ and $y,z$ and let $I$ be the interior of the Jordan curve $\Lambda$. We require that the union $U_x$ of all the connected component of $I\setminus s_z$ whose closure contain $x$ is disjoint from the union $U_z$ of all the connected component of $I\setminus s_x$ whose closure contain $z$. These two disjoint sets are illustrated for the curve $\gamma$ as the two gray regions.

\medskip

Let $\rho:\Gamma_g\to\SL(2,\C)$ then be such a quasi-Fuchsian representation. Suppose that it admits an equivariant continuous map  $u:\widetilde{S_g}\to \Gra^{{\mix}}_2(\mathcal{H})$ with an equivariant fitting flow $\Phi$ on $\Sph u^*\mathcal{E}$. By Proposition \ref{prop:Flow lines of a fitting flow are quasi geodesic and exist between points} there exist a flow line $(\Phi_t(q))_{t\in \R}$ such that its projection $\gamma:\R\to \widetilde{S_g}$ is a quasi-geodesic between the points of $\partial \Gamma$ corresponding to $x$ and $z$ in the limit set.

For all $t\in \R$, let $x_t\in \mathbb{CP}^1$ be the endpoint of the geodesic corresponding to $u(\gamma(t))$ that belongs to the interior of the Jordan curve $\Lambda$. When $t$ goes to $+\infty$, $x_t$ converges to $z$ and it converges to $x$ for $t$ going to $-\infty$. Moreover $x_t$ always belong to the circle determined by $\Phi_t(q)$. 

\medskip

There exist a $t_0\in \R$ such that $y$ belongs to the circle $\Phi_{t_0}(q)$. Note that this great circle splits $\mathbb{CP}^1$ in two parts, one containing $x$ and $s_x$ and one containing $z$ and $s_z$. In particular $(x_t)_{t\geq t_0}$ must lie in $U_z$ and $(x_t)_{t\leq t_0}$ must lie in $U_x$, leading to a contradiction. Hence no such map $u$ can exist.

\begin{figure}[h]
\begin{center}
\begin{tikzpicture}

\fill[gray!30]  (7,0) arc[start angle=0, end angle=180, x radius=3.5, y radius =3.5] arc[start angle=180, end angle=360, x radius=.5, y radius =.5] arc[start angle=180, end angle=0, x radius=2.5, y radius =2.5] -- cycle;
\fill[gray!30]  (2,0) arc[start angle=180, end angle=360, x radius=3.5, y radius =3.5] arc[start angle=0, end angle=180, x radius=.5, y radius =.5] arc[start angle=0, end angle=-180, x radius=2.5, y radius =2.5] -- cycle;
\node[left] () at (0,0) {$x$};
\node[above] () at (2.5,0) {$s_x$};
\node[below] () at (6.5,0) {$s_z$};
\node[right] () at (9,0) {$z$};
\node[above] () at (4.5,0) {$y$};
\draw (7,0) arc[start angle=0, end angle=180, x radius=3.5, y radius =3.5] ;
\draw (1,0) arc[start angle=0, end angle=-180, x radius=.5, y radius =.5] ;
\draw (6,0) arc[start angle=0, end angle=180, x radius=2.5, y radius =2.5] ;
\draw (4,0) arc[start angle=0, end angle=180, x radius=.5, y radius =.5] ;
\draw (9,0) arc[start angle=0, end angle=180, x radius=.5, y radius =.5] ;
\draw (5,0) arc[start angle=0, end angle=180, x radius=1.5, y radius =1.5] ;
\draw (9,0) arc[start angle=0, end angle=-180, x radius=3.5, y radius =3.5] ;
\draw (8,0) arc[start angle=0, end angle=-180, x radius=2.5, y radius =2.5] ;
\draw (7,0) arc[start angle=0, end angle=-180, x radius=1.5, y radius =1.5] ;
\draw (6,0) arc[start angle=0, end angle=-180, x radius=.5, y radius =.5] ;
\draw[dotted, thick, red] (0,0) -- (4.5,0);
\draw[dotted, thick, red] (4.5,0) -- (9,0);
\fill[red] (4.5,0) circle(.1);
\fill[blue] (0,0) circle(.1);
\fill[blue] (9,0) circle(.1);

\end{tikzpicture}
\end{center}
\caption{A Jordan curve in $\mathbb{CP}^1$}
\label{fig:Illustration curve quasi Fuchsian}
\end{figure}
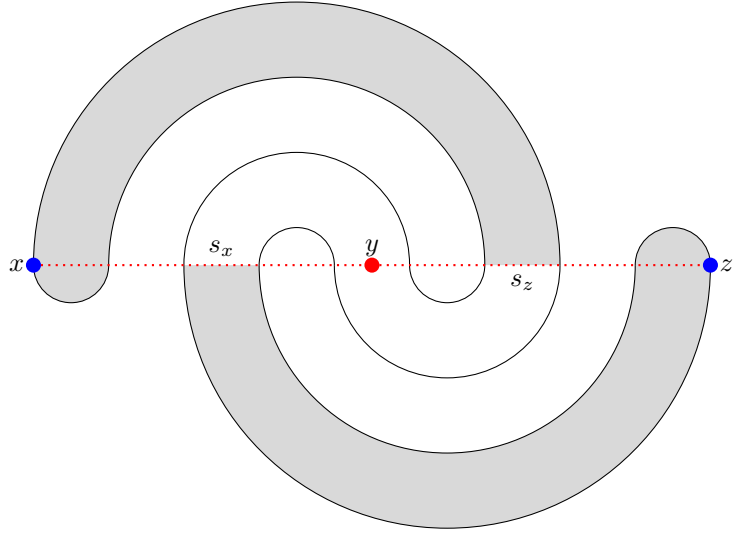

\end{proof}

\section{Fitting immersions and spacelike immersions for \texorpdfstring{$\Sp(4,\R)$}{Sp(4,R)}}
\label{sec:SP4}
In this section we explain how the data of a maximal immersion into the Pseudo-Riemannian space $\mathbb{H}^{2,2}$ with principal curvature in $(-1,1)$ induces a fitting immersion. Combining this with a result from Collier-Tholozan-Toulisse we show that our characterization of maximal representations can be improved in $\Sp(4,\R)$.

\medskip

Note first that one has the following exceptional isomorphism:
$$\PSp(4,\R)\simeq \SO_o(2,3).$$

This isomorphism comes form the fact that $\PSp(4,\R)$ preserves a subspace of dimension $5$ of $\Lambda^2\R^4$, as well as a symmetric bilinear form of signature $(2,3)$ on this subspace. Hence $\PSp(4,\R)$ acts naturally on the pseudo-Riemannian symmetric space with constant negative sectional curvature $\mathbb{H}^{2,2}$, which consists of vectors of norm $-1$ in $\R^{2,3}$. The space of Lagrangians $\mathcal{L}$ in $\R^4$ is naturally identified with the space of isotropic lines $\Ein^{1,2}$ in $\R^{2,3}$.

\medskip

To a pointed totally geodesic spacelike plane $(p,P)$ in $\mathbb{H}^{2,2}$ one can associate an element in a special $G$-orbit of $\Gra^{\max}_2(\mathcal{Q})$. For every geodesic in this plane passing though the base point, we consider the endpoints $\ell_1,\ell_2\in \mathcal{L}\simeq \Ein^{1,2}$, and the space generated by all such quadrics $q_{\ell_1,\ell_2}$ forms a plane which is a an element of $\Gra^\omega_2(\mathcal{Q})$. Indeed for some symplectic basis $(x_1,x_2,y_1,y_2)$ the Lagrangians corresponding to the boundary of the spacelike plane $P$ are for $\theta \in [0,2\pi]$:
$$\ell(\theta)=\langle \cos\left(\frac{\theta}{2}\right)x_1+\sin\left(\frac{\theta}{2}\right)y_1,\cos\left(\frac{\theta}{2}\right)x_2+\sin\left(\frac{\theta}{2}\right)y_2 \rangle.$$

The corresponding quadric $q_{\ell(0),\ell(\pi)}$ for $\theta=0$ in the basis $(x_1,y_1,x_2,y_2)$ is equal to:

$$\begin{pmatrix}
0 & 1 & 0 & 0\\
1 & 0 & 0 & 0\\
0 & 0 & 0 & 1\\
0 & 0 & 1 & 0
\end{pmatrix}.$$

Let $R_\theta$ be the following rotation matrix:
$$\begin{pmatrix}
\cos(\theta) & \sin(\theta) \\
-\sin(\theta) & \cos(\theta) 
\end{pmatrix}.$$

The corresponding quadric $q_{\ell(\theta),\ell(\theta+\pi)}$ in the basis $(x_1,y_1,x_2,y_2)$ is equal to:

$$\begin{pmatrix}
R_{\frac{\theta}{2}} & 0\\
0 & R_{\frac{\theta}{2}}
\end{pmatrix}q_{\ell(0),\ell(\pi)}\begin{pmatrix}
R_{-{\frac{\theta}{2}}} & 0\\
0 & R_{-{\frac{\theta}{2}}} 
\end{pmatrix}.$$

Hence $q_{\ell(\theta),\ell(\theta+\pi)}$ is equal to:
$$\begin{pmatrix}
\sin(\theta) & \cos(\theta) & 0 & 0\\
\cos(\theta) & -\sin(\theta) & 0 & 0\\
0 & 0 & \sin(\theta) & \cos(\theta)\\
0 & 0 & \cos(\theta) & -\sin(\theta)
\end{pmatrix}.$$

These quadrics span a plane in $\mathcal{Q}$ when $\theta$ varies, and these quadrics are $\omega$-regular so this plane is in $\Gra_2^\omega(\mathcal{Q})$.

\medskip

We define the Gauss map $\mathcal{G}u:S\to \Gra^{\max}_2(\mathcal{Q})$ of a spacelike immersion $u:S\to \mathbb{H}^{2,2}$ the map that associates to $x\in S$ the pencils corresponding to the pointed totally geodesic spacelike plane $(u(x),P)$ where $T_{u(x)}P=\mathrm{d}u(T_xS)$. Let $I\!\!I_u$ be the second fundamental form of the immersion $u$.

\begin{thm}
\label{thm:Gauss map is fitting in H22}
Let $u:S\to \mathbb{H}^{2,2}$ be a spacelike immersion such that for all $v\in TS$, $\lVert I\!\!I_u(v,v)\rVert <\lVert v\rVert^2$. The Gauss map $\mathcal{G}u:S\to \Gra^{\max}_2(\mathcal{Q})$ is a fitting immersion.
\end{thm}

Let $S$ be a spacelike surface in $\mathbb{H}^{2,n}$. Let $\gamma:[0,1]\to S$ be a geodesic for the induced metric on $S$ parametrized with unit speed an let $V:[0,1]\to TS$ be the unit orthogonal vector field to $\gamma'$ in $S$ along $\gamma$. We denote by $V^+, V^-:[0,1]\to \Ein^{1,2}$ the endpoints of the geodesic rays starting respectively at $V$ and $-V$. Up to changing the sign of $V$ one can assume that $(V^+,\gamma^+,V^-)$ is a maximal triple where $\gamma^+$ is the endpoint of the geodesic ray starting at $\gamma'$.

\medskip

We say that an immersed curve $c:\R\to \Ein^{1,2}$ is \emph{spacelike} if the tangent vectors to one and hence any lift of $c$ to $\R^{2,n+1}$ is spacelike. At every points $p\in \mathbb{H}^{2,2}$, the set tangent vectors of space type decomposes into two cones, a positive and a negative cone. One can make such a global continuous choice of positive cones on all of $\Ein^{1,2}$, and this choice is preserved by the action of $\SO_o(2,n+1)$, but not $\SO(2,n+1)$. We say that a spacelike curve is \emph{positive} if its derivative lies in the positive cones, and \emph{negative} if it lies in the other cones. 

The identification of $\mathcal{L}$ and $\Ein^{1,2}$ via the isomorphism $\PSp(4,\R)\simeq \SO_o(2,3)$ identifies a point $p\in \Ein^{1,2}$ with $\ell\in \mathcal{L}$ identifies the positive cone of vectors of spacelike type at $p\in \Ein^{1,2}$ with the cone of $T_\ell\mathcal{L}\simeq \Hom(\ell, \R^4/\ell)$ of elements $u\in \Hom(\ell,  \R^4/\ell)$ such that $\omega(\cdot, u(\cdot))$ is positive definite on $\ell$. 

\begin{rem}
This last identification is a consequence of the fact that the boundary of these cones are exactly the directions $v\in T_\ell\Ein^{1,2}\simeq T_\ell\mathcal{L}$ tangent to the locus of flags that are not transverse, i.e. in generic position, with $\ell$. Note that changing the choice of positive cones in $T\Ein^{1,2}$ amounts to changing the symplectic form $\omega$ to $-\omega$.
\end{rem}

\begin{lem}
\label{lem:Secondfundamental from bound H22}
Suppose that $\lVert I\!\!I_u(V,\gamma')\rVert\leq 1$, then the curves $V^+,V^-:[0,1]\to \Ein^{1,2}$ are respectively spacelike positive and spacelike negative, up to reversing the choice of positive cones.
\end{lem}

Note that in this lemma $I\!\!I_u(V,\gamma')$ is timelike, so its norm is the timelike norm that we see as a positive number. 
%A spacelike curve in $\Ein^{1,2}$ correspond exactly to 

\begin{proof}

We fix a orthogonal basis $(e_i)$ of $\R^{2,n+1}$ such that $e_1$ and $e_2$ have norm $1$ and $e_i$ for $i\geq 3$ have norm $-1$. Without any loss of generality we suppose that $\gamma(0)=e_3$, $\gamma'(0)=e_1$ and $V(0)=e_2$. Let $d$ be the flat connection on $\R^{2,n}$ and let $\nabla$ be the Levi-civita connection on $u(S)$ for the induced spacelike metric. We view $V$ as a vector field along the cure $\gamma$. For $t$ close to $0$ we have $V(t)=V(0)+tdV(0)$  and along the curve $\gamma$ we have $dV=\nabla V+ I\!\!I_{\mathbb{H}^{2,n}}(V,\gamma')+I\!\!I_u(V,\gamma')$. 
Here the second fundamental form of $\mathbb{H}^{2,n}$ inside $\R^{2,n+1}$ is equal to $I\!\!I_{\mathbb{H}^{2,n}}(v_1,v_2)=\langle v_1,v_2\rangle v_0$ for $v_1,v_2\in T_{v_0}\mathbb{H}^{2,n}$.
Note also that since $\gamma$ is a geodesic and $V$ a orthogonal unit vector field along $\gamma$, $\nabla V=0$.

\medskip

 Hence $V(t)=e_2+tI\!\!I_u(V,\gamma')+o(t)$. Since $V$ has norm $1$, we can write a representative of the isotropic line $V^+(t)$ as $v^+(t)=V(t)+\gamma(t)$. Therefore:
$$v^+(t)=e_2+t\left(I\!\!I_u(V,\gamma')+e_1\right)+o(t).$$
This curve is spacelike since $e_2$ is spacelike of norm $1$ and the (timelike) norm of $I\!\!I_u(V,\gamma')$ is strictly less than $1$. The same holds for $V^-$, $v^-(t)=-V(t)+\gamma(t)$, and we get:
$$v^-(t)=-e_2+t\left(-I\!\!I_u(V,\gamma')+e_1\right)+o(t).$$

\medskip

Now we check that these two tangent directions lies in different families of cones,:
 $$(v^+)'(0)=e_1+I\!\!I_u(V,\gamma')\in (e_2+e_3)^\perp\simeq T_{[e_2+e_3]}\Ein^{1,2},$$
 $$(v^-)'(0)=e_1-I\!\!I_u(V,\gamma')\in (-e_2+e_3)^\perp\simeq T_{[-e_2+e_3]}\Ein^{1,2}.$$
For that we move continuously the second vector by the rotation of angle $\theta$ in the plane $\langle e_1,e_2\rangle$. When $\theta=\pi$, the vector $(v^-)'(0)$ is sent to:
 $$-e_1-I\!\!I_u(V,\gamma')\in (e_2+e_3)^\perp\simeq T_{[e_2+e_3]}\Ein^{1,2}.$$ 
 This vector is in the opposite cone of spacelike vectors to $(v^+)'(0)$, hence $(v^-)'(0)$ and $(v^+)'(0)$ belong to opposite families of cones of spacelike vectors. The curves $V^+$, $V^-$ are hence spacelike with one being positive and one negative.

\end{proof}

\begin{proof}[{Proof of Theorem \ref{thm:Gauss map is fitting in H22}}]
The Gauss map $\mathcal{G}u:S\to \Gra^{\max}_2(\mathcal{Q})$ comes with an identification between $u^*\mathcal{E}$ and $TS$, $\mathcal{E}$ being the tautological bundle over $\Gra^{\max}_2(\mathcal{Q})$. We consider the geodesic flow $\Phi$ on $\Sph TS$ for the metric induced by $u$.

Lemma \ref{lem:Secondfundamental from bound H22} implies the the curves $V^+$ and $V^-$ are spacelike. Up to exchaging the positive cones and negative cones, we can assume that for all starting point $\gamma(0)$ and tangent vector $\gamma'(0)$, the associated curve $V^+$ is positive and $V^-$ is negative. Hence the corresponding curves of Lagrangians $\ell^+,\ell^-$ in $\R^4$ satisfy the hypothesis of Lemma  \ref{lem:ContactHyperboloidsPositive2}.
 
\medskip 
 
In particular the derivative of the associated quadrics along this flow is positive so $\Phi$ is a fitting flow. Since $u$ admits a fitting flow it is in particular a well-fitting immersion.
\end{proof}

The existence of a maximal spacelike immersion was proven by Collier-Tholozan-Toulisse and a bound of its second fundamental form is a consequence of a result from Cheng.

\begin{thm}[{\cite{Collier_2019},\cite{Cheng}}]
\label{thm:CCTandCheng}
Every maximal representation $\rho:\Gamma_g\to \SO_o(2,3)$ admits a unique $\rho$-equivariant maximal spacelike immersion $u:\widetilde{S_g}\to \mathbb{H}^{2,2}$. Moreover it is an embedding and for all $v\in TS$, $\lVert I\!\!I_u(v,v)\rVert< \lVert v\rVert^2$.
\end{thm}

The bound on the second fundamental from is a consequence of a maximal principle, see \cite[{Corollary 5.2}]{LT}. Note that this reference it is written that the square norm of $I\!\!I_u$ is at most $2$, but since $u$ is maximal it implies that the principal values are at most equal to $1$. 

\medskip

Putting together Theorem \ref{thm:Gauss map is fitting in H22}, Theorem \ref{thm:CCTandCheng} and Theorem \ref{thm: fitting implies maximal} we obtain:

\begin{cor}
\label{cor:MaxReprSP4}
Every maximal representation $\rho:\Gamma_g\to \Sp(4,\R)$ admits a fitting immersion $u:\widetilde{S_g}\to \Gra^{\max}_2(\mathcal{Q})$ that is $\rho$-equivariant. This characterizes representations which are maximal for some orientation of $S_g$.
\end{cor}

\begin{rem}
\label{rem:comparaisonCCT}

Note that a maximal spacelike surface in $\mathbb{H}^{2,2}$ always defines a fibration of the space of photons and of the symmetric space of $\Sp(4,\R)$, without assuming any bound on the second fundamental form. In this case however the bound is useful to deduce that the fibration extends to a fibration of the symmetric space of $\SL(4,\R)$.  
\end{rem}

\newpage

\bibliographystyle{hamsalpha}
\bibliography{biblio}

\providecommand{\bysame}{\leavevmode\hbox to3em{\hrulefill}\thinspace}
\providecommand{\href}[2]{#2}
\providecommand{\arXiv}[1]{\eprint{arXiv:#1}}
\providecommand{\eprint}{\begingroup \urlstyle{rm}\Url}
\begin{thebibliography}{GGKW17}

\bibitem[BG09]{BG}
Jairo Bochi and Nicolas Gourmelon, \emph{Some characterizations of domination},
  Math. Z. \textbf{263} (2009), no.~1, 221--231.

\bibitem[BILW05]{BurgerIozziWienhardLabourie}
Marc Burger, Alessandra Iozzi, Fran\c{c}ois Labourie, and Anna Wienhard,
  \emph{Maximal representations of surface groups: symplectic {A}nosov
  structures}, Pure Appl. Math. Q. \textbf{1} (2005), no.~3, Special Issue: In
  memory of Armand Borel. Part 2, 543--590.

\bibitem[BIW03]{BIW03}
Marc Burger, Alessandra Iozzi, and Anna Wienhard, \emph{{Surface group
  representations with maximal Toledo invariant}}, Comptes Rendus Mathematique
  \textbf{336} (2003), no.~5, 387--390.

\bibitem[BIW11]{BIW10}
Marc Burger, Alessandra Iozzi, and Anna Wienhard, \emph{Higher teichm\"uller
  spaces: from sl(2,r) to other lie groups}, 2011, \eprint{1004.2894}.

\bibitem[BPS19]{BPS}
Jairo Bochi, Rafael Potrie, and Andr\'{e}s Sambarino, \emph{Anosov
  representations and dominated splittings}, J. Eur. Math. Soc. (JEMS)
  \textbf{21} (2019), no.~11, 3343--3414.

\bibitem[Che93]{Cheng}
Qing Cheng, \emph{Space-like surfaces in an anti-de sitter space}, Colloquium
  Mathematicum \textbf{66} (1993), 201--208.

\bibitem[CTT19]{Collier_2019}
Brian Collier, Nicolas Tholozan, and J\'{e}r\'{e}my Toulisse, \emph{The
  geometry of maximal representations of surface groups into {${\rm
  SO}_0(2,n)$}}, Duke Math. J. \textbf{168} (2019), no.~15, 2873--2949.

\bibitem[Dav23]{Dav23}
Colin Davalo, \emph{Nearly geodesic immersions and domains of discontinuity},
  2023, \eprint{2303.11260}.

\bibitem[Eps86]{Epstein}
Charles~L. Epstein, \emph{The hyperbolic {G}auss map and quasiconformal
  reflections}, J. Reine Angew. Math. \textbf{372} (1986), 96--135.

\bibitem[ES22]{Gauss}
Christian~El Emam and Andrea Seppi, \emph{On the gauss map of equivariant
  immersions in hyperbolic space}, 2022, \eprint{2008.07390}.

\bibitem[GGKW17]{Gu_ritaud_2017}
François Guéritaud, Olivier Guichard, Fanny Kassel, and Anna Wienhard,
  \emph{Anosov representations and proper actions}, Geometry \&amp; Topology
  \textbf{21} (2017), no.~1, 485–584.

\bibitem[GW12]{GWDoD}
Olivier Guichard and Anna Wienhard, \emph{Anosov representations: domains of
  discontinuity and applications}, Invent. Math. \textbf{190} (2012), no.~2,
  357--438.

\bibitem[Hat02]{Hatcher}
Allen Hatcher, \emph{Algebraic topology}, Cambridge University Press,
  Cambridge, 2002.

\bibitem[HW15]{QuasiFuchMinSurfaces}
Zheng Huang and Biao Wang, \emph{Counting minimal surfaces in quasi-fuchsian
  three-manifolds}, Transactions of the American Mathematical Society
  \textbf{367} (2015), no.~9, 6063--6083.

\bibitem[LT22]{LT}
François Labourie and Jérémy Toulisse, \emph{Quasicircles and quasiperiodic
  surfaces in pseudo-hyperbolic spaces}, 2022, \eprint{2010.05704}.

\bibitem[LV13]{lion2013weil}
G.~Lion and M.~Vergne, \emph{The weil representation, maslov index and theta
  series}, Progress in Mathematics, Birkh{\"a}user Boston, 2013.

\bibitem[NSS25]{SeppiNearly}
Manh-Tien Nguyen, Jean-Marc Schlenker, and Andrea Seppi, \emph{Weakly
  almost-fuchsian manifolds are nearly-fuchsian}, 2025, \eprint{2501.12277}.

\bibitem[Wig98]{ThesisTopologyFlag}
Mark Wiggerman, \emph{The fundamental group of a real flag manifold}, Indag.
  Math. (N.S.) \textbf{9} (1998), no.~1, 141--153.

\end{thebibliography}
%\nocite{*}
\end{document}